\documentclass{amsart}[14pt,a4paper]
\usepackage[english]{babel}
\usepackage[T1]{fontenc}
\usepackage[cp1252]{inputenc}
\usepackage{geometry}
\usepackage{amsfonts}
\usepackage{hyperref}
\usepackage{amsmath}
\usepackage{galois}
\DeclareMathSizes{10}{10}{7}{5}
\usepackage{mathtools}
\numberwithin{equation}{section}

\usepackage{enumitem}
\usepackage{amssymb}
\usepackage{graphicx}
\usepackage{subfigure}
\usepackage{caption}
\usepackage{epsfig}

\numberwithin{equation}{section}

\usepackage{enumitem}
\usepackage{amssymb}
\usepackage{mathrsfs}

\newtheorem{theorem}{Theorem}[section]
\newtheorem*{theoremA}{Theorem A}
\newtheorem*{theoremB}{Theorem B}

\newtheorem*{remark}{Remark}
\newtheorem{lemma}{Lemma}[section]
\newtheorem{proposition}{Proposition}[section]
\newtheorem{corollary}{Corollary}[section]
\begin{document}
\title{On quadratic Siegel disks with a class of unbounded type rotation numbers}
\author{Hongyu Qu, Jianyong Qiao\textsuperscript{*} and Guangyuan Zhang}
\address{School of Sciences, Beijing University of Posts and Telecommunications, Beijing
100786, P. R. China. \textit{Email:\ hongyuqu2022@126.com}}
\address{School of Sciences, Beijing University of Posts and Telecommunications, Beijing
100786, P. R. China. \textit{Email:} \textit{qjy@bupt.edu.cn}}
\address{Department of Mathematical Sciences, Tsinghua University, Beijing
100084, P. R. China. \textit{Email:} \textit{gyzhang@mail.tsinghua.edu.cn}}
\renewcommand{\thefootnote}{\fnsymbol{footnote}}
\footnotetext[1]{Corresponding author, Email: qjy@bupt.edu.cn}
\footnotetext[2]{The research work was supported by
the National Natural Science Foundation of China under Grants No.12071047 and
the National Natural Science Foundation of China under Grants No.12171264.}
\maketitle

\begin{abstract}
In this paper we explore a class of quadratic polynomials having Siegel disks with unbounded type rotation numbers.
We prove that any boundary point of Siegel disks of these polynomials is a Lebesgue density point of their filled-in Julia sets,
which generalizes the corresponding result of McMullen for bounded type rotation numbers.
As an application, this result can help us construct more quadratic Julia sets with positive area.
Moreover, we also explore the canonical candidate model for quasiconformal surgery of quadratic polynomials with Siegel disks.
We prove that for any irrational rotation number, any boundary point of ``Siegel disk'' of the canonical candidate model
is a Lebesgue density point of its ``filled-in Julia set'',
in particular the critical point $1$ is a measurable deep point of the ``filled-in Julia set''.
\end{abstract}

\section{Introduction}
Let the quadratic polynomial
$$P_{\alpha}(z)=e^{2\pi i\alpha}z+z^2,$$
where $0<\alpha<1$ is an irrational number with continued fraction expansion
$$\alpha=[a_1,a_2,\cdots]=\frac{1}{a_1+\frac{1}{a_2+\ddots}}$$
and the rational approximations
$$\frac{p_n}{q_n}=[a_1,a_2,\cdots,a_n],\ n\geq1.$$
According to \cite{Br} or \cite{Yoc95}, when $\alpha$ is a Brjuno number,
$P_{\alpha}$ has a Siegel disk centering at $0$, denoted by $\Delta_{\alpha}$.
We denote by $\Delta_{\alpha}(r)$ the $r$-neighborhood
of $\Delta_{\alpha}$ for any $r>0$. Set
\[K_r(P_{\alpha}):=\{z\in\Delta_{\alpha}(r):P_{\alpha}^{\comp n}(z)\in\Delta_{\alpha}(r)\ {\rm for}\ {\rm all}\ n\geq0\}.\]
Next we want to explore the Lebesgue density of
$K_r(P_{\alpha})$ near the boundary of $\Delta_{\alpha}$.
When $\alpha$ is of bounded type,
in \cite{Mc} McMullen proved the following result:
\begin{theoremA}[McMullen]
\label{T1.3}If $\alpha$ is of bounded type, then for any $r>0$, every point in $\partial\Delta_{\alpha}$
is a Lebesgue density point of $K_r(P_{\alpha})$.
\end{theoremA}

\begin{remark}
{\rm In fact, in \cite{Mc} McMullen proved that for any $r>0$, every point in $\partial\Delta_{\alpha}$
is a measurable deep point
\footnote[3]{A point $x\in\Lambda$ (compact subset of $\mathbb{C}$) is called \emph{a measurable deep point}
if for some $\delta>0$,
${\rm area}(\mathbb{B}(x,r)\setminus\Lambda)=\mathcal{O}(r^{2+\delta})$
for all $r>0$, where $\mathbb{B}(x,r)$ means a ball centering at $x$ with radius $r$ and
${\rm area}(\mathbb{B}(x,r)\setminus\Lambda)$ means the area of $\mathbb{B}(x,r)\setminus\Lambda$.}
of $K_r(P_{\alpha})$, which gives an order estimate.}
\end{remark}

Theorem A reveals that $K_r(P_{\alpha})$ is very densely distributed near the boundary of $\Delta_{\alpha}$.
The main purpose in this paper is to generalize Theorem A
to a larger class of rotation numbers $\mathcal{E}_0$.

Let $\theta$ be an irrational number between $0$ and $1$. We say $\theta$ satisfies Petersen-Zakeri condition if
$\theta=[a_1,a_2,\cdots]$ with $\log a_n=\mathcal{O}(\sqrt{n})\ {\rm as}\ n\to+\infty$.
All such $\theta$ are denoted by $\mathcal{E}$, that is
$$\mathcal{E}=\{\theta=[a_1,a_2,\cdots]\in\mathbb{R}\setminus\mathbb{Q}:\ \log a_n=\mathcal{O}(\sqrt{n})\ {\rm as}\ n\to+\infty\}.$$
An irrational number $\alpha=[c_1,c_2,\cdots]$ between $0$ and $1$ is said to belong to
$\mathcal{E}_0$ if
there exist $\theta=[a_1,a_2,\cdots]\in\mathcal{E}$, a positive integer $M$ and
two sequences of positive integers $\{s_j\}_{j=1}^{\infty}$ and $\{t_j\}_{j=1}^{\infty}$ such that
\begin{itemize}
\item for all $j\geq1$, $s_j<t_j<s_{j+1}$ and $t_j-s_j>Cs_j$, where $C>0$ is a universal constant and will be fixed in Section \ref{s5};
\item for all $1\leq k\leq s_1$, $c_k\leq a_k$;
\item for all $j\geq1$, $$c_k\left\{\begin{matrix}\leq&M,&s_j<k\leq t_j,\\
\leq&a_{k-t_j},&t_j<k\leq s_{j+1}.\end{matrix}\right.$$
\end{itemize}
It is easy to see that $\mathcal{E}_0$ contains all irrational numbers of bounded type. Furthermore,
when an irrational number with Petersen-Zakeri condition is added many enough uniformly bounded entries, it becomes an element of $\mathcal{S}_0$. Precisely,
given $\theta=[a_1,a_2,\cdots]\in\mathcal{E}$, for any positive integer $M$ and
two sequences of positive integers $\{s_j\}_{j=1}^{\infty}$ and $\{t_j\}_{j=1}^{\infty}$ with
$s_j+t_j<s_{j+1}$ and $t_j>(C+1)s_j$ for all $j\geq1$,
if an irrational number $\alpha=[c_1,c_2,\cdots]$ satisfies that
\begin{itemize}
\item for all $1\leq k\leq s_1$, $c_k=a_k$;
\item for all $j\geq1$, $$c_k\left\{\begin{matrix}\leq&M,&s_j<k\leq s_j+t_j,\\
=&a_{k-\sum_{s=1}^jt_s},&s_j+t_j<k\leq s_{j+1},\end{matrix}\right.$$
\end{itemize}
then we have $\alpha\in\mathcal{E}_0$.
In this paper we will prove the following theorem:
\begin{theorem}
\label{T1}For all $\alpha\in\mathcal{E}_0$ and all $r>0$,
every point in $\partial\Delta_{\alpha}$ is a Lebesgue density point of $K_r(P_{\alpha})$.
\end{theorem}

The proof of this theorem is based on Petersen and Zakeri's famous results on a canonical candidate model for $P_{\alpha}$ (see \cite{PZ04}).
Next, we recall the canonical candidate model.
Consider the degree $3$ Blaschke product
$$f_{\alpha}:z\mapsto e^{2\pi it(\alpha)}z^2\left(\frac{z-3}{1-3z}\right),$$
which has superattracting fixed points at $0$ and $\infty$ and a double critical point at $z=1$.
Here $0<t(\alpha)<1$ is the unique parameter for which the critical circle map
$f_{\alpha}|_{\mathbb{S}^1}:\mathbb{S}^1\to\mathbb{S}^1$ has rotation number $\alpha$.
By Yoccoz's theorem\cite{Yoc84}, there exists a unique homeomorphism $h:\mathbb{S}^1\to\mathbb{S}^1$ with $h(1)=1$ such that
$h\comp f_{\alpha}|_{\mathbb{S}^1}=R_{\alpha}\comp h$, where $R_{\alpha}(z)=e^{2\pi i\alpha}z$ for all $z\in\mathbb{C}$.
For all $r>0$, we set $\mathbb{D}_r:=\{z\in\mathbb{C}:|z|<r\}$ and in particular, set $\mathbb{D}:=\mathbb{D}_1$.
Let $H:\overline{\mathbb{D}}\to\overline{\mathbb{D}}$ be any homeomorphic extension of $h$ and the canonical candidate model is defined by
$$F_{\alpha}(z):=\left\{\begin{matrix}f_{\alpha}(z)&{\rm if\ |z|\geq1}\ \\
(H^{-1}\comp R_{\alpha}\comp H)(z)&{\rm if\ |z|<1}.\end{matrix}\right.$$
By definition, $F_{\alpha}$ is a degree $2$ topological branched covering of the sphere which
is holomorphic outside of $\overline{D}$ and is topologically conjugate to a rigid rotation on $\overline{\mathbb{D}}$.
About this canonical candidate model, in \cite{Pe} Petersen showed that the ``Julia set'' $J(F_{\alpha})$ of $F_{\alpha}$,
that is the boundary of the set consisting of those points whose orbits are bounded,
is locally connected for every irrational number $\alpha$,
and has measure zero for every $\alpha$ of bounded type. In \cite{Ya},
the measure zero statement was extended by Lyubich and Yampolsky to all irrational number $\alpha$.
Set
\[K_r(F_{\alpha}):=\{z\in\mathbb{D}_{1+r}:F_{\alpha}^{\comp n}(z)\in\mathbb{D}_{1+r}\ {\rm for}\ {\rm all}\ n\geq0\}.\]
About the model $F_{\alpha}$, we obtain the following result:
\begin{theorem}
\label{T1.2}For all irrational number $0<\alpha<1$ and all $r>0$,
every point in $\mathbb{S}^1$ is a Lebesgue density point of $K_r(F_{\alpha})$.
In particular, $1$ is a measurable deep point of $K_r(F_{\alpha})$.
\end{theorem}

\begin{remark}
{\rm In \cite{Mc}, when $\theta$ is of bounded type,
McMullen proved that the Julia set $J(P_{\theta})$ is a shallow set and any point in $\partial\Delta_{\theta}$ is a deep point of $K_r(P_{\theta})$,
which implies that any point in $\partial\Delta_{\theta}$ is a measurable deep point
of $K_r(P_{\theta})$,
and hence a Lebesgue density point of $K_r(P_{\theta})$.
But the proof of this theorem can't follow this way,
for it can be proved that for any unbounded type irrational number $\alpha$,  $J(F_{\alpha})$ isn't a shallow set.
See Appendix B for the discussion of non-shallowness of $J(F_{\alpha})$.}
\end{remark}

Existence of nowhere dense rational Julia set with positive area was a famous open problem until
Buff and Ch\'eritat constructed quadratic Julia sets with positive area. In their famous work\cite{BC12},
the following result is a key step to construct quadratic Julia sets with positive area:
\begin{theoremB}[Buff and Ch\'eritat]
If $N$ is sufficiently large, then the following holds.
Assume $\alpha=[a_1,a_2,\cdots]$ is of bounded type and choose a sequence $\{A_n\}_{n=1}^{+\infty}$ of positive integers such that
\[\sqrt[q_n]{A_n}\xrightarrow[n\to+\infty]{}+\infty\ {\rm and}\ \sqrt[q_n]{\log A_n}\xrightarrow[n\to+\infty]{}1.\]
Set
\[\alpha_n:=[a_1,a_2,\cdots,a_n,A_n,N,N,\cdots].\]
Then, for all $\epsilon>0$, if $n$ is sufficiently large,
\begin{itemize}
\item $P_{\alpha_n}$ has a cycle in $D(0,\epsilon)\setminus\{0\}$,
\item ${\rm area}(K(P_{\alpha_n}))\geq(1-\epsilon)\cdot{\rm area}(K(P_{\alpha})),$
\end{itemize}
where $K(P_{\alpha})$ $($resp. $K(P_{\alpha_n})$$)$ means the filled-in Julia set of $P_{\alpha}$ $($resp. $P_{\alpha_n}$$)$
with its Lebesgue measure ${\rm area}(K(P_{\alpha}))$ $($resp. ${\rm area}(K(P_{\alpha_n}))$$)$.
\end{theoremB}
In their paper, Buff and Ch\'eritat wrote that they thought that Theorem B holds for more general sequences $\{\alpha_n\}_{n=1}^{+\infty}$.
In \cite{QQ}, the first two authors proved that the condition $\sqrt[q_n]{\log A_n}\xrightarrow[n\to+\infty]{}1$ in Theorem B is unnecessary.
It follows from [Theorem E, \cite{Ch}] that Theorem B still holds
if all entries after $A_n$ in $\alpha_n$ are replaced by all entries of a Brjuno number of high type.
Now together with the above results, by Buff and Ch\'eritat's method
Theorem \ref{T1} immediately give us a further generalization of Theorem B as follows:
\begin{theorem}
\label{T2}If $N$ is sufficiently large, then the following holds.
Assume $\alpha=[a_1,a_2,\cdots]\in\mathcal{E}_0$ and
$\theta=[b_1,b_2,\cdots]$ with $b_j\geq N$, $\forall j\geq1$ is a Brjuno number.
Let $\{A_n\}_{n=1}^{+\infty}$ be a sequence of positive integers such that
\[\sqrt[q_n]{A_n}\xrightarrow[n\to+\infty]{}+\infty.\]
Set
\[\alpha_n:=[a_1,a_2,\cdots,a_n,A_n,b_1,b_2,\cdots].\]
Then, for all $\epsilon>0$, if $n$ is sufficiently large,
\begin{itemize}
\item $P_{\alpha_n}$ has a cycle in $D(0,\epsilon)\setminus\{0\}$,
\item ${\rm area}(K(P_{\alpha_n}))\geq(1-\epsilon)\cdot{\rm area}(K(P_{\alpha})).$
\end{itemize}
\end{theorem}
It is easy to see that with the help of Theorem \ref{T2}, by Buff and Ch\'eritat's method
one can construct more quadratic Julia sets with positive area.

\section{Orbits of escaping points for the model $F_{\alpha}$}

\subsection{Basic notations\label{s2.1}}
Given an irrational number $0<\alpha<1$, the canonical candidate model
$F_{\alpha}$ is a degree $2$ topological branched covering of the complex plane
which is holomorphic outside of $\overline{D}$ and
is topologically conjugate to a rigid rotation with rotation number $\alpha$ on $\overline{\mathbb{D}}$.
Here we give some notations throughout this paper.
\begin{itemize}
\item[(1)] We denote by $U_0$ the unique component of $F_{\alpha}^{-1}(\mathbb{D})\setminus\mathbb{D}$.
Then $U_0$ is a Jordan domain whose boundary is $F_{\alpha}^{-1}(\mathbb{S}^1)\setminus(\mathbb{S}^1\setminus\{1\})$
and the intersection between $\overline{U_0}$ and $\overline{\mathbb{D}}$ is $\{1\}$.
The angle between the boundary of $U_0$ and $\mathbb{S}^1$ at $1$ is $\frac{\pi}{3}$.

\item[(2)] For any integer $j$, we let $x_j$ be a (necessarily unique) point in $\mathbb{S}^1$ with $F_{\alpha}^{\comp j}(x_j)=1$ and
$x_{0,j-1}$ be a (necessarily unique) point in $\partial U_0$ with $F_{\alpha}(x_{0,j-1})=F_{\alpha}(x_j)$.

\item[(3)] By ${\rm diam}(\cdot)$ and $d(\cdot,\cdot)$ we mean Euclidean diameter and Euclidean distance respectively.
Given $a\in\mathbb{C}$ and $r>0$, we set $\mathbb{B}(a,r):=\{z\in\mathbb{C}:|z-a|<r\}$.

\item[(4)] For any hyperbolic Riemann surface $\mathcal{D}$, by ${\rm diam}_{\mathcal{D}}(\cdot)$ and $d_{\mathcal{D}}(\cdot,\cdot)$
we mean hyperbolic diameter and hyperbolic distance in $\mathcal{D}$ respectively. Moreover,
if $a\in\mathcal{D}$ and $r>0$, we set $\mathbb{B}_{\mathcal{D}}(a,r):=\{z\in\mathcal{D}:d_{\mathcal{D}}(a,z)<r\}$.

\item[(5)] For any three real numbers $a>0, b>0, c>1$, the notation $a\preccurlyeq_c b$ means $a\leq c\cdot b$;
the notation $a\asymp_c b$ means $a\preccurlyeq_c b$ and $b\preccurlyeq_c a$;
the notation $a\preccurlyeq b$ means that
there exists a universal constant $c'>1$ such that $a\preccurlyeq_{c'}b$,
here by a universal constant we mean one which is independent of all the parameters/variables involved;
the notation $a\asymp b$ means $a\preccurlyeq b$ and $b\preccurlyeq a$.
In particular, if both $a(x)$ and $b(x)$ depend on a variable $x$ and $I$ denotes a range of $x$,
for any nonempty subset $I' \subset I$,
``for any/all $x\in I'$, $a(x)\preccurlyeq b(x)$''
means that the universal constant $c'$ does't depend also on $x$ in $I'$.
When $I'$ denotes all large enough $x$, we immediately say
``for large enough $x$, $a(x)\preccurlyeq b(x)$''.
\item[(6)] Given a decreasing sequence $\{r_n\}_{n=1}^{+\infty}$ of positive real numbers,
we call $\{r_n\}_{n=1}^{+\infty}$ \emph{a quasi-log-arithmetic sequence} if
$\lim_{n\to\infty}r_n=0$ and $\log(\frac{r_n}{r_{n+1}})\asymp1$.
\end{itemize}

\subsection{Orbits of escaping points}
Fixing a positive real number $\alpha_0$ and an irrational number $0<\alpha<1$,
we define a map $\Lambda_{\alpha_0}:\mathbb{D}_{1+\alpha_0}\to[0,+\infty]$ as the following:
\begin{itemize}
\item $\Lambda_{\alpha_0}(z)=|F_{\alpha}^{\comp m}(z)-1|$ if $F_{\alpha}^{\comp m}(z)\in U_0$ for some $m\geq0$ and
$F_{\alpha}^{\comp j}(z)\in\mathbb{D}_{1+\alpha_0}$ for all $0\leq j\leq m$;
\item $\Lambda_{\alpha_0}(z)=+\infty$ if $z\not\in K_{\alpha_0}(F_{\alpha})$;
\item $\Lambda_{\alpha_0}(z)=0$ otherwise.
\end{itemize}
A point $z\in\mathbb{D}_{1+\alpha_0}$ is called \emph{an escaping point} for $F_{\alpha}$ if $\Lambda_{\alpha_0}(z)=+\infty$.
In this section, our main purpose is to show that the orbit of any escaping point near the critical point $1$
has a ``good'' behavior of recurrence before leaving the disk $\mathbb{D}_{1+\alpha_0}$.
This phenomenon is formulated as the following proposition:
\begin{proposition}
\label{pr2.1}Given $F_{\alpha}$ with any irrational number $\alpha$,
there exists a sequence $\{z_n\}_{n=1}^{+\infty}\subset\overline{U_0}$ with
a quasi-log-arithmetic sequence $\left\{l_n:=|z_n-1|\right\}_{n=1}^{\infty}$ such that
for any $N>0$ and $n>0$,
if $z\in S_{l_{n+N}}$ with $\Lambda_{\alpha_0}(z)>l_n$, then there exists a nonnegative integer $m$ such that
$$\{F_{\alpha}^{\comp j}(z)\}_{j=0}^m\subseteq\mathbb{D}_{1+\alpha_0}\ {\rm and}\
d_{\mathbb{C}\setminus\overline{\mathbb{D}}}(F_{\alpha}^{\comp m}(z),z_n)\preccurlyeq1,$$
where $S_l=\{e^{-iz}:-l\leq{\rm Re}(z)\leq l\ {\rm and}\ -l\leq{\rm Im}(z)\leq l\}$ for all $l>0$.
\end{proposition}
\begin{remark}
{\rm By (5) in Subsection \ref{s2.1},
the implicit factor in the notation ``$\preccurlyeq$'' of the proposition is a universal constant greater than $1$ not depending on any variable.}
\end{remark}

This proposition may be well-known for experts, but for the sake of completeness,
we will give a proof based on the theory of Petersen's puzzle.
The rest of this section is devoted to giving this proof.
\subsection{The sequence of critical puzzle pieces}
We fix an irrational number $\alpha$.
For any $z,w\in\mathbb{S}^1\setminus\{F_{\alpha}(1)\}$,
we denote by $[z,w]_{\mathbb{S}^1}$ the closed arc in $\mathbb{S}^1$ not containing $F_{\alpha}(1)$ and
by $]z,w[_{\mathbb{S}^1}$ the open arc in $\mathbb{S}^1$ not containing $F_{\alpha}(1)$.
For any $z,w\in\partial U_0\setminus\{x_{0,-2}\}$, we denote by $[z,w]_{\partial U_0}$ the closed arc in $\partial U_0$ not containing $x_{0,-2}$ and
by $]z,w[_{\partial U_0}$ the open arc in $\partial U_0$ not containing $x_{0,-2}$.
For any curve $\gamma$ in $\mathbb{C}$, we denote by $l(\gamma)$ the Euclidean length of $\gamma$ and
for any curve $\gamma'$ in a hyperbolic Riemann surface $\mathcal{S}$,
we denote by $l_{\mathcal{S}}(\gamma')$ the hyperbolic length of $\gamma'$.

By the notation $P(I,O,R,B,G)$ we mean a closed Jordan domain in $\mathbb{C}$ whose boundary consists of
five arcs $I,O,R,G,B\subset\mathbb{C}\setminus\mathbb{D}$ connected end to end and satisfies the following properties:
\begin{itemize}
\item $I\subset\mathbb{S}^1$ with an endpoint $1$;
\item $O\subset\partial U_0$ with an endpoint $1$;
\item $R,G,B$ are contained in $\mathbb{C}\setminus(\overline{\mathbb{D}\cup U_0})$
except an endpoint of $B$ in $\mathbb{S}^1$ and an endpoint of $R$ in $\partial U_0$.
\end{itemize}

\begin{figure}[h]
\centering
\includegraphics[scale=0.5]{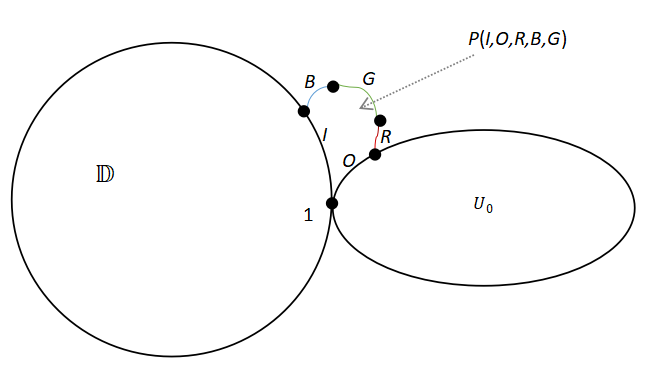}
\caption{$P(I,O,R,B,G)$ is the bounded closed Jordan domain enclosed by arcs $I,B,G,R,O$.}
\end{figure}
\noindent It follows from \cite{PZ04} that
there exists a sequence $\{P^n:=P(I_n,O_n,R_n,B_n,G_n)\}_{n=1}^{+\infty}$, called the sequence of critical puzzle pieces, with
$I_n=[1,x_{q_n}]_{\mathbb{S}^1}$ and $O_n=[1,x_{0,q_n+q_{n-1}-1}]_{\partial U_0}$
which satisfies the following properties (a)(b)(c)(d)(e):

Let $J^n:=]x_{q_n-q_{n+1}},1[_{\mathbb{S}^1}$ and $\mathbb{C}^*_{J^n}:=(\mathbb{C}^*\setminus\mathbb{S}^1)\cup J^n$. Then

\vspace{0.1cm}
\begin{itemize}
\item[(a)] For all large enough $n$, we have $l_{\mathbb{C}\setminus\overline{\mathbb{D}}}(R_n)\preccurlyeq1$,
$l_{\mathbb{C}\setminus\overline{\mathbb{D}}}(G_n)\preccurlyeq1$ and
$l_{\mathbb{C}^*_{J^n}}(B_n)\preccurlyeq1$.
\end{itemize}

\vspace{0.1cm}
\noindent We define $\tau$: $z\mapsto\frac{1}{\overline{z}}$,
$R_{0,n}:=F_{\alpha}^{\comp-1}(\tau(F_{\alpha}(R_n)))\cap\overline{U_0}$,
$G_{0,n}:=F_{\alpha}^{\comp-1}(\tau(F_{\alpha}(G_n)))\cap\overline{U_0}$ and
$B_{0,n}:=F_{\alpha}^{\comp-1}(\tau(F_{\alpha}(B_n)))\cap\overline{U_0}$. Then

\vspace{0.1cm}
\begin{itemize}
\item[(b)] For all large enough $n$, we have $l_{\mathbb{C}\setminus\overline{\mathbb{D}}}(R_{0,n})\preccurlyeq1$,
$l_{\mathbb{C}\setminus\overline{\mathbb{D}}}(G_{0,n})\preccurlyeq1$ and
$l_{\mathbb{C}\setminus\overline{\mathbb{D}}}(B_{0,n})\preccurlyeq1$.
\end{itemize}

\vspace{0.1cm}
\noindent For all $t\geq1$,
by $\varphi_S(P^t)$ we mean the closure of the component of $F_{\alpha}^{\comp-q_{t+1}}((P^t)^{\circ})$ containing $x_{q_{t+1}}$ as a boundary point.
\begin{itemize}
\item[(c)] For all large enough $n$, $F_{\alpha}^{\comp q_n}(P^n)$ can be also written as the form $P(I,O,R,B,G)$ and satisfies
$I=[1,x_{q_{n-1}}]_{\mathbb{S}^1},\ F_{\alpha}^{\comp q_n}(B_n)=O,\ F_{\alpha}^{\comp q_n}(R_n)=B,\
F_{\alpha}^{\comp q_n}(G_n)=R\cup G,\ l(O)\asymp l(I),\
l_{\mathbb{C}\setminus\overline{\mathbb{D}}}(R)\preccurlyeq1\ {\rm and}\ l_{\mathbb{C}\setminus\overline{\mathbb{D}}}(G)\preccurlyeq1$;
$F_{\alpha}^{\comp q_{n-1}}(P(I,O,R,B,G))$ can be written as $P(I',O',R',B',G')$ with
$l(I')\asymp l([1,x_{q_{n-2}}]_{\mathbb{S}^1})\asymp l(O')$ and $F_{\alpha}^{\comp q_{n-1}}(B)=O'$.
Moreover, $\varphi_S(P^n)$ can be also written as the form $P(I'',O'',R'',B'',G'')$ and satisfies
$I''=[1,x_{q_{n+1}}]_{\mathbb{S}^1}$, $O''=[1,x_{0,q_{n+1}+q_n-1}]_{\partial U_0}$,
$l_{\mathbb{C}\setminus\overline{\mathbb{D}}}(R'')\preccurlyeq1$,
$l_{\mathbb{C}\setminus\overline{\mathbb{D}}}(G'')\preccurlyeq1$ and
$l_{\mathbb{C}^*_{J^{n+1}}}(B'')\preccurlyeq1$.
\end{itemize}

\noindent Set
\begin{align*}
r_{in}(P^n)&:=\sup\{r:\mathbb{B}(1,r)\subseteq(\mathbb{C}\setminus\partial P^n)\cup\mathbb{S}^1\cup\partial U_0\},\\
r_{in}(\varphi_S(P^n))&:=\sup\{r:\mathbb{B}(1,r)\subseteq(\mathbb{C}\setminus\partial\varphi_S(P^n))\cup\mathbb{S}^1\cup\partial U_0\},\\
r_{out}(P^n)&:=\inf\{r:\mathbb{B}(1,r)\supseteq P^n\}\
{\rm and}\
r_{out}(\varphi_S(P^n)):=\inf\{r:\mathbb{B}(1,r)\supseteq\varphi_S(P^n)\}.
\end{align*}
Then

\vspace{0.1cm}
\begin{itemize}
\item[(d)] For all large enough $n$, we have $r_{in}(P^n)\asymp l(I^n)\asymp r_{out}(P^n)$ and
$r_{in}(\varphi_S(P^n))\asymp l(I^{n+1})\asymp r_{out}(\varphi_S(P^n))$.
\end{itemize}

\vspace{0.1cm}
\noindent Let
$\mathcal{P}=\{{\rm component\ of}\ F_{\alpha}^{\comp-m}((P^n)^{\circ}):m\geq-1\ {\rm and}\ n\geq1\}\cup
\{{\rm component\ of}\ F_{\alpha}^{\comp-m}(U_0):m\geq0\}$. Then

\vspace{0.1cm}
\begin{itemize}
\item[(e)] For any $P_1, P_2\in\mathcal{P}$, $P_1\cap P_2=\emptyset$ or $P_1\subseteq P_2$ or $P_1\supseteq P_2$.
Moreover, $P^1\cap P^2=\{1\}$.
\end{itemize}

\vspace{0.2cm}
\noindent In fact, it follows from [the definition of puzzles and Theorems 4.2 and 4.3, \cite{PZ04}] that (a), (c), (d) and (e) hold.
Next, we prove (b). For all large enough $n$, we set
$$K^n:=]x_{q_n-q_{n+1}},x_{-q_n}[_{\mathbb{S}^1}$$
and
$$\mathcal{D}_n:=\mathbb{C}\setminus(\gamma_0([0,1])\cup(\mathbb{S}^1\setminus K^n)\cup\tau\comp\gamma_0([0,1])),$$
where $\gamma_0:[0,1]\to\hat{\mathbb{C}}$ is a Jordan arc connecting $x_{-1}$ and
$\infty$ not intersecting $\overline{U_0\cup\mathbb{D}}$.
Let us express $F_{\alpha}^{\comp q_n}(P^n)$ as $P(I,O,R,B,G)$ and
$F_{\alpha}^{\comp q_{n-1}}(P(I,O,R,B,G))$ as $P(I',O',R',B',G')$, which satisfies the corresponding conditions in (c).
Then by the \'Swiatek-Herman real a priori bounds and an appropriate choice of $\gamma_0$,
we have
$$l_{\mathcal{D}_n}(\tau(O))\preccurlyeq1\ {\rm and}\ l_{\mathcal{D}_{n+1}}(\tau(O'))\preccurlyeq1.$$
Since the simply connected region $\mathcal{D}_n$ (resp. $\mathcal{D}_{n+1}$) contains none of critical values of
$f_{\alpha}^{\comp q_n}$ (resp. $f_{\alpha}^{\comp(q_{n-1}+q_n)}$),
there exists a univalent branch $h$ (resp. $h'$) of $f_{\alpha}^{\comp-q_n}$ (resp. $f_{\alpha}^{\comp-(q_{n-1}+q_n)}$)
on $\mathcal{D}_n$ (resp. $\mathcal{D}_{n+1}$) mapping $1$ to $x_{0,q_n-1}$ (resp. $x_{0,q_{n-1}+q_n-1}$).
Observe that $B_{0,n}=h(\tau(O))$, $R_{0,n}=h'(\tau(O'))$ and
$h$ (resp. $h'$) is an embedding from $\mathcal{D}_n$ (resp. $\mathcal{D}_{n+1}$) to $\mathbb{C}\setminus\overline{\mathbb{D}}$.
Then by the Schwarz lemma,
$$l_{\mathbb{C}\setminus\overline{\mathbb{D}}}(B_{0,n})=
l_{\mathbb{C}\setminus\overline{\mathbb{D}}}(h(\tau(O)))\leq l_{\mathcal{D}_n}(\tau(O))\preccurlyeq1$$
and
$$l_{\mathbb{C}\setminus\overline{\mathbb{D}}}(R_{0,n})=
l_{\mathbb{C}\setminus\overline{\mathbb{D}}}(h'(\tau(O')))\leq l_{\mathcal{D}_{n+1}}(\tau(O'))\preccurlyeq1.$$
Let $E^*$ be the component of $f_{\alpha}^{\comp-q_n}(\mathbb{D}^*)$ having a boundary point $x_{0,q_{n-1}+q_n-1}$,
which is contained in $U_0$.
Since $\mathbb{D}^*$ contains none of critical values of $f_{\alpha}^{\comp q_n}$,
the restriction of $f_{\alpha}^{\comp q_n}$ to $E^*$ is a holomorphic covering map.
Observe that $G_{0,n}\subseteq E^*$ is a lift of $\tau(R\cup G)$ under $f_{\alpha}^{\comp q_n}$.
Thus the Schwarz lemma and (c) give
$$l_{\mathbb{C}\setminus\overline{\mathbb{D}}}(G_{0,n})\leq l_{E^*}(G_{0,n})=l_{\mathbb{D}^*}(\tau(R\cup G))=
l_{\mathbb{C}\setminus\overline{\mathbb{D}}}(R\cup G)\preccurlyeq1.$$

\subsection{Escaping steps and pre-escaping steps}
For all $j\geq1$, we set
$$P^{0,2j}:=F_{\alpha}^{\comp-1}(\tau\comp F_{\alpha}(P^{2j}))\cap\overline{U_0}\ {\rm and}\ \tilde{P}^j:=P^{2j-1}\cup P^{2j}\cup P^{0,2j}.$$
Let
$$r_{in}(\tilde{P}^j):=\sup\{r:\mathbb{B}(1,r)\subseteq\mathbb{D}\cup\tilde{P}^j\}\
{\rm and}\ r_{out}(\tilde{P}^j):=\inf\{r:\mathbb{B}(1,r)\supseteq\tilde{P}^j\}.$$
Then the \'Swiatek-Herman real a priori bounds, (b) and (d) give that for all large enough $j$,
$$r_{in}(\tilde{P}^j)\asymp r_{out}(\tilde{P}^j)\asymp l(I_{2j-1}).$$
In the rest of this section we assume that $j$ is large enough.
In this case, $(a)-(d)$ hold and $\tilde{P}^j\subset\mathbb{D}_{1+\alpha_0}$.
\begin{figure}[h]
\centering
\includegraphics[scale=0.5]{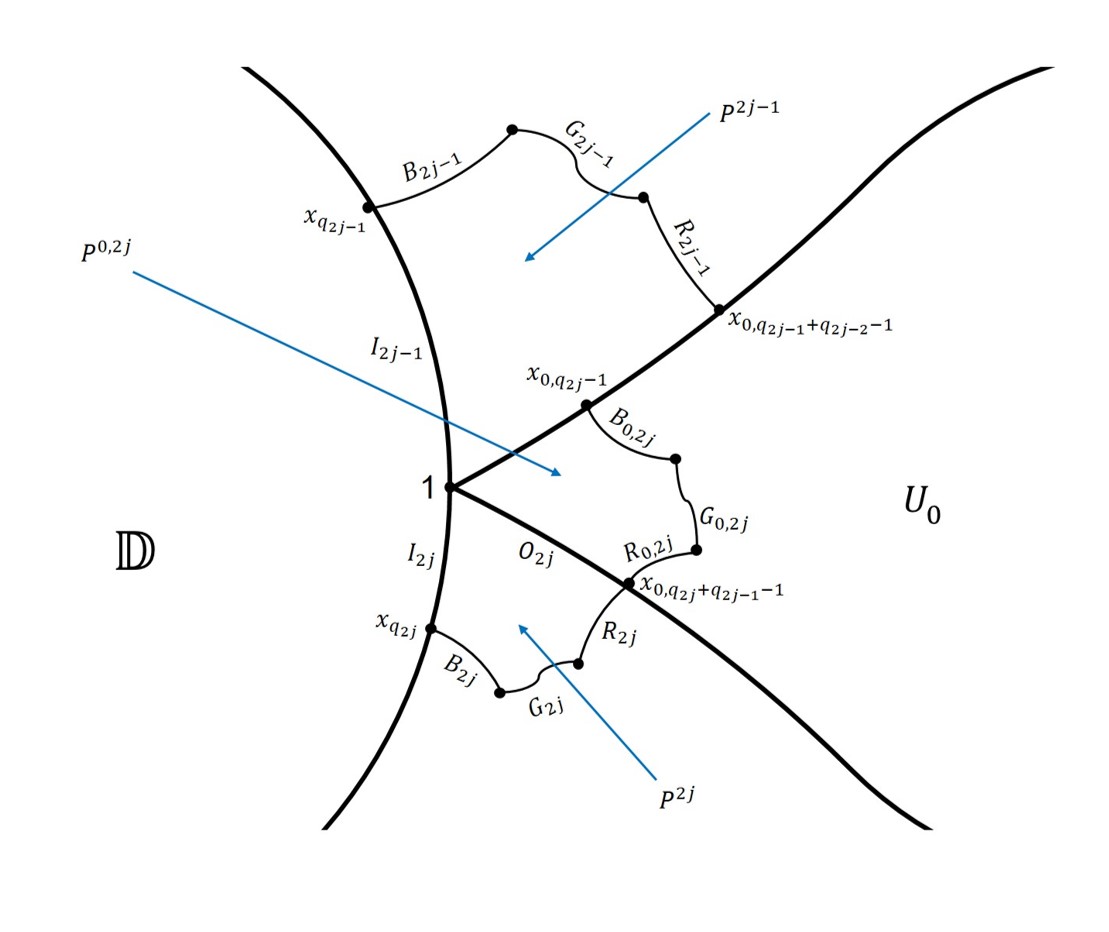}
\caption{$\tilde{P}^j$ {\rm consists of} $P^{2j-1},\ P^{2j}\ {\rm and}\ P^{0,2j}$.}
\end{figure}

We denote by $E_{2j+2}$ the closure of the component of $F_{\alpha}^{\comp-q_{2j+2}}((\tilde{P}^j)^{\comp})$
containing $x_{q_{2j+2}}$ as a boundary point
and denote by $E_{2j+1}$ the closure  of the component of $F_{\alpha}^{\comp-q_{2j+1}}((\tilde{P}^j)^{\comp})$
containing $x_{q_{2j+1}}$ as a boundary point.
It follows from (e) that $E_{2j+2}\subseteq P^{2j}$ and $E_{2j+1}\subseteq P^{2j-1}$.
Since the Jordan domain $(\tilde{P}^j)^{\comp}$ has none of critical values of $F_{\alpha}^{\comp q_{2j+1}}$ (resp. $F_{\alpha}^{\comp q_{2j+2}}$),
there exists a homeomorphic branch of $F_{\alpha}^{\comp-q_{2j+1}}$ (resp. $F_{\alpha}^{\comp-q_{2j+2}}$)
from $\tilde{P}^j$ to $E_{2j+1}$ (resp. $E_{2j+2}$), written by $h_j^{(1)}$ (resp. $h_j^{(2)}$),
and hence $E_{2j+1}$ (resp. $E_{2j+2}$) is a closed Jordan domain.
Furthermore, we have the following claim:

\vspace{0.2cm}
{\bf Claim 1.}
\begin{itemize}
\item[(aa)]
$d(\partial h_j^{(1)}(P^{2j})\setminus\partial(U_0\cup\mathbb{D}),1)\asymp l(I_{2j+1})$;
\item[(bb)]
$d(\partial h_j^{(2)}(P^{2j-1})\setminus\partial(U_0\cup\mathbb{D}),1)\asymp l(I_{2j+2})$;
\item[(cc)]
$h_j^{(1)}=F_{\alpha}^{\comp(q_{2j+2}-q_{2j+1})}\comp h_j^{(2)}$;
\item[(dd)] $E_{2j+1}\subsetneq P^{2j-1}$ and $E_{2j+2}\subsetneq P^{2j}$;
\item[(ee)] $(B_{2j-1}\cup G_{2j-1}\cup R_{2j-1})\cap\partial(P^{2j-1}\setminus E_{2j+1})\not=\emptyset$;
\item[(ff)] $(B_{2j}\cup G_{2j}\cup R_{2j})\cap\partial(P^{2j}\setminus E_{2j+2})\not=\emptyset$.
\end{itemize}

\begin{figure}[h]
\centering
\includegraphics[scale=0.3]{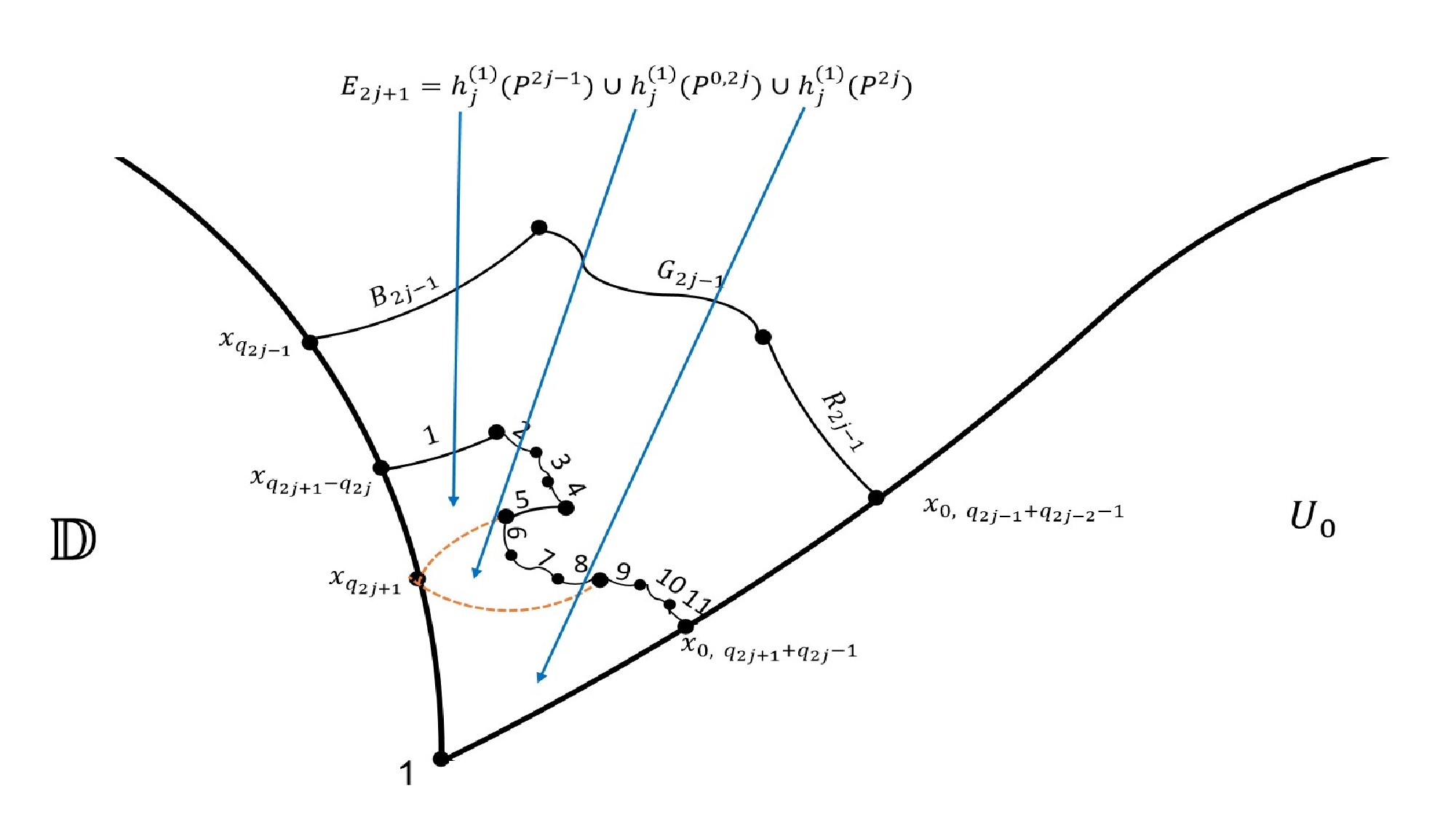}
\caption{Arcs $1-11$ mean arcs $h_j^{(1)}([x_{q_{2j-1}},x_{-q_{2j}}]_{\mathbb{S}^1}), h_j^{(1)}(B_{2j-1}),
h_j^{(1)}(G_{2j-1}), h_j^{(1)}(R_{2j-1}),$\\
$h_j^{(1)}([x_{0,q_{2j-1}+q_{2j-2}-1},x_{0,q_{2j}-1}]_{U_0}), h_j^{(1)}(B_{0,2j}), h_j^{(1)}(G_{0,2j}), h_j^{(1)}(R_{0,2j}),
h_j^{(1)}(R_{2j}), h_j^{(1)}(G_{2j}),$\\
$h_j^{(1)}(B_{2j})$ respectively.}
\end{figure}

\begin{figure}[h]
\centering
\includegraphics[scale=0.4]{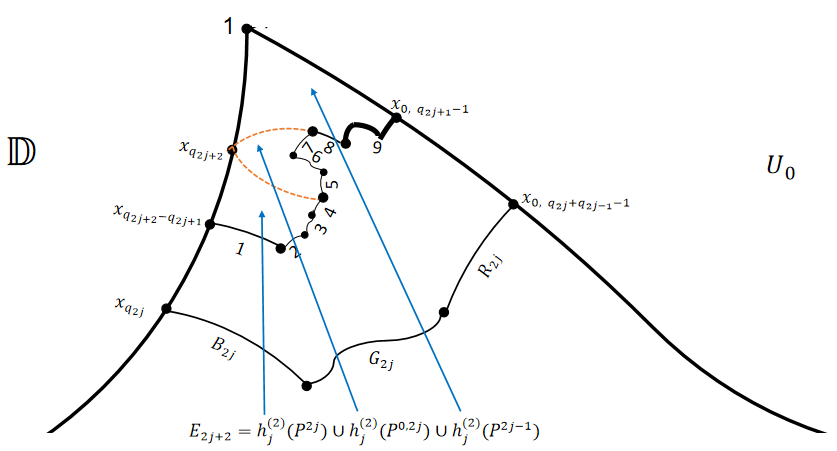}
\caption{Arcs $1-9$ mean arcs $h_j^{(2)}([x_{q_{2j}},x_{-q_{2j+1}}]_{\mathbb{S}^1}), h_j^{(2)}(B_{2j}),
h_j^{(2)}(G_{2j}), h_j^{(2)}(R_{2j}),$\\
$h_j^{(2)}(R_{0,2j}), h_j^{(2)}(G_{0,2j}), h_j^{(2)}(B_{0,2j}), h_j^{(2)}([x_{0,q_{2j-1}+q_{2j-2}-1},x_{0,q_{2j}-1}]_{U_0}),\\
h_j^{(2)}(R_{2j-1}\cup G_{2j-1}\cup B_{2j-1}\cup[x_{q_{2j-1}},x_{q_{2j+1}-q_{2j+2}}]_{\mathbb{S}^1})$ respectively,
where arc $9$ may be very complicated.}
\end{figure}

\vspace{0.2cm}
\noindent In fact, observe $h_j^{(1)}(P^{2j})=\varphi_{S}(P^{2j})$ and hence
(aa) follows immediately from (d).
By (e), we have $P^{2j-1}\supseteq P^{2j+1}$. Then (d) gives
$$l(I_{2j+2})\asymp d(\partial h_j^{(2)}(P^{2j+1})\setminus\partial(U_0\cup\mathbb{D}),1)\preccurlyeq
d(\partial h_j^{(2)}(P^{2j-1})\setminus\partial(U_0\cup\mathbb{D}),1).$$
Since $x_{q_{2j+2}}\in\overline{\partial h_j^{(2)}(P^{2j-1})\setminus\partial(U_0\cup\mathbb{D})}$,
we have
$$d(\partial h_j^{(2)}(P^{2j-1})\setminus\partial(U_0\cup\mathbb{D}),1)\leq l(I_{2j+2}).$$
Thus (bb) holds.
By definitions of $h_j^{(1)}$ and $h_j^{(2)}$, (cc) holds.
Let $U_{0,q_{2j+1}}$ be the component of $F_{\alpha}^{\comp-q_{2j+1}}(U_0)$ containing $x_{q_{2j+1}}$ as a boundary point.
By (e), we have $U_{0,q_{2j+1}}\subseteq P^{2j-1}$.
Since $U_{0,q_{2j+1}}\setminus E_{2j+1}=U_{0,q_{2j+1}}\setminus h_j^{(1)}(P^{0,2j})\not=\emptyset$,
we have $E_{2j+1}\subsetneq P^{2j-1}$. Similarly, we can obtain $E_{2j+2}\subsetneq P^{2j}$. Thus (dd) holds.
Observe that $[x_{0,q_{2j+1}+q_{2j}-1},x_{0,q_{2j-1}+q_{2j-2}-1}]_{U_0}$ is not a single point set and
$[x_{0,q_{2j+1}+q_{2j}-1},x_{0,q_{2j-1}+q_{2j-2}-1}]_{U_0}\setminus\{x_{0,q_{2j+1}+q_{2j}-1}\}$
doesn't intersect $E_{2j+1}$.
Thus $x_{0,q_{2j-1}+q_{2j-2}-1}\in(B_{2j-1}\cup G_{2j-1}\cup R_{2j-1})\cap\partial(P^{2j-1}\setminus E_{2j+1})$.
This implies (ee).
Next, we prove (ff).
It is sufficient to prove that the boundary of each component of $(P^{2j})^{\comp}\setminus E_{2j+2}$
intersects $B_{2j}\cup G_{2j}\cup R_{2j}$.
For any component $D$ of $(P^{2j})^{\comp}\setminus E_{2j+2}$,
if the boundary $\partial D$ doesn't intersect $B_{2j}\cup G_{2j}\cup R_{2j}$,
then $\partial D\subseteq[1,x_{q_{2j}}]_{\mathbb{S}^1}\cup[1,x_{0,q_{2j}+q_{2j-1}-1}]_{U_0}\cup\partial E_{2j+2}$.
Thus
$$F_{\alpha}^{\comp q_{2j+2}}(\partial D)\subseteq
F_{\alpha}^{\comp q_{2j+2}}([1,x_{q_{2j}}]_{\mathbb{S}^1}\cup[1,x_{0,q_{2j}+q_{2j-1}-1}]_{U_0}\cup\partial E_{2j+2})
\subseteq\partial\tilde{P}^j.$$
This implies that
$F_{\alpha}^{\comp q_{2j+2}}(D)=(\tilde{P}^j)^{\comp}$ and $F_{\alpha}^{\comp q_{2j+2}}(\partial D)=\partial\tilde{P}^j$.
Thus $F_{\alpha}^{\comp-q_{2j+2}}(1)\cap\partial D\not=\emptyset$.
Observe
$$F_{\alpha}^{\comp-q_{2j+2}}(1)\cap\left([1,x_{q_{2j}}]_{\mathbb{S}^1}\cup[1,x_{0,q_{2j}+q_{2j-1}-1}]_{U_0}\cup\partial E_{2j+2}\right)=
\{x_{q_{2j+2}}\}$$
and $x_{q_{2j+2}}\not\in\overline{P^{2j}\setminus E_{2j+2}}$.
Then $F_{\alpha}^{\comp-q_{2j+2}}(1)\cap\partial D=\emptyset$ and this is a contradiction.
Thus (ff) holds.

\emph{The escaping step of level $j$ for $F_{\alpha}$} is defined by
$$Y_j:=(P^{2j}\setminus E_{2j+2})\cup(P^{2j-1}\setminus E_{2j+1}),$$
and \emph{the corresponding pre-escaping step of level $j$ for $F_{\alpha}$} is defined by the following two sets:
$$X_{2j+1}:=h_j^{(1)}(Y_j)\ {\rm and}\ X_{2j+2}:=h_j^{(2)}(Y_j).$$
By (dd), we have $Y_j\not=\emptyset$, and hence $X_{2j+1}\not=\emptyset$ and $X_{2j+2}\not=\emptyset$.
Furthermore, we have the following lemma, that indicates the origin of the name of ``escaping step''.
\begin{lemma}
\label{le2.1}For all $z\in P^{2j-1}\cup P^{2j}$ with $\Lambda_{\alpha_0}(z)>{\rm diam}(P^{0,2j})$,
there exists a nonnegative integer $m_1$ such that $F_{\alpha}^{\comp m_1}(z)\in Y_j$.
Furthermore, if $z\in E_{2j+1}\cup E_{2j+2}$ with $\Lambda_{\alpha_0}(z)>{\rm diam}(P^{0,2j})$,
then there exists a nonnegative integer $m_2$ such that $F_{\alpha}^{\comp m_2}(z)\in X_{2j+1}\cup X_{2j+2}$.
\end{lemma}
\begin{proof}
Define a map $H_j:E_{2j+1}\cup E_{2j+2}\to\tilde{P}^j$ as follows:
$$H_j(z)=\left\{\begin{matrix}F_{\alpha}^{\comp q_{2j+1}}(z),&z\in E_{2j+1}\\
F_{\alpha}^{\comp q_{2j+2}}(z),&z\in E_{2j+2}\setminus\{1\}\end{matrix}\right..$$
It follows from $P^1\cap P^2=\{1\}$, $E_{2j+2}\subseteq P^{2j}$ and $E_{2j+1}\subseteq P^{2j-1}$ that the map $H_j$ is well-defined.
For any $z\in P^{2j-1}\cup P^{2j}$ with $\Lambda_{\alpha_0}(z)>{\rm diam}(P^{0,2j})$,
if $z\in Y_j$, then the lemma holds.
If $z\in E_{2j+1}\cup E_{2j+2}$ with $\Lambda_{\alpha_0}(z)>{\rm diam}(P^{0,2j})$, then
$H_j(z)\in P^{2j-1}\cup P^{2j}$.
If $H_j(z)\in E_{2j+1}\cup E_{2j+2}$,
then $\Lambda_{\alpha_0}(H_j(z))=\Lambda_{\alpha_0}(z)>{\rm diam}(P^{0,2j})$ and hence $H_j^{\comp2}(z)\in P^{2j-1}\cup P^{2j}$.
Similarly, until there exists a positive integer $k$ such that
$H_j^{\comp k}(z)\in Y_j$.
Then
$H_j^{\comp(k-1)}(z)\in X_{2j+1}\cup X_{2j+2}$.
Thus there exist two nonnegative integers $m_1$ and $m_2$ such that
$F_{\alpha}^{\comp m_1}(z)=H_j^{\comp k}(z)\in Y_j$
and
$F_{\alpha}^{\comp m_2}(z)=H_j^{\comp(k-1)}(z)\in X_{2j+1}\cup X_{2j+2}$.
\end{proof}

\subsection{The geometry of pre-escaping steps}
This subsection is devoted to proving the following claim:

\vspace{0.2cm}
{\bf Claim 2.}
$\sup_{x\in X_{2j+1}}d_{\mathbb{C}\setminus\overline{\mathbb{D}}}(x_{0,q_{2j}+q_{2j+1}-1},x)\preccurlyeq1$
and
$\sup_{x\in X_{2j+2}}d_{\mathbb{C}\setminus\overline{\mathbb{D}}}(x_{0,q_{2j+1}+q_{2j+2}-1},x)\preccurlyeq1$.

\vspace{0.2cm}
\noindent This claim follows immediately from the following four properties:

\begin{itemize}
\item[(1)] ${\rm diam}_{\mathbb{C}\setminus\overline{\mathbb{D}}}(h_j^{(1)}(P^{2j-1}\setminus E_{2j+1}))\preccurlyeq1\ {\rm and}\
{\rm diam}_{\mathbb{C}\setminus\overline{\mathbb{D}}}(h_j^{(2)}(P^{2j-1}\setminus E_{2j+1}))\preccurlyeq1$;
\item[(2)] ${\rm diam}_{\mathbb{C}\setminus\overline{\mathbb{D}}}(h_j^{(1)}(P^{2j}\setminus E_{2j+2}))\preccurlyeq1\ {\rm and}\
{\rm diam}_{\mathbb{C}\setminus\overline{\mathbb{D}}}(h_j^{(2)}(P^{2j}\setminus E_{2j+2}))\preccurlyeq1$;
\item[(3)] $d_{\mathbb{C}\setminus\overline{\mathbb{D}}}(x_{0,q_{2j}+q_{2j+1}-1},h_j^{(1)}(P^{2j-1}\setminus E_{2j+1}))
\preccurlyeq1$
and
$d_{\mathbb{C}\setminus\overline{\mathbb{D}}}(x_{0,q_{2j}+q_{2j+1}-1},h_j^{(1)}(P^{2j}\setminus E_{2j+2}))\preccurlyeq1$;
\item[(4)]$d_{\mathbb{C}\setminus\overline{\mathbb{D}}}(x_{0,q_{2j+1}+q_{2j+2}-1},h_j^{(2)}(P^{2j-1}\setminus E_{2j+1}))
\preccurlyeq1$
and
$d_{\mathbb{C}\setminus\overline{\mathbb{D}}}(x_{0,q_{2j+1}+q_{2j+2}-1},h_j^{(2)}(P^{2j}\setminus E_{2j+2}))
\preccurlyeq1$.
\end{itemize}

\vspace{0.2cm}
\noindent{\bf The proof of Property (1):}
Observe that $P^{2j-1}\setminus E_{2j+1}$ is enclosed by $16$ arcs
$R_{2j-1}$, $G_{2j-1}$, $B_{2j-1}$, $[x_{q_{2j+1}-q_{2j}},x_{q_{2j-1}}]_{\mathbb{S}^1}$,
$h_j^{(1)}(G_{2j})$, $h_j^{(1)}(R_{2j})$, $h_j^{(1)}(R_{0,2j})$, $h_j^{(1)}(G_{0,2j})$, $h_j^{(1)}(B_{0,2j})$,\\
$h_j^{(1)}([x_{0,q_{2j}-1},x_{0,q_{2j-1}+q_{2j-2}-1}]_{\partial U_0})$, $h_j^{(1)}(R_{2j-1})$, $h_j^{(1)}(G_{2j-1})$,
$h_j^{(1)}(B_{2j})$, $h_j^{(1)}(B_{2j-1})$, $h_j^{(1)}([x_{-q_{2j}},x_{q_{2j-1}}]_{\mathbb{S}^1})$ and
$[x_{0,q_{2j}+q_{2j+1}-1},x_{0,q_{2j-1}+q_{2j-2}-1}]_{\partial U_0}$, see Figure $3$.
Thus to prove Property (1), we only need to prove that for each $\gamma$ of these $16$ arcs, we have
$$l_{\mathbb{C}\setminus\overline{\mathbb{D}}}(h_j^{(1)}(\gamma))\preccurlyeq1\ {\rm and}\
l_{\mathbb{C}\setminus\overline{\mathbb{D}}}(h_j^{(2)}(\gamma))\preccurlyeq1.$$
It follows from (cc) and the Schwarz lemma that it is sufficient to prove
\begin{equation}
\label{e2.2}l_{\mathbb{C}\setminus\overline{\mathbb{D}}}(F_{\alpha}^{\comp(q_{2j+1}-q_{2j})}\comp h_j^{(1)}(\gamma))\preccurlyeq1.
\end{equation}
Firstly, the \'Swiatek-Herman real a priori bounds gives
$$l_{\mathbb{C}\setminus\overline{\mathbb{D}}}([x_{0,q_{2j}+q_{2j+1}-1},x_{0,q_{2j-1}+q_{2j-2}-1}]_{\partial U_0})\preccurlyeq1$$
and by (a), (b), the \'Swiatek-Herman real a priori bounds and the Schwarz lemma, we have
$l_{\mathbb{C}\setminus\overline{\mathbb{D}}}(R_{2j-1})\preccurlyeq1$,
$l_{\mathbb{C}\setminus\overline{\mathbb{D}}}(G_{2j-1})\preccurlyeq1$,
$l_{\mathbb{C}\setminus\overline{\mathbb{D}}}(h_j^{(1)}(G_{2j}))\preccurlyeq1$,
$l_{\mathbb{C}\setminus\overline{\mathbb{D}}}(h_j^{(1)}(R_{2j}))\preccurlyeq1$,
$l_{\mathbb{C}\setminus\overline{\mathbb{D}}}(h_j^{(1)}(R_{0,2j}))\preccurlyeq1$,
$l_{\mathbb{C}\setminus\overline{\mathbb{D}}}(h_j^{(1)}(G_{0,2j}))\preccurlyeq1$,
$l_{\mathbb{C}\setminus\overline{\mathbb{D}}}(h_j^{(1)}(B_{0,2j}))\preccurlyeq1$,
$l_{\mathbb{C}\setminus\overline{\mathbb{D}}}(h_j^{(1)}(R_{2j-1}))\preccurlyeq1$,
$l_{\mathbb{C}\setminus\overline{\mathbb{D}}}(h_j^{(1)}(G_{2j-1}))\preccurlyeq1$ and
$$l_{\mathbb{C}\setminus\overline{\mathbb{D}}}(h_j^{(1)}([x_{0,q_{2j}-1},x_{0,q_{2j-1}+q_{2j-2}-1}]_{\partial U_0}))\preccurlyeq1.$$

Next, we prove
$$l_{\mathbb{C}\setminus\overline{\mathbb{D}}}(h_j^{(1)}(B_{2j}))\preccurlyeq1\
{\rm and}\ l_{\mathbb{C}\setminus\overline{\mathbb{D}}}(h_j^{(1)}(B_{2j-1})\preccurlyeq1.$$
By (c), $F_{\alpha}^{\comp q_{2j}}(P^{2j})$ can be expressed as
$P(I,O,B,R,G)$
with
$I=[1,x_{q_{2j-1}}]_{\mathbb{S}^1}\ {\rm and}\ l(O)\asymp l(I)$;
$F_{\alpha}^{\comp q_{2j-1}}(P^{2j-1})$ can be expressed as $P(I',O',B',R',G')$
with
$I'=[1,x_{q_{2j-2}}]_{\mathbb{S}^1}\ {\rm and}\ l(O')\asymp l(I')$.
By the \'Swiatek-Herman real a priori bounds and an appropriate choice of $\gamma_0$,
we have $l_{\mathcal{D}_{2j+2}}(O)\asymp1$ and $l_{\mathcal{D}_{2j+1}}(O')\asymp1$.
Since the simply connected region $\mathcal{D}_{2j+2}$ (resp. $\mathcal{D}_{2j+1}$) contains none of critical values of
$f_{\alpha}^{\comp(q_{2j+1}+q_{2j})}$ (resp. $f_{\alpha}^{\comp(q_{2j}+q_{2j-1})}$),
there exists a univalent branch $h$ (resp. $h'$) of $f_{\alpha}^{\comp-(q_{2j+1}+q_{2j})}$ (resp. $f_{\alpha}^{\comp-(q_{2j}+q_{2j-1})}$)
on $\mathcal{D}_{2j+2}$ (resp. $\mathcal{D}_{2j+1}$) mapping $1$ to $x_{0,q_{2j+1}+q_{2j}-1}$ (resp. $x_{0,q_{2j}+q_{2j-1}-1}$).
Observe that $h$ (resp. $h'$) is an embedding from $\mathcal{D}_{2j+2}$ (resp. $\mathcal{D}_{2j+1}$) to $\mathbb{C}\setminus\overline{\mathbb{D}}$.
Then by the Schwarz lemma,
$$l_{\mathbb{C}\setminus\overline{\mathbb{D}}}(h(O))\preccurlyeq1\ {\rm and}\
l_{\mathbb{C}\setminus\overline{\mathbb{D}}}(h'(O'))\preccurlyeq1.$$
Observe
$$h(O)=h_j^{(1)}(B_{2j})\ {\rm and}\ h'(O')=F_{\alpha}^{\comp(q_{2j+1}-q_{2j})}(h_j^{(1)}(B_{2j-1})).$$
Thus
\begin{equation}
\label{e2.1}l_{\mathbb{C}\setminus\overline{\mathbb{D}}}(h_j^{(1)}(B_{2j}))\preccurlyeq1\ {\rm and}\
l_{\mathbb{C}\setminus\overline{\mathbb{D}}}(F_{\alpha}^{\comp(q_{2j+1}-q_{2j})}(h_j^{(1)}(B_{2j-1})))\preccurlyeq1.
\end{equation}
By the Schwarz lemma,
\begin{equation*}
l_{\mathbb{C}\setminus\overline{\mathbb{D}}}(h_j^{(1)}(B_{2j-1}))\preccurlyeq1.
\end{equation*}

The above two steps give that except these three arcs $B_{2j-1}$, $[x_{q_{2j+1}-q_{2j}},x_{q_{2j-1}}]_{\mathbb{S}^1}$
and $h_j^{(1)}([x_{-q_{2j}},x_{q_{2j-1}}]_{\mathbb{S}^1})$,
each other arc $\gamma$ in the boundary of $P^{2j-1}\setminus E_{2j+1}$ satisfies
$l_{\mathbb{C}\setminus\overline{\mathbb{D}}}(\gamma)\preccurlyeq1$, and
by the Schwarz lemma we have
$$l_{\mathbb{C}\setminus\overline{\mathbb{D}}}(F_{\alpha}^{\comp(q_{2j+1}-q_{2j})}\comp h_j^{(1)}(\gamma))\preccurlyeq1.$$
Thus, at last, the rest is to prove
$$l_{\mathbb{C}\setminus\overline{\mathbb{D}}}(F_{\alpha}^{\comp(q_{2j+1}-q_{2j})}\comp h_j^{(1)}(B_{2j-1}))\preccurlyeq1,$$
$$l_{\mathbb{C}\setminus\overline{\mathbb{D}}}(F_{\alpha}^{\comp(q_{2j+1}-q_{2j})}\comp h_j^{(1)}([x_{q_{2j+1}-q_{2j}},x_{q_{2j-1}}]_{\mathbb{S}^1}))\preccurlyeq1$$
and
$$l_{\mathbb{C}\setminus\overline{\mathbb{D}}}(F_{\alpha}^{\comp(q_{2j+1}-q_{2j})}\comp h_j^{(1)}(h_j^{(1)}([x_{-q_{2j}},x_{q_{2j-1}}]_{\mathbb{S}^1})))\preccurlyeq1.$$
By (\ref{e2.1}), we have
$l_{\mathbb{C}\setminus\overline{\mathbb{D}}}(F_{\alpha}^{\comp(q_{2j+1}-q_{2j})}\comp h_j^{(1)}(B_{2j-1}))\preccurlyeq1$.
Observe
$$F_{\alpha}^{\comp(q_{2j+1}-q_{2j})}\comp h_j^{(1)}([x_{q_{2j+1}-q_{2j}},x_{q_{2j-1}}]_{\mathbb{S}^1})
=[x_{0,q_{2j+1}-1},x_{0,q_{2j-1}+q_{2j}-1}]_{\partial U_0}.$$
Again, by the \'Swiatek-Herman real a priori bounds
$$l_{\mathbb{C}\setminus\overline{\mathbb{D}}}([x_{0,q_{2j+1}-1},x_{0,q_{2j-1}+q_{2j}-1}]_{\partial U_0})\preccurlyeq1.$$
Thus we have
$$l_{\mathbb{C}\setminus\overline{\mathbb{D}}}(F_{\alpha}^{\comp(q_{2j+1}-q_{2j})}\comp h_j^{(1)}([x_{q_{2j+1}-q_{2j}},x_{q_{2j-1}}]_{\mathbb{S}^1}))\preccurlyeq1.$$
Since $\mathcal{D}_{2j+1}$ contains none of critical values of
$f_{\alpha}^{\comp q_{2j+1}}$, there exists a univalent branch
$h''$ of $f_{\alpha}^{\comp-q_{2j+1}}$ on $\mathcal{D}_{2j+1}$ mapping $1$ to $x_{0,q_{2j+1}-1}$ and
$h''$ is an embedding from $\mathcal{D}_{2j+1}$ to $\mathbb{C}\setminus\overline{\mathbb{D}}$.
By \'Swiatek-Herman real a priori bounds and an appropriate choice of $\gamma_0$, $l_{\mathcal{D}_{2j+1}}([1,x_{0,q_{2j-1}+q_{2j}-1}]_{\partial U_0})\asymp1$.
Then it follows from the Schwarz lemma that
$$l_{\mathbb{C}\setminus\overline{\mathbb{D}}}(h''([1,x_{0,q_{2j-1}+q_{2j}-1}]_{\partial U_0}))\preccurlyeq1.$$
Again, observe
$$F_{\alpha}^{\comp(q_{2j+1}-q_{2j})}\comp h_j^{(1)}(h_j^{(1)}([x_{-q_{2j}},x_{q_{2j-1}}]_{\mathbb{S}^1}))=
h''([1,x_{0,q_{2j-1}+q_{2j}-1}]_{\partial U_0}).$$
Then we have
$$l_{\mathbb{C}\setminus\overline{\mathbb{D}}}(F_{\alpha}^{\comp(q_{2j+1}-q_{2j})}\comp h_j^{(1)}(h_j^{(1)}([x_{-q_{2j}},x_{q_{2j-1}}]_{\mathbb{S}^1})))\preccurlyeq1.$$

\vspace{0.2cm}
\noindent{\bf The proof of Property (2):}
Set
$$\tilde{E}_{2j+2}:=h_j^{(2)}(P^{2j+1})\cup h_j^{(2)}(P^{2j})\cup h_j^{(2)}(P^{0,2j+1}).$$
By (e), we have $\tilde{E}_{2j+2}\subseteq E_{2j+2}$.
Observe that (\ref{e2.2}) in the proof of Property (1) implies
$${\rm diam}_{\mathbb{C}\setminus\overline{\mathbb{D}}}(F_{\alpha}^{\comp(q_{2j+1}-q_{2j})}\comp h_j^{(1)}(P^{2j-1}\setminus E_{2j+1}))\preccurlyeq1.$$
In the same way, one can obtain
\begin{equation*}
{\rm diam}_{\mathbb{C}\setminus\overline{\mathbb{D}}}(F_{\alpha}^{\comp(q_{2j+2}-q_{2j+1})}\comp  h_j^{(2)}(P^{2j}\setminus\tilde{E}_{2j+2}))\preccurlyeq1.
\end{equation*}
Then by (cc), we have
\begin{equation*}
{\rm diam}_{\mathbb{C}\setminus\overline{\mathbb{D}}}(h_j^{(1)}(P^{2j}\setminus\tilde{E}_{2j+2}))\preccurlyeq1
\end{equation*}
and by the Schwarz lemma
\begin{equation*}
{\rm diam}_{\mathbb{C}\setminus\overline{\mathbb{D}}}(h_j^{(2)}(P^{2j}\setminus\tilde{E}_{2j+2}))\preccurlyeq1.
\end{equation*}
Thus
\begin{equation*}
{\rm diam}_{\mathbb{C}\setminus\overline{\mathbb{D}}}(h_j^{(1)}(P^{2j}\setminus E_{2j+2}))\preccurlyeq1\ {\rm and}\
{\rm diam}_{\mathbb{C}\setminus\overline{\mathbb{D}}}(h_j^{(2)}(P^{2j}\setminus E_{2j+2}))\preccurlyeq1.
\end{equation*}

\vspace{0.2cm}
\noindent{\bf The proof of Property (3):} Observe that the path
$h_j^{(1)}(B_{2j})+h_j^{(1)}(G_{2j})+h_j^{(1)}(R_{2j})+h_j^{(1)}(R_{0,2j})+h_j^{(1)}(G_{0,2j})+h_j^{(1)}(B_{0,2j})$
connects $x_{0,q_{2j}+q_{2j+1}-1}=h_j^{(1)}(x_{q_{2j}})$ and the boundary of $h_j^{(1)}(P^{2j-1}\setminus E_{2j+1})$.
Again, in the proof of Property (1) we see that
$l_{\mathbb{C}\setminus\overline{\mathbb{D}}}(h_j^{(1)}(B_{2j}))\preccurlyeq1,\
l_{\mathbb{C}\setminus\overline{\mathbb{D}}}(h_j^{(1)}(G_{2j}))\preccurlyeq1,\
l_{\mathbb{C}\setminus\overline{\mathbb{D}}}(h_j^{(1)}(R_{2j}))\preccurlyeq1,\
l_{\mathbb{C}\setminus\overline{\mathbb{D}}}(h_j^{(1)}(R_{0,2j}))\preccurlyeq1,\
l_{\mathbb{C}\setminus\overline{\mathbb{D}}}(h_j^{(1)}(G_{0,2j}))\preccurlyeq1,\
l_{\mathbb{C}\setminus\overline{\mathbb{D}}}(h_j^{(1)}(B_{0,2j}))
\preccurlyeq1$.
Thus
\begin{align*}
&d_{\mathbb{C}\setminus\overline{\mathbb{D}}}(x_{0,q_{2j}+q_{2j+1}-1},h_j^{(1)}(P^{2j-1}\setminus E_{2j+1}))\\
\leq& l_{\mathbb{C}\setminus\overline{\mathbb{D}}}(h_j^{(1)}(B_{2j}))+l_{\mathbb{C}\setminus\overline{\mathbb{D}}}(h_j^{(1)}(G_{2j}))+
l_{\mathbb{C}\setminus\overline{\mathbb{D}}}(h_j^{(1)}(R_{2j}))
+l_{\mathbb{C}\setminus\overline{\mathbb{D}}}(h_j^{(1)}(R_{0,2j}))
+l_{\mathbb{C}\setminus\overline{\mathbb{D}}}(h_j^{(1)}(G_{0,2j}))\\
&+l_{\mathbb{C}\setminus\overline{\mathbb{D}}}(h_j^{(1)}(B_{0,2j}))\\
\preccurlyeq&1.
\end{align*}
By (ff), the path $h_j^{(1)}(B_{2j})+h_j^{(1)}(G_{2j})+h_j^{(1)}(R_{2j})$ or its subpath connects
$x_{0,q_{2j}+q_{2j+1}-1}$ and the boundary of $h_j^{(1)}(P^{2j}\setminus E_{2j+2})$, and hence
\begin{align*}
&d_{\mathbb{C}\setminus\overline{\mathbb{D}}}(x_{0,q_{2j}+q_{2j+1}-1},h_j^{(1)}(P^{2j}\setminus E_{2j+2}))\\
\leq& l_{\mathbb{C}\setminus\overline{\mathbb{D}}}(h_j^{(1)}(B_{2j}))+l_{\mathbb{C}\setminus\overline{\mathbb{D}}}(h_j^{(1)}(G_{2j}))+
l_{\mathbb{C}\setminus\overline{\mathbb{D}}}(h_j^{(1)}(R_{2j}))\\
\preccurlyeq&1.
\end{align*}
Thus Property (3) follows.

\vspace{0.2cm}
\noindent{\bf The proof of Property (4):}
In the same way as proving
$$l_{\mathbb{C}\setminus\overline{\mathbb{D}}}(h_j^{(1)}(B_{2j}))\preccurlyeq1,
l_{\mathbb{C}\setminus\overline{\mathbb{D}}}(h_j^{(1)}(G_{2j}))\preccurlyeq1,
l_{\mathbb{C}\setminus\overline{\mathbb{D}}}(h_j^{(1)}(R_{2j}))\preccurlyeq1,$$
$$l_{\mathbb{C}\setminus\overline{\mathbb{D}}}(h_j^{(1)}(R_{0,2j}))\preccurlyeq1,
l_{\mathbb{C}\setminus\overline{\mathbb{D}}}(h_j^{(1)}(G_{0,2j}))\preccurlyeq1,
l_{\mathbb{C}\setminus\overline{\mathbb{D}}}(h_j^{(1)}(B_{0,2j}))\preccurlyeq1$$
and
$$l_{\mathbb{C}\setminus\overline{\mathbb{D}}}(h_j^{(1)}([x_{0,q_{2j}-1},x_{0,q_{2j-1}+q_{2j-2}-1}]_{\partial U_0}))
\preccurlyeq1,$$
we can prove
$$l_{\mathbb{C}\setminus\overline{\mathbb{D}}}(h_j^{(2)}(B_{2j+1}))\preccurlyeq1,
l_{\mathbb{C}\setminus\overline{\mathbb{D}}}(h_j^{(2)}(G_{2j+1}))\preccurlyeq1,
l_{\mathbb{C}\setminus\overline{\mathbb{D}}}(h_j^{(2)}(R_{2j+1}))\preccurlyeq1,$$
$$l_{\mathbb{C}\setminus\overline{\mathbb{D}}}(h_j^{(2)}(R_{0,2j+1}))\preccurlyeq1,
l_{\mathbb{C}\setminus\overline{\mathbb{D}}}(h_j^{(2)}(G_{0,2j+1}))\preccurlyeq1,
l_{\mathbb{C}\setminus\overline{\mathbb{D}}}(h_j^{(2)}(B_{0,2j+1}))\preccurlyeq1$$
and
$$l_{\mathbb{C}\setminus\overline{\mathbb{D}}}(h_j^{(2)}([x_{0,q_{2j+1}-1},x_{0,q_{2j}+q_{2j-1}-1}]_{\partial U_0}))
\preccurlyeq1.$$
Applying the Schwarz lemma to $l_{\mathbb{C}\setminus\overline{\mathbb{D}}}(h_j^{(1)}(B_{2j}))\preccurlyeq1,
l_{\mathbb{C}\setminus\overline{\mathbb{D}}}(h_j^{(1)}(G_{2j}))\preccurlyeq1$ and
$l_{\mathbb{C}\setminus\overline{\mathbb{D}}}(h_j^{(1)}(R_{2j}))\preccurlyeq1,$
we have
$$l_{\mathbb{C}\setminus\overline{\mathbb{D}}}(h_j^{(2)}(B_{2j}))\preccurlyeq1,
l_{\mathbb{C}\setminus\overline{\mathbb{D}}}(h_j^{(2)}(G_{2j}))\preccurlyeq1\ {\rm and}\
l_{\mathbb{C}\setminus\overline{\mathbb{D}}}(h_j^{(2)}(R_{2j}))\preccurlyeq1.$$
Observe that the path $h_j^{(2)}(B_{2j+1})+h_j^{(2)}(G_{2j+1})+h_j^{(2)}(R_{2j+1})$ connects $x_{0,q_{2j+1}+q_{2j+2}-1}$ and
the boundary of $h_j^{(2)}(P^{2j-1}\setminus E_{2j+1})$; by (ff),
the path $h_j^{(2)}(B_{2j+1})+h_j^{(2)}(G_{2j+1})+h_j^{(2)}(R_{2j+1})+h_j^{(2)}(R_{0,2j+1})+h_j^{(2)}(G_{0,2j+1})+h_j^{(2)}(B_{0,2j+1})+
h_j^{(2)}([x_{0,q_{2j+1}-1},x_{0,q_{2j}+q_{2j-1}-1}]_{\partial U_0})+h_j^{(2)}(R_{2j})+h_j^{(2)}(G_{2j})+h_j^{(2)}(B_{2j})$ or its subpath
connects $x_{0,q_{2j+1}+q_{2j+2}-1}$ and the boundary of $h_j^{(2)}(P^{2j}\setminus E_{2j+2})$.
Thus
\begin{align*}
&d_{\mathbb{C}\setminus\overline{\mathbb{D}}}(x_{0,q_{2j+1}+q_{2j+2}-1},h_j^{(2)}(P^{2j-1}\setminus E_{2j+1}))\\
\leq& l_{\mathbb{C}\setminus\overline{\mathbb{D}}}(h_j^{(2)}(B_{2j+1}))+l_{\mathbb{C}\setminus\overline{\mathbb{D}}}(h_j^{(2)}(G_{2j+1}))+
l_{\mathbb{C}\setminus\overline{\mathbb{D}}}(h_j^{(2)}(R_{2j+1}))\\
\preccurlyeq&1
\end{align*}
and
\begin{align*}
&d_{\mathbb{C}\setminus\overline{\mathbb{D}}}(x_{0,q_{2j+1}+q_{2j+2}-1},h_j^{(2)}(P^{2j}\setminus E_{2j+2}))\\
\leq& l_{\mathbb{C}\setminus\overline{\mathbb{D}}}(h_j^{(2)}(B_{2j+1}))+l_{\mathbb{C}\setminus\overline{\mathbb{D}}}(h_j^{(2)}(G_{2j+1}))+
l_{\mathbb{C}\setminus\overline{\mathbb{D}}}(h_j^{(2)}(R_{2j+1}))
+l_{\mathbb{C}\setminus\overline{\mathbb{D}}}(h_j^{(2)}(R_{0,2j+1}))
+l_{\mathbb{C}\setminus\overline{\mathbb{D}}}(h_j^{(2)}(G_{0,2j+1}))\\
&+l_{\mathbb{C}\setminus\overline{\mathbb{D}}}(h_j^{(2)}(B_{0,2j+1}))
+l_{\mathbb{C}\setminus\overline{\mathbb{D}}}(h_j^{(2)}([x_{0,q_{2j+1}-1},x_{0,q_{2j}+q_{2j-1}-1}]_{\partial U_0}))\\
&+l_{\mathbb{C}\setminus\overline{\mathbb{D}}}(h_j^{(2)}(R_{2j}))+l_{\mathbb{C}\setminus\overline{\mathbb{D}}}(h_j^{(2)}(G_{2j}))+
l_{\mathbb{C}\setminus\overline{\mathbb{D}}}(h_j^{(2)}(B_{2j}))\\
\preccurlyeq&1.
\end{align*}

\subsection{The proof of Proposition \ref{pr2.1}}
We define a decreasing sequence $\{l_t\}_{t=1}^{+\infty}$ of positive real numbers and
an increasing sequence $\{n_t\}_{t=1}^{+\infty}$ of large enough positive integers such that for all $t\geq1$,
\begin{itemize}
\item[($A_1$)] $S_{l_t}\supseteq\tilde{P}^{n_{2t-1}}$
with $l_t\asymp{\rm diam}(\tilde{P}^{n_{2t-1}})$ and $l_t>{\rm diam}(P^{0,2n_{2t-1}})$;
\item[($A_2$)] $S_{l_{t+1}}\subseteq\tilde{P}^{n_{2t}}$ with $l_{t+1}\asymp{\rm diam}(\tilde{P}^{n_{2t}})$;
\item[($A_3$)] $E^{n_{2t-1}}:=E_{2n_{2t-1}+1}\cup E_{2n_{2t-1}+2}\supseteq P^{2n_{2t}-1}\cup P^{2n_{2t}}$ and $n_{t+1}-n_t\preccurlyeq1$.
\end{itemize}
The existence follows easily from (b), (d), (aa), (bb) and  the \'Swiatek-Herman real a priori bounds.
By ($A_1$), ($A_2$), ($A_3$) and the \'Swiatek-Herman real a priori bounds, we have
$$l_t\geq l([1,x_{q_{2n_{2t-1}-1}}]_{\mathbb{S}^1}),\
l_{t+1}\leq l([1,x_{q_{2n_{2t}}}]_{\mathbb{S}^1})\ {\rm and}\
l_t\asymp l_{t+1}.$$
Then $\{l_t\}_{t=1}^{+\infty}$ is a quasi-log-arithmetic sequence.
In this case, we take $z_t$ as a point in the intersection
between $\partial\mathbb{B}(1,l_t)$ and $\overline{U_0}$.
At last, we check that for large enough $t_0$,
$\{z_t\}_{t\geq t_0}$ is what we want.
We assume that $t$ is large enough. By \'Swiatek-Herman real a priori bounds and ($A_1$), we have
$$d_{\mathbb{C}\setminus\overline{\mathbb{D}}}(x_{0,q_{2n_{2t-1}}+q_{2n_{2t-1}+1}-1},z_t)\preccurlyeq1\ {\rm and}\
d_{\mathbb{C}\setminus\overline{\mathbb{D}}}(x_{0,q_{2n_{2t-1}+1}+q_{2n_{2t-1}+2}-1},z_t)\preccurlyeq1.$$
Then by Claim $2$,
\begin{equation}
\label{e2.5}\sup_{z\in X_{2n_{2t-1}+1}\cup X_{2n_{2t-1}+2}}d_{\mathbb{C}\setminus\overline{\mathbb{D}}}(z_t,z)\preccurlyeq1.
\end{equation}
For all $z\in S_{l_{t+1}}$ with $\Lambda_{\alpha_0}(z)>l_t$, we have
$z\in E^{n_{2t-1}}$ and $\Lambda_{\alpha_0}(z)>{\rm diam}(P^{0,2n_{2t-1}})$.
It follows from Lemma \ref{le2.1}
that there exists $m\geq0$ such that $F_{\alpha}^{\comp m}(z)\in X_{2n_{2t-1}+1}\cup X_{2n_{2t-1}+2}$, and hence
by (\ref{e2.5}), we have
$$d_{\mathbb{C}\setminus\overline{\mathbb{D}}}(F_{\alpha}^{\comp m}(z),z_t)\preccurlyeq1.$$
Now to complete the rest of the proof, we only need to prove that
for large enough $t$,
$$F_{\alpha}^{\comp s}(h_t^{(1)}(P^{2t-1}\cup P^{2t}))
\subseteq\mathbb{D}_{1+\alpha_0}$$
for all $0\leq s\leq q_{2t+1}$;
$$F_{\alpha}^{\comp s}(h_t^{(2)}(P^{2t-1}\cup P^{2t}))\subseteq\mathbb{D}_{1+\alpha_0}$$
for all $0\leq s\leq q_{2t+2}$.
In fact, this follows from the following lemma:
\begin{lemma}
For large enough $t$, we have
$$W:=\cup_{s=0}^{q_{t+1}+q_t-1}
F_{\alpha}^{\comp(q_{t+2}+q_{t+1}-s)}(\varphi_S^{\comp2}(P^t))\subseteq\mathbb{D}_{1+\alpha_0},$$
where $\varphi_S^{\comp2}(P^t)$ means the closure of the component of $F_{\alpha}^{-\comp(q_{t+1}+q_{t+2})}((P^t)^{\circ})$
containing $x_{q_{t+2}}$ as a boundary point. Moreover, for any positive integer $s$,
each component $P$ of $F_{\alpha}^{\comp-s}((P^t)^{\comp})$
whose closure intersects $\mathbb{S}^1$ satisfies $P\subseteq W$.
\end{lemma}
\begin{proof}
Let
$$J:=]x_{q_{t+1}},x_{q_t-q_{t+1}}[_{\mathbb{S}^1},\ \mathbb{C}^*_J:=J\cup(\mathbb{C}^*\setminus\mathbb{S}^1),$$
$$J':=]x_{q_{t+2}},x_{q_{t+1}-q_{t+2}}[_{\mathbb{S}^1},\ \mathbb{C}^*_{J'}:=J'\cup(\mathbb{C}^*\setminus\mathbb{S}^1).$$
Then by the \'Swiatek-Herman real a priori bounds and (a), we have
$$\sup_{z\in P^t}d_{\mathbb{C}^*_J}(J,z)\preccurlyeq1.$$
For all $0\leq s\leq q_{t+1}-1$,
by the Schwarz lemma we have
$$\sup_{z\in F_{\alpha}^{\comp(q_{t+2}+q_{t+1}-s)}(\varphi_S^{\comp2}(P^t))}
d_{\mathbb{C}^*_{(F_{\alpha}|_{\mathbb{S}^1})^{\comp-s}(J)}}\left((F_{\alpha}|_{\mathbb{S}^1})^{\comp-s}(J),z\right)\preccurlyeq1.$$
Similarly, by the \'Swiatek-Herman real a priori bounds and (c), we have
$$\sup_{z\in \varphi_S(P^t)}d_{\mathbb{C}^*_{J'}}(J',z)\preccurlyeq1.$$
For all $q_{t+1}\leq s\leq q_{t+1}+q_t-1$,
by the Schwarz lemma we have
$$\sup_{z\in F_{\alpha}^{\comp(q_{t+2}+q_{t+1}-s)}(\varphi_S^{\comp2}(P^t))}
d_{\mathbb{C}^*_{(F_{\alpha}|_{\mathbb{S}^1})^{\comp(q_{t+1}-s)}(J')}}
((F_{\alpha}|_{\mathbb{S}^1})^{\comp(q_{t+1}-s)}(J'),z)\preccurlyeq1.$$
Observe
$$\lim_{t\to\infty}\max\{l((F_{\alpha}|_{\mathbb{S}^1})^{\comp-s}(J)):0\leq s\leq q_{t+1}-1\}=0$$
and
$$\lim_{t\to\infty}\max\{l((F_{\alpha}|_{\mathbb{S}^1})^{\comp(q_{t+1}-s)}(J')):q_{t+1}\leq s\leq q_{t+1}+q_t-1\}=0.$$
Thus $W\subseteq\mathbb{D}_{1+\alpha_0}$ for large enough $t$.

For all $0\leq s\leq q_{t+1}+q_t-1$, we set
$$I(F_{\alpha}^{\comp(q_{t+2}+q_{t+1}-s)}(\varphi_S^{\comp2}(P^t))):=
F_{\alpha}^{\comp(q_{t+2}+q_{t+1}-s)}(\varphi_S^{\comp2}(P^t))\cap\mathbb{S}^1.$$
Observe that
$${\rm II}:=\{I(F_{\alpha}^{\comp(q_{t+2}+q_{t+1}-s)}(\varphi_S^{\comp2}(P^t))):0\leq s\leq q_{t+1}+q_t-1\}$$
is a dynamical partition of $\mathbb{S}^1$. Then for all $0\leq s\leq q_{t+1}+q_t-1$,
each component $P$ of $F_{\alpha}^{\comp-s}((P^t)^{\comp})$, whose closure intersects $\mathbb{S}^1$, satisfies
$\overline{P}=F_{\alpha}^{\comp(q_{t+1}+q_t-s)}(\varphi_S^{\comp2}(P^t))$, and hence $P\subseteq W$;
for all $s\geq q_{t+1}+q_t$,
each component $P$ of $F_{\alpha}^{\comp-s}((P^t)^{\comp})$, whose closure intersects $\mathbb{S}^1$, satisfies that
$\overline{P}\cap\mathbb{S}^1$ is a proper subarc of some $I'\in{\rm II}$, and hence by (e), we have $P\subseteq W$.
\end{proof}

\section{Uniform bounds for area rate of measurable subsets in a square}
In this section we will give two main technique lemmas of this paper for area estimates,
which will be used in proofs of Theorems \ref{T1} and \ref{T1.2}.
The two main technique lemmas are as follows:
\begin{lemma}
\label{L1}Let $c>0$ and $0<\eta<1$ be two real numbers.
For any large enough positive integer $N$, we have:
for any square $S\subset\mathbb{C}$ with side length $l$ and any nonempty measurable subset $E$ of $S$,
if there exist $N$ mappings $r_1,r_2,\cdots,r_N$ from $E$ to $\mathbb{R}$
and $N$ mappings $y_n:E\to\mathbb{C}$, $1\leq n\leq N$
such that for all $x\in E$,
the following conditions hold:
\begin{itemize}
\item[{\rm(a)}] $0<r_1(x)\leq l\cdot\eta$,
\item[{\rm(b)}] $0<r_{n+1}(x)\leq\eta\cdot r_n(x)$ for all $1\leq n<N$,
\item[{\rm(c)}] $y_n(x)\in\mathbb{B}(x,c\cdot r_n(x))$ for all $1\leq n\leq N$,
\item[{\rm(d)}] $\mathbb{B}(y_n(x),r_n(x))\cap E=\emptyset$ for all $1\leq n\leq N$,
\end{itemize}
then
$${\rm area}(E)\leq\lambda^N\cdot{\rm area}(S),$$
where $0<\lambda<1$ is a constant only depending on $c$ and $\eta$.
\end{lemma}

\begin{lemma}
\label{l2}Let $S\subset\mathbb{C}$ be a square with side length $l$ and $E$ be a nonempty measurable subset of $S$.
Assume $0<\lambda<1$ be a real number.
If there exists a map $r$ from $E$ to $\mathbb{R}$ such that for all $x\in E$
\begin{itemize}
\item $0<r(x)<l$,
\item ${\rm area}(\mathbb{B}(x,r(x))\cap E)\leq\lambda\cdot{\rm area}(\mathbb{B}(x,r(x)))$,
\end{itemize}
then we have
$${\rm area}(E)\leq c\cdot\lambda\cdot{\rm area}(S),$$
where $c>0$ is a universal constant.
\end{lemma}

\noindent{\bf Proofs of two lemmas:}
\begin{proof}[The proof of Lemma \ref{L1}]
Let $n_0$ and $M$ (odd number) be two smallest positive integers such that
\begin{equation}
\label{F1}\eta^{n_0}<\frac{1}{M}<\frac{1}{5\sqrt{2}(c+2)}
\end{equation}
and let
$$N_0=\left\lfloor\frac{N}{n_0}\right\rfloor.$$
Next, we suppose $N\geq2n_0$ and hence $N_0\geq2$.
Then it follows from (a)(b)(c)(d) that
\begin{itemize}
\item[${\rm(\tilde{a})}$] $0<\tilde{r}_{n_0}(x)\leq\frac{l}{M}$,
\item[${\rm(\tilde{b})}$] $\tilde{r}_{n+1}(x)\leq\frac{1}{M}\cdot \tilde{r}_n(x)$ for all $1\leq n<N_0$,
\item[${\rm(\tilde{c})}$] $\tilde{y}_n(x)\in\mathbb{B}(x,c\cdot\tilde{r}_n(x))$ for all $1\leq n\leq N_0$,
\item[${\rm(\tilde{d})}$] $\mathbb{B}(\tilde{y}_n(x),\tilde{r}_n(x))\cap E=\emptyset$ for all $1\leq n\leq N_0$,
\end{itemize}
where $\tilde{r}_n:=r_{nn_0}$ and $\tilde{y}_n:=y_{nn_0}$ for all $1\leq n\leq N_0$.

We divide $S$ into $M^2$ small squares with side length $\frac{l}{M}$.
For all $1\leq i_1\leq M^2$, there exist unique two integers $p_1$ and $q_1$ such that $i_1=p_1M+q_1$ and $0<q_1\leq M$.
We use $S_{i_1}$ to denote the square in row $p_1+1$ and column $q_1$ in $S$.
By reduction, $S_{i_1i_2\cdots i_k}$ is defined as follows:
We divide $S_{i_1i_2\cdots i_{k-1}}$ into $M^2$ small squares with side length $\frac{l}{M^k}$.
Let $p_k$ and $q_k$ be the unique two integers such that $i_k=p_kM+q_k$ and $0<q_k\leq M$. Then
$S_{i_1i_2\cdots i_k}$ denotes the square in row $p_k+1$ and column $q_k$ in $S_{i_1i_2\cdots i_{k-1}}$.
For all $1\leq n\leq N_0$ and any square $S_{i_1i_2\cdots i_k}$ with $S_{i_1i_2\cdots i_k}\cap E\not=\emptyset$,
we set
$$R_{i_1i_2\cdots i_k}^n:=\sup_{x\in S_{i_1i_2\cdots i_k}\cap E}\tilde{r}_n(x).$$
Any square $S_{i_1i_2\cdots i_k}$ with $S_{i_1i_2\cdots i_k}\cap E\not=\emptyset$
is called \emph{an admissible square of generation $n\geq1$}
if
$$\sqrt{2}\frac{l}{M^{k+1}}<R_{i_1i_2\cdots i_k}^n\leq\sqrt{2}\frac{l}{M^k}\ {\rm and}\
R_{i_1i_2\cdots i_j}^n\leq\sqrt{2}\frac{l}{M^{j+1}},\ j=1,2,\cdots,k-1.$$
By ($\tilde{b}$), for any $1\leq m<n\leq N_0$, we have
\begin{equation}
\label{e3.2}R_{i_1i_2\cdots i_k}^n\leq \frac{1}{M^{n-m}}\cdot R_{i_1i_2\cdots i_k}^m.
\end{equation}
This implies that each admissible square corresponds to a unique generation.
Moreover, by definition it is easy to see that
for any two admissible squares  $S_{i_1i_2\cdots i_k}$ and $S_{i_1'i_2'\cdots i_{k'}'}$ of generation $n\geq1$,
one of the following two cases must occur:
\begin{itemize}
\item $S_{i_1i_2\cdots i_k}=S_{i_1'i_2'\cdots i_{k'}'}$;
\item $S_{i_1i_2\cdots i_k}\cap(S_{i_1'i_2'\cdots i_{k'}'})^{\circ}=\emptyset$.
\end{itemize}
For any admissible square $S_{i_1i_2\cdots i_k}$, we let $F_{i_1i_2\cdots i_k}$ denote a square
$S_{i_1'i_2'\cdots i_{k-1}'i_k'i_{k+1}'}$ such that
\begin{itemize}
\item[\rm{(a')}] $S_{i_1'i_2'\cdots i_{k-1}'i_k'i_{k+1}'}\cap E=\emptyset$,
\item[\rm(b')] the distance between $S_{i_1'i_2'\cdots i_{k-1}'i_k'i_{k+1}'}$ and $S_{i_1i_2\cdots i_k}$
is less than or equal to $\sqrt{2}(c+1)\frac{l}{M^k}$.
\end{itemize}
Such $F_{i_1i_2\cdots i_k}$ always exists.
In fact,
without loss of generality, we set $S_{i_1i_2\cdots i_k}$ is of generation $n\geq1$.
Since $S_{i_1i_2\cdots i_k}$ is an admissible square of generation $n$, we have
$\sqrt{2}\frac{l}{M^{k+1}}<R_{i_1i_2\cdots i_k}^n\leq\sqrt{2}\frac{l}{M^k}$ and hence
there exists $x\in S_{i_1i_2\cdots i_k}$ such that
$|\tilde{y}_n(x)-x|\leq c\cdot\tilde{r}_n(x)$, $\mathbb{B}(\tilde{y}_n(x),\tilde{r}_n(x))\cap E=\emptyset$
and $\sqrt{2}\frac{l}{M^{k+1}}<\tilde{r}_n(x)\leq\sqrt{2}\frac{l}{M^k}$.
Thus the ball $\mathbb{B}(\tilde{y}_n(x),\tilde{r}_n(x))$ contains at least a square of form
$S_{i_1'i_2'\cdots i_{k-1}'i_k'i_{k+1}'}$.
We take $F_{i_1i_2\cdots i_k}=S_{i_1'i_2'\cdots i_{k-1}'i_k'i_{k+1}'}$. Then (a') holds immediately.
Since the distance between $F_{i_1i_2\cdots i_k}$ and $S_{i_1i_2\cdots i_k}$
is less than or equal to $|\tilde{y}_n(x)-x|+\tilde{r}_n(x)\leq\sqrt{2}(c+1)\frac{l}{M^k}$,
(b') also holds. Thus $F_{i_1i_2\cdots i_k}$ always exists.

Let
$$S^n:=\cup\{{\rm all\ admissible\ squares\ of\ generation\ n}\},\ 1\leq n\leq N_0$$
and
$$F^n=F^{n1}\cup(F^{n2}\setminus F^{n3}),\ 1\leq n\leq N_0-1,$$
where
\begin{align*}
F^{n1}:=\cup\{&F_{i_1'i_2'\cdots i_{k'}'}:S_{i_1i_2\cdots i_k}\ {\rm is\ an\ admissible\ square\ of\ generation\ n}\
{\rm and}\\
&S_{i_1'i_2'\cdots i_{k'}'}(\subseteq S_{i_1i_2\cdots i_k\frac{M^2+1}{2}})\ {\rm is\ an\ admissible\ square\ of\ generation\ n+1}\},
\end{align*}
\begin{align*}
F^{n2}:=\cup\{S_{i_1i_2\cdots i_k\frac{M^2+1}{2}}:\ S_{i_1i_2\cdots i_k}\ {\rm is\ an\ admissible\ square\ of\ generation\ n}\}
\end{align*}
and
\begin{align*}
F^{n3}:=\cup\{&S_{i_1'i_2'\cdots i_{k'}'}:S_{i_1i_2\cdots i_k}\ {\rm is\ an\ admissible\ square\ of\ generation\ n}\
{\rm and}\\
&S_{i_1'i_2'\cdots i_{k'}'}(\subseteq S_{i_1i_2\cdots i_k\frac{M^2+1}{2}})\ {\rm is\ an\ admissible\ square\ of\ generation\ n+1}\}.
\end{align*}

About $S^n$, $F^n$ and $E$, we have the following properties:
\begin{itemize}
\item[(1)] $\{S^n\}_{n=1}^{N_0}$ is decreasing,
\item[(2)] for all $1\leq n\leq N_0$, $E\subseteq S^n$,
\item[(3)] for all $1\leq n\leq N_0-1$, $F^n\cap E=\emptyset$ and $F^n\subseteq(S^n)^{\circ}$,
\item[(4)] for all $x\in\cup_{n=1}^{N_0-1}F^n$, at most two of $F^1,F^2,\cdots,F^{N_0-1}$ contain $x$,
\item[(5)] for all $1\leq n\leq N_0-1$,
$$\zeta\cdot{\rm area}(S^n)\leq{\rm area}(F^n)\leq {\rm area}(S^n),$$
where $\zeta:=\frac{1}{2\pi\cdot M^2\cdot(\frac{\sqrt{2}}{2}+\sqrt{2}(c+2)M)^2}$.
\end{itemize}
The proof of these properties is postponed to Appendix A at the end of this paper.
Now we apply these properties to complete the proof of this lemma.

For any measurable subset $U$ of $S$, we set
$$U+2l:=\{x+2l:x\in U\}$$
and evidently,
$$U\cap(U+2l)=\emptyset.$$
We define
$$\tilde{S}=S\cup(S+2l),$$
$$\tilde{E}=E\cup(E+2l)$$
and for all $1\leq n\leq N_0$,
$$\tilde{S}^n=S^n\cup(S^n+2l).$$
Next, we define $\tilde{F}_n$ as follows:
firstly, define $\tilde{F}_1=F^1$ and then by reduction,
for any $2\leq n\leq N_0-1$, we define
$$\tilde{F}_n=F_n^{(1)}\cup F_n^{(2)},$$
where
$$F_n^{(1)}=F^n\setminus(\cup_{j=1}^{n-1}\tilde{F}_j)$$
and
$$F_n^{(2)}=(F^n\setminus F_n^{(1)})+2l.$$
By (1)(2)(3)(4)(5), we have the following corresponding properties:
\begin{itemize}
\item[(1')] $\{\tilde{S}^n\}_{n=1}^{N_0}$ is decreasing,
\item[(2')] for all $1\leq n\leq N_0$, $\tilde{E}\subseteq\tilde{S}^n$,
\item[(3')] for all $1\leq n\leq N_0-1$, $\tilde{F}_n\subseteq\tilde{S}^n$ and $\tilde{F}_n\cap\tilde{E}=\emptyset$,
\item[(4')] for all $1\leq i<j\leq N_0-1$, $\tilde{F}_i\cap\tilde{F}_j=\emptyset$,
\item[(5')] for all $1\leq n\leq N_0-1$,
$$\frac{\zeta}{2}\cdot{\rm area}(\tilde{S}^n)\leq{\rm area}(\tilde{F}_n)\leq {\rm area}(\tilde{S}^n).$$
\end{itemize}
It follows from (1'), (3'), (4') and (5') that
for all $2\leq n\leq N_0-1$, we have
\begin{align*}
{\rm area}(\tilde{S}^n\setminus(\cup_{j=1}^n\tilde{F}_j))&={\rm area}((\tilde{S}^n\setminus(\cup_{j=1}^{n-1}\tilde{F}_j))\setminus\tilde{F}_n)\\
&\leq(1-\frac{\zeta}{2})\cdot{\rm area}(\tilde{S}^n\setminus(\cup_{j=1}^{n-1}\tilde{F}_j))\\
&\leq(1-\frac{\zeta}{2})\cdot{\rm area}(\tilde{S}^{n-1}\setminus(\cup_{j=1}^{n-1}\tilde{F}_j)).
\end{align*}
Then
$${\rm area}(\tilde{S}^{N_0-1}\setminus(\cup_{j=1}^{N_0-1}\tilde{F}_j))\leq
(1-\frac{\zeta}{2})^{N_0-2}{\rm area}(\tilde{S}^1\setminus\tilde{F}_1)\leq
(1-\frac{\zeta}{2})^{N_0-1}\cdot{\rm area}(\tilde{S}).$$
By (2') and (3'), we have
$${\rm area}(\tilde{E})\leq{\rm area}(\tilde{S}^{N_0-1}\setminus(\cup_{j=1}^{N_0-1}\tilde{F}_j))$$
and hence
\begin{align*}
{\rm area}(\tilde{E})&\leq(1-\frac{\zeta}{2})^{N_0-1}\cdot{\rm area}(\tilde{S})\\
&=(1-\frac{\zeta}{2})^{\left\lfloor\frac{N}{n_0}\right\rfloor-1}\cdot{\rm area}(\tilde{S})\\
&\leq((1-\frac{\zeta}{2})^{\frac{1}{n_0}-\frac{2}{N}})^N\cdot{\rm area}(\tilde{S}).
\end{align*}
Since
$${\rm area}(\tilde{E})=2\cdot{\rm area}(E)\ {\rm and}\ {\rm area}(\tilde{S})=2\cdot{\rm area}(S),$$
we have
\begin{align*}
{\rm area}(E)\leq((1-\frac{\zeta}{2})^{\frac{1}{n_0}-\frac{2}{N}})^N\cdot{\rm area}(S)
\end{align*}
and hence the proof is completed by taking $\lambda=(1-\frac{\zeta}{2})^{\frac{1}{2n_0}}$.

\end{proof}

\begin{proof}[The proof of Lemma \ref{l2}]
We divide $S$ into $4$ small squares with side length $\frac{l}{2}$.
For all $1\leq i_1\leq 4$, there exist unique two integers $p_1$ and $q_1$ such that $i_1=2p_1+q_1$ and $0<q_1\leq2$.
We use $S_{i_1}$ to denote the square in row $p_1+1$ and column $q_1$ in $S$.
By reduction, $S_{i_1i_2\cdots i_k}$ is defined as follows:
We divide $S_{i_1i_2\cdots i_{k-1}}$ into $4$ small squares with side length $\frac{l}{2^k}$.
Let $p_k$ and $q_k$ be the unique two integers such that $i_k=2p_k+q_k$ and $0<q_k\leq2$. Then
$S_{i_1i_2\cdots i_k}$ denotes the square in row $p_k+1$ and column $q_k$ in $S_{i_1i_2\cdots i_{k-1}}$.
For any square $S_{i_1i_2\cdots i_k}$ with $S_{i_1i_2\cdots i_k}\cap E\not=\emptyset$, we set
$$R_{i_1i_2\cdots i_k}:=\sup_{x\in S_{i_1i_2\cdots i_k}\cap E}r(x).$$
Any square $S_{i_1i_2\cdots i_k}$ with $S_{i_1i_2\cdots i_k}\cap E\not=\emptyset$
is called an admissible square if
$$\sqrt{2}\frac{l}{2^k}<R_{i_1i_2\cdots i_k}\leq\sqrt{2}\frac{l}{2^{k-1}}\ {\rm and}\
R_{i_1i_2\cdots i_j}\leq\sqrt{2}\frac{l}{2^j},\ j=1,2,\cdots,k-1.$$
Let $AS$ denote the union of all admissible squares.
To complete the proof, it is sufficient to prove the following three properties:
\begin{itemize}
\item[(1)] $E\subseteq AS$;
\item[(2)] for any two admissible squares $S_{i_1i_2\cdots i_k}$ and $S_{i_1'i_2'\cdots i_{k'}'}$,
if
$$S_{i_1i_2\cdots i_k}\not=S_{i_1'i_2'\cdots i_{k'}'},$$
then
$$S_{i_1i_2\cdots i_k}\cap(S_{i_1'i_2'\cdots i_{k'}'})^{\circ}=\emptyset;$$
\item[(3)] there exists a universal constant $c>0$ such that for any admissible square $S_{i_1i_2\cdots i_k}$,
$${\rm area}(S_{i_1i_2\cdots i_k}\cap E)\leq c\cdot\lambda\cdot{\rm area}(S_{i_1i_2\cdots i_k}).$$
\end{itemize}
The proof of (1) is similar to that of (2) in the proof of Appendix $A$.
In fact, for any $x\in E$, we let $k$ be a positive integer such that $r(x)>\sqrt{2}\frac{l}{2^k}$
and $S_{i_1i_2\cdots i_k}$ be a square containing $x$.
Then $R_{i_1i_2\cdots i_k}>\sqrt{2}\frac{l}{2^k}$.
Next, we let $\tilde{k}$ be the smallest positive integer such that
\begin{equation}
\label{F5.3}R_{i_1i_2\cdots i_{\tilde{k}}}>\sqrt{2}\frac{l}{2^{\tilde{k}}}.
\end{equation}
If $\tilde{k}=1$, we have
$R_{i_1}>\sqrt{2}\frac{l}{2}$.
Again since $R_{i_1}\leq l<\sqrt{2}\cdot l$, we have that
$S_{i_1}$ is an admissible square and hence $x\in AS$.
If $\tilde{k}>1$, then
\begin{equation}
\label{F5.2}R_{i_1i_2\cdots i_j}\leq\sqrt{2}\frac{l}{2^j}
\end{equation}
for $j=1,2,\cdots \tilde{k}-1$.
In particular,
$R_{i_1i_2\cdots i_{\tilde{k}-1}}\leq\sqrt{2}\frac{l}{2^{\tilde{k}-1}}$.
Then we have
\begin{equation}
\label{F5.1}R_{i_1i_2\cdots i_{\tilde{k}}}\leq R_{i_1i_2\cdots i_{\tilde{k}-1}}\leq\sqrt{2}\frac{l}{2^{\tilde{k}-1}}.
\end{equation}
Combining (\ref{F5.3}), (\ref{F5.2}) and (\ref{F5.1}), we have that
$S_{i_1i_2\cdots i_{\tilde{k}}}$ is an admissible square and hence $x\in AS$.
Thus $E\subseteq AS$.

The proof of (2): for any two admissible squares $S_{i_1i_2\cdots i_k}$ and $S_{i_1'i_2'\cdots i_{k'}'}$,
if
$$S_{i_1i_2\cdots i_k}\not=S_{i_1'i_2'\cdots i_{k'}'},$$
then the following three cases may occur:
\begin{itemize}
\item $S_{i_1i_2\cdots i_k}\cap(S_{i_1'i_2'\cdots i_{k'}'})^{\circ}=\emptyset;$
\item $S_{i_1i_2\cdots i_k}\subseteq S_{i_1'i_2'\cdots i_{k'}'};$
\item $S_{i_1i_2\cdots i_k}\supseteq S_{i_1'i_2'\cdots i_{k'}'}.$
\end{itemize}
To complete the proof of (2), we only need to prove the last two cases can't occur.
Assume $S_{i_1i_2\cdots i_k}\subseteq S_{i_1'i_2'\cdots i_{k'}'}$.
Then $k'<k$, $i_1=i_1', i_2=i_2',\cdots, i_{k'}=i_{k'}'$.
Since $S_{i_1i_2\cdots i_k}$ is an admissible square, we have that
$$\sqrt{2}\frac{l}{2^k}<R_{i_1i_2\cdots i_k}\leq\sqrt{2}\frac{l}{2^{k-1}}\ {\rm and}\
R_{i_1i_2\cdots i_j}\leq\sqrt{2}\frac{l}{2^j},\ j=1,2,\cdots,k-1.$$
In particular,
$$R_{i_1i_2\cdots i_{k'}}\leq\sqrt{2}\frac{l}{2^{k'}}.$$
This implies $S_{i_1i_2\cdots i_{k'}}$ is not admissible, which contradicts the original condition.
Thus $S_{i_1i_2\cdots i_k}\subseteq S_{i_1'i_2'\cdots i_{k'}'}$ can't occur.
Similarly, one can prove that $S_{i_1i_2\cdots i_k}\supseteq S_{i_1'i_2'\cdots i_{k'}'}$ can't occur.

The proof of (3): for any admissible square $S_{i_1i_2\cdots i_k}$,
by definition we have
$$\sqrt{2}\frac{l}{2^k}<R_{i_1i_2\cdots i_k}\leq\sqrt{2}\frac{l}{2^{k-1}}.$$
Then there exists $x\in S_{i_1i_2\cdots i_k}$ such that
$$\sqrt{2}\frac{l}{2^k}<r(x)\leq\sqrt{2}\frac{l}{2^{k-1}}$$
and hence $S_{i_1i_2\cdots i_k}\subseteq\mathbb{B}(x,r(x))$.
Since
${\rm area}(\mathbb{B}(x,r(x))\cap E)\leq\lambda\cdot{\rm area}(\mathbb{B}(x,r(x)))$,
this implies
\begin{align*}
{\rm area}(S_{i_1i_2\cdots i_k}\cap E)&\leq
\lambda\cdot{\rm area}(\mathbb{B}(x,r(x)))\\
&\leq8\pi\lambda\cdot{\rm area}(S_{i_1i_2\cdots i_k}).
\end{align*}
Thus the proof is completed by taking $c=8\pi$.

\end{proof}

\section{David homeomorphisms}
Let $\mu=\mu(z)d\overline{z}/dz$ be a Beltrami differential in a planar domain $\Omega$, that is
a measurable (-1,1)-form with $|\mu(z)|<1$ almost everywhere in $\Omega$.
The Beltrami differential $\mu$ is called a David-Beltrami differential if
there exist constants $M>0, m>0$ and $0<\epsilon_0<1$ such that
\begin{equation}
\label{e3.1a}{\rm area}\{z\in\Omega: |\mu(z)|>1-\epsilon\}\leq M\cdot e^{-\frac{m}{\epsilon}}
\end{equation}
for all $0<\epsilon<\epsilon_0$.
By the classical Ahlfors-Bers-Morry theorem,
quasiconformal mappings arise as the solutions of
the Beltrami equation $\overline{\partial}\varphi=\mu\cdot\partial\varphi$ if $||\mu||_{\infty}<1$.
In \cite{Da} David proved the analogue of the Ahlfors-Bers-Morry theorem \cite{AB} holds
for the class of David-Beltrami differentials:

\begin{theorem}[David]
\label{TD1}Let $\Omega$ be a domain in $\mathbb{C}$ and $\mu$ be a David-Beltrami differential in $\Omega$.
Then there exists an orientation-preserving homeomorphism $\varphi:\Omega\to\Omega'$ in $W^{1,1}_{loc}(\Omega)$
which satisfies $\overline{\partial}\varphi= \mu\cdot\partial\varphi$ almost everywhere.
Moreover, $\varphi$ is unique up to postcomposition with a conformal map.
In other words, if $\psi:\Omega\to\Omega''$ is another homeomorphic
solution of the same Beltrami equation in $W^{1,1}_{loc}(\Omega)$,
then $\psi\comp\varphi^{\comp-1}:\Omega'\to\Omega''$ is a conformal map.
\end{theorem}

Solutions of the Beltrami equation given by this theorem are called David
homeomorphisms. If $\varphi:\Omega\to\Omega'$ is a David homeomorphism, then
$\partial\varphi\not=0$ almost everywhere in $\Omega$. Thus the complex dilatation of $\varphi$, defined by
the measurable (-1, 1)-form
$$\mu_{\varphi}:=\frac{\varphi_{\overline{z}}}{\varphi_z}\cdot\frac{d\overline{z}}{dz}$$
is a well-defined David-Beltrami differential and
$\varphi$ is exactly one solution of the equation $\overline{\partial}\varphi=\mu_{\varphi}\cdot\partial\varphi$.
As a generalization of classical quasiconformal maps,
David homeomorphisms enjoy some convenient properties of quasiconformal maps such as compactness (see \cite{T}).
However, they differ also in many respects.
A significant example is the fact that the inverse of a David homeomorphism is not necessarily David.
About David homeomorphisms, we also need the following result\cite{Da}:
\begin{theorem}[David]
\label{TD2}Let $\varphi:\hat{\mathbb{C}}\to\hat{\mathbb{C}}$ be a David homeomorphism fixing $0,1,\infty$
with $\mu_{\varphi}$ of form {\rm(\ref{e3.1a})}.
Then for all $r>0$ and any measurable set $E\subseteq\mathbb{D}_r,$ we have
$${\rm area}(\varphi(E))\leq C(r,M,m,\epsilon_0)\cdot\left(\log\left(1+\frac{1}{{\rm area}(E)}\right)\right)^{-mc},$$
where $c>0$ is universal and $C(r,M,m,\epsilon_0)$ depends only on $r,M,m,\epsilon_0$.
Moreover, for any $z_1,z_2\in\mathbb{D}_r$, we have
$$|\varphi(z_1)-\varphi(z_2)|\geq\frac{1}{C'(r,M,m,\epsilon_0)}\cdot e^{-\frac{m}{c'}\log^2(2+\frac{1}{|z_1-z_2|})},$$
where $c'>0$ is universal and $C'(r,M,m,\epsilon_0)>0$ depends only on $r,M,m,\epsilon_0$.
\end{theorem}
The following corollary follows easily from Theorems \ref{TD1} and \ref{TD2}:
\begin{corollary}
\label{c4.1}Let $\varphi:\mathbb{B}(1,1)\to\mathbb{B}(1,1)$ be a David homeomorphism fixing $1$
with $\mu_{\varphi}$ of form {\rm(\ref{e3.1a})}.
Then for any measurable set $E\subseteq\mathbb{B}(1,1),$ we have
$${\rm area}(\varphi(E))\leq C(M,m,\epsilon_0)\cdot\left(\log\left(1+\frac{1}{{\rm area}(E)}\right)\right)^{-mc},$$
where $c>0$ is universal and $C(M,m,\epsilon_0)$ depends only on $M,m,\epsilon_0$.
Moreover, for any $z_1,z_2\in\mathbb{B}(1,1)$, we have
$$|\varphi(z_1)-\varphi(z_2)|\geq\frac{1}{C'(M,m,\epsilon_0)}\cdot e^{-\frac{m}{c'}\log^2(2+\frac{1}{|z_1-z_2|})},$$
where $c'>0$ is universal and $C'(M,m,\epsilon_0)>0$ depends only on $M,m,\epsilon_0$.
\end{corollary}

\begin{proof}
We fix an $0<\epsilon<1$ so that the conformal modulus of $\varphi(\mathbb{B}(1,1)\setminus\overline{\mathbb{B}(1,\epsilon)})$
is greater than $1$.
We define
$$\tilde{\mu}_{\varphi}(z)=\left\{\begin{matrix}\mu_{\varphi}(z)&,&z\in\mathbb{B}(1,1)\\
0&,&z\in\partial\mathbb{B}(1,1)\\(\frac{z-1}{\overline{z-1}})^2\overline{\mu(\frac{1}{\overline{z-1}}+1)}&,
&z\in\mathbb{B}(1,\frac{1}{\epsilon})\setminus\overline{\mathbb{B}(1,1)}\\
0&,&z\in\mathbb{C}\setminus\mathbb{B}(1,\frac{1}{\epsilon})\end{matrix}\right.$$
Then by Theorem \ref{TD1},
there exists an orientation-preserving homeomorphism $\Phi:\mathbb{C}\to\mathbb{C}$ fixing $0, 1$ in $W^{1,1}_{loc}(\mathbb{C})$
which satisfies $\overline{\partial}\Phi=\tilde{\mu}_{\varphi}\cdot\partial\Phi$ almost everywhere.
Moreover, $\varphi\comp\Phi^{-1}$ is a conformal map from $\Phi(\mathbb{B}(1,1))$ to $\mathbb{B}(1,1)$ and hence
$\varphi\comp\Phi^{-1}$ can be extended to a homeomorphism from $\overline{\Phi(\mathbb{B}(1,1))}$ to $\overline{\mathbb{B}(1,1)}$.
Since $\varphi(z)=(\varphi\comp\Phi^{-1})\comp\Phi(z)$, $z\in\mathbb{B}(1,1)$,
$\varphi$ can be extended to a homeomorphism $\hat{\varphi}$ from $\overline{\mathbb{B}(1,1)}$ to $\overline{\mathbb{B}(1,1)}$.
We define
$$\tilde{\varphi}(z)=\left\{\begin{matrix}\hat{\varphi}(z)&,&z\in\overline{\mathbb{B}(1,1)}\\
\frac{1}{\overline{\varphi(\frac{1}{\overline{z-1}}+1)-1}}+1&,&z\in\mathbb{C}\setminus\overline{\mathbb{B}(1,1)}\end{matrix}\right.$$
It is easy to see that $\mu_{\tilde{\varphi}}(z)=\tilde{\mu}_{\phi}(z)\ a.e.$ in $\mathbb{B}(1,\frac{1}{\epsilon})$.
Then by Theorem \ref{TD1}, $\tilde{\varphi}\comp\Phi^{-1}$ is a conformal map on $\Phi(\mathbb{B}(1,\frac{1}{\epsilon}))$.
Observe that
$$\tilde{\varphi}\comp\Phi^{-1}(1)=1,\ |\tilde{\varphi}\comp\Phi^{-1}(0)-1|=1$$
and
$$\tilde{\varphi}\comp\Phi^{-1}(\Phi(\mathbb{B}(1,\frac{1}{\epsilon}))\setminus\overline{\Phi(\mathbb{B}(1,1))})=
\tilde{\varphi}(\mathbb{B}(1,\frac{1}{\epsilon})\setminus\overline{\mathbb{B}(1,1)})=
\frac{1}{\overline{\varphi(\mathbb{B}(1,1)\setminus\overline{\mathbb{B}(1,\epsilon)})-1}}+1$$
has a conformal modulus greater than $1$.
Then by Koebe distortion theorem, $$\sup_{z\in\Phi(\mathbb{B}(1,1))}|(\tilde{\varphi}\comp\Phi^{-1})'(z)|\asymp\inf_{z\in\Phi(\mathbb{B}(1,1))}|(\tilde{\varphi}\comp\Phi^{-1})'(z)|\asymp1.$$
Then any two points $z_1,z_2\in\Phi(\mathbb{B}(1,1))$ and any measurable set $E\subset\Phi(\mathbb{B}(1,1))$, we have
$$|\tilde{\varphi}\comp\Phi^{-1}(z_1)-\tilde{\varphi}\comp\Phi^{-1}(z_2)|\asymp|z_1-z_2|\ {\rm and}\
{\rm area}(E)\asymp{\rm area}(\tilde{\varphi}\comp\Phi^{-1}(E)).$$
Thus by Theorem \ref{TD2}, the restriction of
$\tilde{\varphi}(z)=(\tilde{\varphi}\comp\Phi^{-1})\comp\Phi(z)$, $z\in\mathbb{B}(1,\frac{1}{\epsilon})$,
to $\mathbb{B}(1,1)$, that is $\varphi$, has the conclusion in Corollary \ref{c4.1}.
\end{proof}

For any David Beltrami differential $\mu$ with the form (\ref{e3.1a}),
if we set $K:=\frac{2-\epsilon}{\epsilon}$, then
(\ref{e3.1a}) can be equivalently written as
$${\rm area}\{z\in\Omega:\frac{1+|\mu|}{1-|\mu|}>K\}\leq K_1e^{-\frac{K}{K_2}}\ {\rm for\ all}\ K>K_0$$
for $K_2=\frac{2}{m}, K_1=Me^{-\frac{m}{2}}$ and $K_0=\frac{2-\epsilon_0}{\epsilon_0}$.
The set of all such David-Beltrami differentials is denoted by
$\mathcal{F}(\Omega)(K_0,K_1,K_2)$. In particular,
if $\Omega=\{z\in\mathbb{C}:|z-1|<1\}$, then $\mathcal{F}(\Omega)(K_0,K_1,K_2)$ will be simply written
as $\mathcal{F}(K_0,K_1,K_2)$.

\section{The canonical $F_{\alpha}$-invariant David-Beltrami differential}
This section is devoted to recalling the canonical $F_{\alpha}$-invariant David-Beltrami differential constructed by Petersen and Zakeri for
$\alpha=[a_1,a_2,\cdots]$ under Petersen and Zakeri's condition. More details see Petersen and Zakeri's paper \cite{PZ04}.
\subsection{Two imbedded graphs in the upper half-plane}
Given any irrational number $\alpha=[a_1,a_2,\cdots]$ and
the corresponding rational approximation $\{\frac{p_n}{q_n}\}_{n\geq0}$.
For any two positive integers $j\not=k$, we denote by the interval $]x_j,x_k[_{\mathbb{S}^1}^*$
the component of $\mathbb{S}^1\setminus\{x_j,x_k\}$ not containing $1$.
Given $n\geq0$, we consider the set
$Q_n:=\{x_j:0\leq j<q_n\}$ on $\mathbb{S}^1$.
One can see that
the complement of $Q_n$ in $\mathbb{S}^1$ is the union of these disjoint pairwise intervals:
$$\mathbb{S}^1\setminus Q_n=\left(\cup_{0\leq j<q_n-q_{n-1}}]x_{j+q_{n-1}},x_j[_{\mathbb{S}^1}^*\right)\cup
\left(\cup_{0\leq j<q_{n-1}}]x_j,x_{j+q_n-q_{n-1}}[_{\mathbb{S}^1}^*\right).$$
Moreover, $x_j$ and $x_k$, with $j<k$, are adjacent in both $Q_n$ and $Q_{n+1}$
if and only if $a_{n+1}=1,k=j+q_{n-1},$ and $0\leq j<q_n-q_{n-1}$.
We denote by $\tilde{Q}_n$ the lift of $Q_n$ to the real line $\mathbb{R}$ under the map $\exp(\cdot):=e^{2\pi i(\cdot)}:\mathbb{R}\to\mathbb{S}^1$.
Then $\tilde{Q}_n$ is a translation-invariant set $\{z\in\mathbb{R}:e^{2\pi iz}\in Q_n\}$.
Moreover, if $n\geq1$, then
$\tilde{x},\tilde{y}\in\tilde{Q}_n$ are adjacent if and only if $|\tilde{x}-\tilde{y}|<1$ and the images
$e^{2\pi i\tilde{x}}, e^{2\pi i\tilde{y}}$ are adjacent in $Q_n$;
if $n=0$, then $\tilde{x},\tilde{y}\in\tilde{Q}_0$ are adjacent if and only if $|\tilde{x}-\tilde{y}|=1$.
For any $\tilde{x}\in\tilde{Q}_n$, define
$$M_n(\tilde{x}):=\frac{x_r-x_l}{2},$$
where $x_r$ and $x_l$ are the points in $\tilde{Q}_n$
immediately adjacent to the right and left of $\tilde{x}$.
Evidently, $M_n(\tilde{x})>M_{n+1}(\tilde{x})$ unless $x_r$ and $x_l$ are adjacent to $\tilde{x}$ in $\tilde{Q}_{n+1}$ also,
in which case $M_n(\tilde{x})=M_{n+1}(\tilde{x})$.
Now \emph{an imbedded graph $\Gamma$} induced by $\{\tilde{Q}_n\}_{n\geq0}$ in the upper half-plane $\mathbb{H}$ is defined as follows:
\begin{itemize}
\item The vertices of $\Gamma$ are the points $\{z_n(\tilde{x})=\tilde{x}+i M_n(\tilde{x}):n\geq0\ {\rm and}\ \tilde{x}\in\tilde{Q}_n\}$
\item The edges of $\Gamma$ are the vertical segments
$$\{[z_n(\tilde{x}),z_{n+1}(\tilde{x})]:n\geq0\ {\rm and}\ \tilde{x}\in\tilde{Q}_n\ {\rm with}\ M_n(\tilde{x})\not=M_{n+1}(\tilde{x})\}$$
and the non-vertical segments
$$\{[z_n(\tilde{x}),z_n(\tilde{y})]:n\geq0\ {\rm and}\ \tilde{x},\tilde{y}\ {\rm are\ adjacent\ in}\ \tilde{Q}_n\}.$$
\end{itemize}
Note that $z_n(x)=z_{n+1}(x)$ if and only if $M_n(x)=M_{n+1}(x)$, in which case the corresponding vertex of $\Gamma$ is
doubly labelled. Moreover, $z_0(\tilde{x})=\tilde{x}+i$ for all $\tilde{x}\in\tilde{Q}_0=\mathbb{Z}$.

By a cell of $\Gamma$ one means the closure of any bounded connected component of $\mathbb{H}\setminus\Gamma$.
Any cell $\gamma$ of $\Gamma$ is uniquely determined by a pair of adjacent points $x<y$
in $\tilde{Q}_n$ with the property that either $M_n(x)\not=M_{n+1}(x)$ or $M_n(y)\not=M_{n+1}(y)$.
In this case, we say that $\gamma$ is an $n$-cell.
The top of the $n$-cell $\gamma$ is formed by the non-vertical edge $[z_n(x), z_n(y)]$ while
its bottom is formed by the union of non-vertical edges.
$$[z_{n+1}(t_0),z_{n+1}(t_1)]\cup[z_{n+1}(t_1),z_{n+1}(t_2)]\cup\cdots\cup[z_{n+1}(t_{k-1}),z_{n+1}(t_k)],$$
where the points $x=t_0<t_1<\cdots<t_k=y$ form the intersection $[x,y]\cap\tilde{Q}_{n+1}$.
If $k=1$ so that $x, y$ are also adjacent in $\tilde{Q}_{n+1}$,
then $\gamma$ is either a triangle or a trapezoid.
Otherwise $k\geq2$, $\gamma$ is a ($k+3$)-gon, where $k$ is either
$a_{n+1}$ or $a_{n+1}+1$. About these cells, the \'Swiatek-Herman real a priori bounds imply the following:
\begin{itemize}
\item[(\#1)] Fixing any integer $n\geq0$, the union of all the $m$-cells of $\Gamma$
for all $m\geq n$ is contained in a horizontal strip $\{z\in\mathbb{H}: 0\leq{\rm Im}(z)\leq l\}$
whose height satisfies an asymptotically universal bound $l\preccurlyeq\sigma^n$,
where $0<\sigma<1$ is a universal constant;
\item[(\#2)] The cells of $\Gamma$ have ``bounded geometry'' in the following
sense: There is a constant $c>1$ such that the top, bottom, and sides of
any $n$-cell $\gamma$ of $\Gamma$ have lengths comparable up to $c$. Moreover, the slopes of
non-vertical edges of $\gamma$ are bounded by an asymptotically
universal bound.
\end{itemize}

In a similar fashion, one can construct the above objects for the rigid
rotation $R_{\alpha}$, for which these similar but ``primed'' notations are chosen. Thus, the corresponding notations
$x_j'$, $Q_n'$, $\tilde{Q}'_n$, $M_n'(\cdot)$, $z_n'(\cdot)$, $\Gamma'$ and $\gamma'$ are well defined.
About these cells $\gamma'$, the following holds:
\begin{itemize}
\item The cells of $\Gamma'$ have ``bounded geometry'' in the following
sense: There is a universal constant $c>1$ such that the top, bottom, and sides
of any cell $\gamma'$ of $\Gamma'$ have lengths comparable up to $c$. Moreover, the slopes of
non-vertical edges of $\gamma$ are bounded by $\frac{1}{2}$.
\end{itemize}

Let $\tilde{x}_{-1}$ be a preimage of $x_{-1}$ under $\exp$.
For any $a,b\in\mathbb{C}$, by $[a,b]$ we denote the segment with ends $a,b$ in $\mathbb{C}$.
Set
$$E_n^+:=[x_{q_{n-1}-1},z_n(x_{q_{n-1}-1})]\cup[x_{q_n-1},z_n(x_{q_n-1})]\cup[z_n(x_{q_{n-1}-1}),z_n(x_{q_n-1})]$$
and
$$E_n^-:=[x_{q_{n-1}},z_n(x_{q_{n-1}})]\cup[x_{q_n-q_{n-1}},z_n(x_{q_n-q_{n-1}})]
\cup[z_n(x_{q_{n-1}}),z_n(x_0)]\cup[z_n(x_{q_n-q_{n-1}}),z_n(x_0)].$$
Define
$$r_n^+:=\sup\{r:\mathbb{B}(\tilde{x}_{-1},r)\cap E_n^+=\emptyset\}$$
and
$$\tilde{r}_n^+:=\frac{r_n^++r_{n+1}^+}{2};$$
define
$$r_n^-:=\sup\{r:\mathbb{B}(0,r)\cap E_n^-=\emptyset\}$$
and
$$\tilde{r}_n^-:=\frac{r_n^-+r_{n+1}^-}{2}.$$
We denote by $g_0$ the branch of $F_{\alpha}^{\comp-1}$ from $\overline{\mathbb{D}}$ to $\overline{U_0}$.
For every $n\geq1$,
we define
$$l_n^+:=\inf\left\{|g_0\comp\exp(\tilde{z})-1|:\tilde{z}=\tilde{r}_n^+\cdot e^{2\pi i\theta}+\tilde{x}_{-1},\ 0\leq\theta\leq\pi\right\}$$
and
$$l_n^-:=\inf\left\{|\exp(\tilde{z})-1|:\tilde{z}=\tilde{r}_n^-\cdot e^{2\pi i\theta},\ 0\leq\theta\leq\pi\right\}.$$
By the \'Swiatek-Herman real a priori bounds and above definitions, we have the following properties:
\begin{itemize}
\item Both $\{l_n^+\}$ and $\{l_n^-\}$ are quasi-log-arithmetic sequences;
\item Images of all $m$-cells ($m<n$) of $\Gamma$ under $\exp$ don't intersect with the ball $\mathbb{B}(1,l_n^-)$;
\item Images of all $m$-cells ($m<n$) of $\Gamma$ under $g_0\comp\exp$ don't intersect with the ball $\mathbb{B}(1,l_n^+)$.
\end{itemize}

\subsection{The construction of the canonical David-Beltrami differential\label{ss5.2}}
Recall that by Yoccoz's theorem,
there exists a unique homeomorphism $h:\mathbb{S}^1\to\mathbb{S}^1$ with $h(1)=1$ such that
$h\comp f_{\alpha}|_{\mathbb{S}^1}=R_{\alpha}\comp h$.
Let $\tilde{h}:\mathbb{R}\to\mathbb{R}$ be its lift with
$\tilde{h}(0)=0$. The lift $\tilde{h}$ fixes all integer points and $\tilde{h}(\tilde{Q}_n)=\tilde{Q}_n'$ for all $n\geq0$.
Then one can extend $\tilde{h}$ to a homeomorphism $\tilde{H}$ between the imbedded graphs $\Gamma$
and $\Gamma'$ by mapping each vertex of $\Gamma$ to the corresponding vertex of $\Gamma'$ and each
edge of $\Gamma$ affinely to the corresponding edge of $\Gamma'$.
Note that for each $n$-cell $\gamma$ of $\Gamma$, the boundary $\partial\gamma$ is mapped by $\tilde{H}$
homeomorphically and edgewise affinely onto the boundary $\partial\gamma'$ of a unique
$n$-cell $\gamma'$ of $\Gamma'$. Moreover,
$\tilde{H}$ defined in this way is the identity on the horizontal
line $\mathbb{R}+i$ so that we can define $\tilde{H}(z)=z$ for all $z\in\mathbb{H}$ with ${\rm Im}(z)\geq1$.
Following [Appendix, \cite{PZ04}],
there exists a homeomorphism extension $\hat{\mathbb{H}}:\mathbb{H}\to\mathbb{H}$ such that
the restriction $\hat{H}|_{\gamma}:\gamma\to\gamma'$ is a quasiconformal homeomorphism whose
dilatation is $\mathcal{O}(1+(\log a_{n+1})^2)$. Further,
let $H:\overline{\mathbb{D}}\to\overline{\mathbb{D}}$ be the induced homeomorphism under $z\mapsto e^{2\pi iz}$ and
if the rotation number $\alpha$ satisfies Petersen and Zakeri's condition,
then the extension $H:\overline{\mathbb{D}}\to\overline{\mathbb{D}}$ satisfies
\begin{equation}
\label{e3.1}{\rm area}\left\{z\in\mathbb{D}:\left|\frac{\overline{\partial}H(z)}{\partial H(z)}\right|>1-\epsilon\right\}
\leq K_1\cdot e^{-\frac{c}{\epsilon}}\ {\rm for}\ {\rm all}\ 0<\epsilon<\epsilon_0.
\end{equation}
Here $K_1>0$ is a universal constant, while in general the constant $c>0$
depends on
$\limsup_{n\to\infty}\frac{\log a_n}{\sqrt{n}}$
and the constant $0<\epsilon_0<1$ depends on $\alpha$.

Now we define a Beltrami differential $\mu_0$ as
$$\mu_0(z)=\left\{\begin{matrix}\frac{\overline{\partial}H(z)}{\partial H(z)}\frac{d\overline{z}}{dz},&z\in\mathbb{D}\\
(F_{\alpha}^{\comp n})^*(\frac{\overline{\partial}H}{\partial H}\frac{d\overline{z}}{dz})(z),&z\in F_{\alpha}^{\comp -n}(\mathbb{D}),n\geq1\\
0,&{\rm else}\end{matrix}\right..$$
In \cite{PZ04} Petersen and Zakeri proved the following theorem:
\begin{theorem}[Petersen and Zakeri]
\label{T3.1}There exist a universal constant $0<\beta<1$ and a constant $C_1>0$ {\rm(}depending on $\alpha${\rm)} such that
for every measurable set $E\subseteq\mathbb{D}$
$$\nu(E)\leq C_1({\rm area}(E))^{\beta},$$
here $\nu$ is a measure supported on $\mathbb{D}$ and satisfies that
for any measurable set $E\subseteq\mathbb{D}$,
$$\nu(E)={\rm area}(E)+\sum_g{\rm area}(g(E)),$$
where the summation is over all the univalent branches $g=F_{\alpha}^{\comp -k}$, $k\geq1$ on $\mathbb{D}$.
\end{theorem}
\noindent It follows from (\ref{e3.1}) and Theorem \ref{T3.1} that
if the rotation number $\alpha$ satisfies Petersen and Zakeri's condition, then
$\mu_0$ is a David-Beltrami differential. In this case,
there exists a David homeomorphism $\phi_{\alpha}$ such that
$$\mu_{\phi_{\alpha}}=\mu_0\ {\rm and}\ \phi_{\alpha}\comp F_{\alpha}\comp\phi_{\alpha}^{\comp -1}(z)=P_{\alpha_{\alpha}}(z).$$
We call such $\mu_0$ \emph{a canonical $F_{\alpha_{\alpha}}$-invariant David-Beltrami differential} and
$\phi_{\alpha}$ is the corresponding coordinate map.

\section{The canonical David-Beltrami differential at a small scale\label{s5}}
This section is devoted to exploring the canonical David-Beltrami differential at a small scale near the critical point.
We assume $\alpha\in\mathcal{E}_0$
and let $\mu_0$ be a canonical $F_{\alpha}$-invariant David-Beltrami differential.
For all $r>0$, we let $\mu_r:=\mu_r(z)\frac{d\overline{z}}{dz}$ be the restriction of $\mu_0$ to $\mathbb{B}(1,r)$.
We define
$$\hat{\mu}_r(z):=\mu_r(r(z-1)+1)\ {\rm and}\ \hat{\mu}_r:=\hat{\mu}_r(z)\frac{d\overline{z}}{dz}$$
on $\mathbb{B}(1,1)$.
The main result of this section is the following proposition:
\begin{proposition}
\label{p1}There exists a sequence $\{r_n\}_{n=1}^{\infty}$ of positive real numbers with $\lim_{n\to\infty}r_n=0$
and three positive real numbers $K_0,K_1,K_2$ such that for all large enough $n$, $\hat{\mu}_{r_n}\in\mathcal{F}(K_0,K_1,K_2)$.
\end{proposition}

\begin{proof}
Since $\alpha=[b_1,b_2,\cdots]\in\mathcal{E}_0$, by definition
there exist $\theta=[a_1,a_2,\cdots]\in\mathcal{E}$, a positive integer $M$ and
two sequences of positive integers $\{s_j\}_{j=1}^{\infty}$ and $\{t_j\}_{j=1}^{\infty}$ such that
\begin{itemize}
\item for all $j\geq1$, $s_j<t_j<s_{j+1}$ and $t_j-s_j>Cs_j$, where $C$ will be determined in this proof;
\item for all $1\leq k\leq s_1$, $b_k\leq a_k$;
\item for all $j\geq1$, $$b_k\left\{\begin{matrix}\leq&M,&s_j<k\leq t_j\\
\leq&a_{k-t_j},&t_j<k\leq s_{j+1}.\end{matrix}\right.$$
\end{itemize}
Since $\theta\in\mathcal{E}$, there exists $C_1>0$ such that
$\log a_n\leq C_1\sqrt{n}$ for all $n\geq1$.
For all $j\geq1$, we define $i_j:=\left\lfloor c_1s_j\right\rfloor$ ($c_1>1$ will be given later) and
$r_j:=|x_{q_{i_j}}-1|$.
By the \'Swiatek-Herman real a priori bounds,
for all $j\geq1$, $r_j\geq C_2\lambda^{i_j}$, where $0<\lambda<1$ is a universal constant and
$C_2$ is a positive constant.

\vspace{0.2cm}
\noindent\emph{Step 1: Fix the universal constant $C$}

\vspace{0.2cm}
\noindent Let
$R_n=\cup_{m\geq n}\{\exp(\gamma):\gamma\ {\rm is\ an\ m-cell\ of\ \Gamma}\}$.
Then by (\#1) we have ${\rm area}(R_n)\preccurlyeq\sigma^n$,
and hence by Theorem \ref{T3.1} we have that for all $n\geq1$, $\nu(R_n)\leq C_3(\sigma^{\beta})^n$,
where $C_3$ is a positive constant. Then for all $n\geq t_j$,
\begin{align*}
\frac{\nu(R_n)}{{\rm area}(\mathbb{B}(1,r_j))}&\leq C_3\frac{(\sigma^{\beta})^n}{\pi r_j^2}\\
&\leq\frac{C_3}{C_2^2}\frac{(\sigma^{\beta})^{t_j}}{\pi\lambda^{2i_j}}\cdot(\sigma^{\beta})^{n-t_j}\\
&\leq\frac{C_3}{C_2^2}\frac{(\sigma^{\beta})^{(1+C)s_j}}{\pi\lambda^{2(c_1s_j+1)}}\cdot(\sigma^{\beta})^{n-t_j}\\
&\leq\frac{C_3}{C_2^2\pi\lambda^2}\cdot\left(\frac{\sigma^{\beta(1+C)}}{\lambda^{2c_1}}\right)^{s_j}\cdot(\sigma^{\beta})^{n-t_j}.
\end{align*}
We take $C$ large enough so that
\begin{equation}
C>\frac{2c_1\ln\lambda}{\beta\ln\sigma}-1,
\end{equation}
that is $\frac{\sigma^{\beta(1+C)}}{\lambda^{2c_1}}<1$.
Then for large enough $j$ and all $n\geq t_j$,
$$\nu(R_n)\leq\frac{1}{\pi\lambda^2}\cdot(\sigma^{\beta})^{n-t_j}\cdot{\rm area}(\mathbb{B}(1,r_j)).$$

\vspace{0.2cm}
\noindent\emph{Step 2: Fix the universal constant $c_1$}

\vspace{0.2cm}
\noindent By Proposition \ref{pr2.1},
there exists a sequence $\{z_n\}_{n=1}^{+\infty}\subset\overline{U_0}$ with a quasi-log-arithmetic sequence
$\{l_n:=|z_n-1|\}_{n=1}^{+\infty}$ such that
for any $N>0$ and $n>0$, if $z\in S_{l_{n+N}}$ with $\Lambda_{\alpha_0}(z)>l_n$, then
there exists a nonnegative integer $m$ such that
$$\{F_{\alpha}^{\comp s}(z)\}_{s=0}^m\subseteq\mathbb{D}_{1+\alpha_0}\ {\rm and}\
d_{\mathbb{C}\setminus\overline{\mathbb{D}}}(F_{\alpha}^{\comp m}(z),z_n)\preccurlyeq1.$$
By the \'Swiatek-Herman real a priori bounds, $\{|x_{q_n}-1|\}_{n\gg1}$ is a quasi-log-arithmetic sequence.
For large enough $j$, there exists $n_j$ such that
$$\mathbb{B}(1,r_j)\subseteq S_{l_{n_j}}\ {\rm and}\ r_j\asymp l_{n_j},$$
and hence for large enough $j$,
$$n_j\asymp-\log l_{n_j}\asymp-\log r_j=-\log|x_{q_{i_j}}-1|\asymp i_j.$$
This implies
\begin{equation}
\label{e4.1}n_j\asymp c_1s_j.
\end{equation}
For large enough $j$, there exists $n_j'$ such that
\begin{equation}
\label{e4.5}l_{n_j'}<{\rm diam}(S_{l_{n_j'}})<l_{s_j}^+\ {\rm and}\ l_{n_j'}\asymp l_{s_j}^+.
\end{equation}
This implies that for large enough $j$,
\begin{equation}
\label{e4.2}n_j'\asymp s_j.
\end{equation}
Let
$$T_j:=\{t:n_j'\leq t\leq n_j-1,s_j\leq t\leq i_j\}.$$
We take $c_1$ large enough so that
\begin{itemize}
\item for large enough $j$, we have $r_j<l_{s_j}^-$;
\item for large enough $j$, we have $\#|T_j|\asymp s_j$. (Thanks to (\ref{e4.1}) and (\ref{e4.2}))
\end{itemize}

\vspace{0.2cm}
\noindent\emph{Step 3: A sequence of hyperbolic balls contained in $U_0$}

\vspace{0.2cm}
\noindent For large enough $j$, we fix a point $z\in S_{l_{n_j}}$ with $\Lambda_{\alpha_0}(z)>l_{s_j}^+$,
then by (\ref{e4.5}), we have
$\Lambda_{\alpha_0}(z)>l_t$ for all $t$ with $n_j'\leq t\leq n_j-1$,
and hence for all $n_j'\leq t\leq n_j-1$,
there exists a nonnegative integer $m_t$ such that
\begin{equation}
\label{e4.7}
\{F_{\alpha}^{\comp s}(z)\}_{s=0}^{m_t}\subseteq\mathbb{D}_{1+\alpha_0}\ {\rm and}\
d_{\mathbb{C}\setminus\overline{\mathbb{D}}}(F_{\alpha}^{\comp m_t}(z),z_t)\preccurlyeq1.
\end{equation}
For large enough $j$, we can choose a sequence $\{\hat{\mathbb{B}}_t\}_{n_j'}^{n_j-1}$ of hyperbolic balls
with the same size with respect to the hyperbolic metric
on $\mathbb{C}\setminus\overline{\mathbb{D}}$ such that for all $n_j'\leq t\leq n_j-1$,
\begin{equation}
\label{e7.9}\hat{\mathbb{B}}_t\subseteq U_0\cap\Lambda_{\alpha_0}^{\comp-1}(]0,l_{s_j}^+[),\
d_{\mathbb{C}\setminus\overline{\mathbb{D}}}(\hat{\mathbb{B}}_t,z_t)\preccurlyeq1\ {\rm and}\
{\rm diam}_{\mathbb{C}\setminus\overline{\mathbb{D}}}(\hat{\mathbb{B}}_t)\asymp1.
\end{equation}
Let $\mathbb{B}_t$ be a hyperbolic ball with the same center and the half size as $\hat{\mathbb{B}}_t$ for all $n_j'\leq t\leq n_j-1$.
Thus by (\ref{e4.7}) and (\ref{e7.9}), for all $n_j'\leq t\leq n_j-1$, we have
$$d_{\mathbb{C}\setminus\overline{\mathbb{D}}}(F_{\alpha}^{\comp m_t}(z),\mathbb{B}_t)\preccurlyeq1.$$
In particular, for all $t\in T_j$,
\begin{equation}
\label{e4.9}d_{\mathbb{C}\setminus\overline{\mathbb{D}}}(F_{\alpha}^{\comp m_t}(z),\mathbb{B}_t)\preccurlyeq1.
\end{equation}

\vspace{0.2cm}
\noindent\emph{Step 4: Pull back the sequence of hyperbolic balls}

\vspace{0.2cm}
\noindent We need the following lemma to estimate hyperbolic distances. (see [Proposition 4.9, \cite{McMu}] and [Corollary 2.29, \cite{McM}] for the proof)
\begin{lemma}
\label{l6}Let $f: D\subset\mathbb{C}\setminus\overline{\mathbb{D}}$ be an inclusion and
$||f'||$ denote the norm of the derivative with respect to two hyperbolic metrics.
Then for any $x,x_1,x_2\in D$ with $\tilde{s}=d_{D}(x_1,x_2)$ is less than one half of injective radius of $x_1$ in $D$,
we have
$$||f'(x)||<K(s)<1\ {\rm and}\ \frac{1}{\tilde{K}(\tilde{s})}\leq\frac{||f'(x_1)||}{||f'(x_2)||}\leq\tilde{K}(\tilde{s})$$
where $s=d_{\mathbb{C}\setminus\overline{\mathbb{D}}}(x,(\mathbb{C}\setminus\overline{\mathbb{D}})\setminus D)$ and
$K(s), \tilde{K}(\tilde{s})$ are two positive functions.
\end{lemma}

From now on, we assume that $j$ is large enough.
We set
$$M_j:=\{m_t:t\in T_j\}=\{w_1,w_2,\cdots,w_{\#|M_j|}\}$$
with
$$w_1<w_2<\cdots<w_{\#|M_j|}.$$
Since $\#|T_j|\asymp s_j$ and (\ref{e4.7}), we have
\begin{equation}
\label{e6.8}\#|M_j|\asymp s_j.
\end{equation}
Fixing $j_0\geq2$,
we set
$$M^{j_0}_{j}:=\{w_{j_0},w_{2j_0},\cdots,w_{\left\lfloor\frac{\#|M_j|}{j_0}\right\rfloor j_0}\}.$$
For any $1\leq t\leq \left\lfloor\frac{\#|M_j|}{j_0}\right\rfloor$,
we set $m_{j_t}:=w_{tj_0}$ and hence by (\ref{e4.9}) we have
\begin{equation*}
d_{\mathbb{C}\setminus\overline{\mathbb{D}}}(F_{\alpha}^{\comp w_{tj_0}}(z),\mathbb{B}_{j_t})\preccurlyeq1.
\end{equation*}
Let $L_{j_t}\subset\mathbb{C}\setminus\overline{\mathbb{D}}$ be a curve of the length
$l_{\mathbb{C}\setminus\overline{\mathbb{D}}}(L_{j_t})\preccurlyeq1$
connecting $F_{\alpha}^{\comp w_{tj_0}}(z)$ and $\mathbb{B}_{j_t}$.
For any $1\leq t\leq \left\lfloor\frac{\#|M_j|}{j_0}\right\rfloor$ and $1\leq w\leq w_{tj_0}$,
we consider the pullback of $\{F_{\alpha}^{\comp w_{tj_0}}(z)\}\cup L_{j_t}\cup\mathbb{B}_{j_t}$
under $F_{\alpha}^{\comp w}$ along $F_{\alpha}^{\comp w_{tj_0}}(z)$, $F_{\alpha}^{\comp(w_{tj_0}-1)}(z)$,
$\cdots$, $F_{\alpha}^{\comp(w_{tj_0}-w)}(z)$.
By $\mathbb{B}_{j_t}^{-w}$ and $L_{j_t}^{-w}$ we denote pre-images of $\mathbb{B}_{j_t}$ and $L_{j_t}$ in the pullback respectively.
It follows from Lemma \ref{l6} and (\ref{e4.7}) that there exists a universal constant $\eta>1$ such that
for any $1\leq t\leq \left\lfloor\frac{\#|M_j|}{j_0}\right\rfloor$,
\begin{equation}
\label{e7.12}\eta^{j_0}\cdot\rho_{\mathbb{C}\setminus\overline{\mathbb{D}}}\left(F_{\alpha}^{\comp w_{(t-1)j_0}}(z)\right)\leq
\rho_{\mathbb{C}\setminus\overline{\mathbb{D}}}\left(F_{\alpha}^{\comp w_{tj_0}}(z)\right)
\cdot\left|\left(F_{\alpha}^{\comp(w_{tj_0}-w_{(t-1)j_0})}\right)'(z)\right|,
\end{equation}
where $w_0=0$ and $\rho_{\mathbb{C}\setminus\overline{\mathbb{D}}}$ is the hyperbolic metric density on $\mathbb{C}\setminus\overline{\mathbb{D}}$.
Next, we will prove that for all $1\leq t\leq \left\lfloor\frac{\#|M_j|}{j_0}\right\rfloor-1$,
$$\eta^{j_0}\cdot{\rm diam}_{\mathbb{C}\setminus\overline{\mathbb{D}}}(\mathbb{B}_{j_{t+1}}^{-w_{(t+1)j_0}})\preccurlyeq{\rm diam}_{\mathbb{C}\setminus\overline{\mathbb{D}}}(\mathbb{B}_{j_t}^{-w_{tj_0}}),$$
and for all $1\leq t\leq \left\lfloor\frac{\#|M_j|}{j_0}\right\rfloor$,
$$d_{\mathbb{C}\setminus\overline{\mathbb{D}}}(\mathbb{B}_{j_t}^{-w_{tj_0}},z)\preccurlyeq
{\rm diam}_{\mathbb{C}\setminus\overline{\mathbb{D}}}(\mathbb{B}_{j_t}^{-w_{tj_0}})\ {\rm and}\
{\rm diam}_{\mathbb{C}\setminus\overline{\mathbb{D}}}(\mathbb{B}_{j_t}^{-w_{tj_0}})\preccurlyeq1.$$
In fact,
by (\ref{e7.12}) and Lemma \ref{l6}, for all $1\leq t\leq \left\lfloor\frac{\#|M_j|}{j_0}\right\rfloor-1$, we have
\begin{equation}
\label{e7.10}\eta^{j_0}\cdot{\rm diam}_{\mathbb{C}\setminus\overline{\mathbb{D}}}(\mathbb{B}_{j_{t+1}}^{-(w_{(t+1)j_0}-w_{tj_0})})\preccurlyeq
{\rm diam}_{\mathbb{C}\setminus\overline{\mathbb{D}}}(\mathbb{B}_{j_{t+1}})=
{\rm diam}_{\mathbb{C}\setminus\overline{\mathbb{D}}}(\mathbb{B}_{j_t}).
\end{equation}
Observe
$$d_{\mathbb{C}\setminus\overline{\mathbb{D}}}(F_{\alpha}^{\comp w_{tj_0}}(z),\mathbb{B}_{j_{t+1}}^{-(w_{(t+1)j_0}-w_{tj_0})})\leq
l_{\mathbb{C}\setminus\overline{\mathbb{D}}}(L_{j_{t+1}}^{-(w_{(t+1)j_0}-w_{tj_0})})\preccurlyeq1$$
and
$$d_{\mathbb{C}\setminus\overline{\mathbb{D}}}(F_{\alpha}^{\comp w_{tj_0}}(z),\mathbb{B}_{j_t})\preccurlyeq1.$$
Then applying Lemma \ref{l6} to the pullback of two sides of (\ref{e7.10}) by $F_{\alpha}^{\comp w_{tj_0}}$, we get
$$\eta^{j_0}\cdot{\rm diam}_{\mathbb{C}\setminus\overline{\mathbb{D}}}(\mathbb{B}_{j_{t+1}}^{-w_{(t+1)j_0}})
\preccurlyeq{\rm diam}_{\mathbb{C}\setminus\overline{\mathbb{D}}}(\mathbb{B}_{j_t}^{-w_{tj_0}}).$$
For all $1\leq t\leq \left\lfloor\frac{\#|M_j|}{j_0}\right\rfloor$, by Lemma \ref{l6}, we have
$$d_{\mathbb{C}\setminus\overline{\mathbb{D}}}(\mathbb{B}_{j_t}^{-w_{tj_0}},z)\leq
l_{\mathbb{C}\setminus\overline{\mathbb{D}}}(L_{j_t}^{-w_{tj_0}})\preccurlyeq
{\rm diam}_{\mathbb{C}\setminus\overline{\mathbb{D}}}(\mathbb{B}_{j_t}^{-w_{tj_0}})$$
and by the Schwarz lemma
$${\rm diam}_{\mathbb{C}\setminus\overline{\mathbb{D}}}(\mathbb{B}_{j_t}^{-w_{tj_0}})\preccurlyeq1.$$

\vspace{0.2cm}
\noindent\emph{Step 5: Renormalization of the sequence pulled back}

\vspace{0.2cm}
\noindent For all $1\leq t\leq \left\lfloor\frac{\#|M_j|}{j_0}\right\rfloor$,
we let $\hat{\mathbb{B}}_{j_t}^{-w_{tj_0}}$ be the component of $F_{\alpha}^{\comp-w_{tj_0}}(\hat{\mathbb{B}}_{j_t})$
containing $\mathbb{B}_{j_t}^{-w_{tj_0}}$.
Since $\hat{\mathbb{B}}_{j_t}$ does't intersect the critical orbit of $F_{\alpha}$,
we have that $F_{\alpha}^{\comp-w_{tj_0}}$ has a univalent branch from $\hat{\mathbb{B}}_{j_t}$
to $\hat{\mathbb{B}}_{j_t}^{-w_{tj_0}}$, written as $h_{j_t}$.
Observe that $h_{j_t}(\mathbb{B}_{j_t})=\mathbb{B}_{j_t}^{-w_{tj_0}}$ and
${\rm diam}_{\mathbb{C}\setminus\overline{\mathbb{D}}}(\mathbb{B}_{j_t}^{-w_{tj_0}})\preccurlyeq1$.
Then by the Koebe distortion theorem
there exists a hyperbolic ball $B_{j_t}\subseteq\mathbb{B}_{j_t}^{-w_{tj_0}}$ with respect to
the hyperbolic metric on $\mathbb{C}\setminus\overline{\mathbb{D}}$ such that
$${\rm diam}_{\mathbb{C}\setminus\overline{\mathbb{D}}}(B_{j_t})\asymp
{\rm diam}_{\mathbb{C}\setminus\overline{\mathbb{D}}}(\mathbb{B}_{j_t}^{-w_{tj_0}}).$$
Thus by step 4 , for all $1\leq t\leq \left\lfloor\frac{\#|M_j|}{j_0}\right\rfloor-1$, we have
$$\eta^{j_0}\cdot{\rm diam}_{\mathbb{C}\setminus\overline{\mathbb{D}}}(B_{j_{t+1}})\preccurlyeq
{\rm diam}_{\mathbb{C}\setminus\overline{\mathbb{D}}}(B_{j_t})$$
and for all $1\leq t\leq \left\lfloor\frac{\#|M_j|}{j_0}\right\rfloor$,
$$d_{\mathbb{C}\setminus\overline{\mathbb{D}}}(B_{j_t},z)\preccurlyeq
{\rm diam}_{\mathbb{C}\setminus\overline{\mathbb{D}}}(B_{j_t})\ {\rm and}\
{\rm diam}_{\mathbb{C}\setminus\overline{\mathbb{D}}}(B_{j_t})<
{\rm diam}_{\mathbb{C}\setminus\overline{\mathbb{D}}}(\mathbb{B}_{j_t}^{-w_{tj_0}})\preccurlyeq1.$$

\vspace{0.2cm}
\noindent\emph{Step 6: Lift to the upper half plane}

\vspace{0.2cm}
\noindent Let $p(z)=e^{-iz}$ and $p^{\comp-1}_0$ be the univalent branch of $p^{\comp-1}$ on $\mathbb{C}\setminus(-\infty,0]$
whose value domain contains $0$.
For all $1\leq t\leq \left\lfloor\frac{\#|M_j|}{j_0}\right\rfloor$,
we set $\tilde{B}_{j_t}:=p^{\comp-1}_0(B_{j_t})$ and $\tilde{z}:=p^{\comp-1}_0(z)$.
Since $p$ is a local isometry from the upper half plane $\mathbb{H}$ to $\mathbb{C}\setminus\overline{\mathbb{D}}$,
we have that $\tilde{B}_{j_t}$ is a hyperbolic ball in the upper half plane.
Thus by step 5, for all $1\leq t\leq \left\lfloor\frac{\#|M_j|}{j_0}\right\rfloor-1$, we have
$$\eta^{j_0}\cdot{\rm diam}_{\mathbb{H}}(\tilde{B}_{j_{t+1}})\preccurlyeq
{\rm diam}_{\mathbb{H}}(\tilde{B}_{j_t})$$
and for all $1\leq t\leq \left\lfloor\frac{\#|M_j|}{j_0}\right\rfloor$,
$$d_{\mathbb{H}}(\tilde{B}_{j_t},\tilde{z})\preccurlyeq{\rm diam}_{\mathbb{H}}(\tilde{B}_{j_t})\ {\rm and}\
{\rm diam}_{\mathbb{H}}(\tilde{B}_{j_t})\preccurlyeq1.$$
Observe that a hyperbolic ball in $\mathbb{H}$ is an Euclidean ball in $\mathbb{H}$ and
for any two points $w_1,w_2\in\mathbb{H}$ with $d_{\mathbb{H}}(w_1,w_2)\preccurlyeq1$,
we have $d_{\mathbb{H}}(w_1,w_2)\asymp\frac{d(w_1,w_2)}{{\rm Im}(w_1)}$.
Thus for all $1\leq t\leq \left\lfloor\frac{\#|M_j|}{j_0}\right\rfloor-1$,
\begin{equation}
\label{e4.11}\eta^{j_0}\cdot{\rm diam}(\tilde{B}_{j_{t+1}})\preccurlyeq
{\rm diam}(\tilde{B}_{j_t})
\end{equation}
and for all $1\leq t\leq \left\lfloor\frac{\#|M_j|}{j_0}\right\rfloor$,
\begin{equation}
\label{e4.12}d(\tilde{B}_{j_t},\tilde{z})\preccurlyeq{\rm diam}(\tilde{B}_{j_t})\ {\rm and}\
{\rm diam}(\tilde{B}_{j_t})\preccurlyeq{\rm Im}(\tilde{z})<l_{n_j}.
\end{equation}

\vspace{0.2cm}
\noindent\emph{Step 7: Apply Lemma \ref{L1}}

\vspace{0.2cm}
\noindent Let $\tilde{S}_{l_{n_j}}$ be the component of $p^{\comp-1}(S_{l_{n_j}})$ containing $0$.
Then $\tilde{S}_{l_{n_j}}$ is a square as the following:
$$\tilde{S}_{l_{n_j}}=\{w:\left|{\rm Re}(w)\right|\leq l_{n_j},\left|{\rm Im}(w)\right|\leq l_{n_j}\}.$$
Since
$\{F_{\alpha}^{\comp j}(z)\}_{j=0}^{m_t}\subseteq\mathbb{D}_{1+\alpha_0}$ for all $t\in T_j$,
by (\ref{e7.9}) and (\ref{e4.9}) there exists a universal constant $c>1$ such that
\begin{equation}
\label{e7.13}\mathbb{B}_{j_t}^{-w}\subseteq\mathbb{D}_{1+c\alpha_0}
\end{equation}
for all $0\leq w\leq w_{tj_0}$.
Let
$$E_j:=\{z\in\tilde{S}_{l_{n_j}}:p(z)\in S_{l_{n_j}}\ {\rm with}\ \Lambda_{c\alpha_0}(p(z))>l_{s_j}^+\}.$$
Then
$$p(E_j)=\{z\in S_{l_{n_j}}:\Lambda_{c\alpha_0}(z)>l_{s_j}^+\}.$$
For all $1\leq t\leq \left\lfloor\frac{\#|M_j|}{j_0}\right\rfloor$,
since $F_{\alpha}^{\comp w_{tj_0}}(p(\tilde{B}_{j_t}))\subseteq\mathbb{B}_{j_t}\subset U_0\cap\Lambda_{\alpha_0}^{-1}(]0,l_{s_j}^+[)$ and
(\ref{e7.13}), we have $\tilde{B}_{j_t}\cap E_j=\emptyset$.
By (\ref{e4.11}) and (\ref{e4.12}), for large enough $j_0$, we have that
for all $1\leq t\leq \left\lfloor\frac{\#|M_j|}{j_0}\right\rfloor-1$,
$${\rm diam}(\tilde{B}_{j_{t+1}})\leq\lambda_1^{j_0}\cdot{\rm diam}(\tilde{B}_{j_t})$$
and for all $3\leq t\leq \left\lfloor\frac{\#|M_j|}{j_0}\right\rfloor$,
$$d(\tilde{B}_{j_t},\tilde{z})\preccurlyeq{\rm diam}(\tilde{B}_{j_t})\ {\rm and}\
{\rm diam}(\tilde{B}_{j_t})<\lambda_1^{j_0}\cdot l_{n_j},$$
where $0<\lambda_1<1$ is a universal constant.
Applying Lemma \ref{L1} to $\tilde{S}_{l_{n_j}}$ and $E_j$, by (\ref{e6.8})
we have
$${\rm area}(E_j)\leq
\lambda_2^{s_j}\cdot{\rm area}(\tilde{S}_{l_{n_j}}),$$
where $0<\lambda_2<1$ is a universal constant.
Then
${\rm area}(p(E_j))\leq c_2\cdot\lambda_2^{s_j}\cdot{\rm area}(S_{l_{n_j}})$,
where $c_2$ is a universal constant.
Since for large enough $j$,
$\mathbb{B}(1,r_j)\subseteq S_{l_{n_j}}\ {\rm and}\ r_j\asymp l_{n_j}$,
we have
$${\rm area}(p(E_j)\cap\mathbb{B}(1,r_j))\leq c_3\cdot\lambda_2^{s_j}\cdot{\rm area}(\mathbb{B}(1,r_j)),$$
where $c_3$ is a universal constant.

\vspace{0.2cm}
\noindent\emph{Step 8: Complete the proof}

\vspace{0.2cm}
\noindent Let
$$D_n:=\cup_{m<n}\{\exp(\gamma):\gamma\ {\rm is\ an\ m-cell\ of\ \Gamma}\},$$
$$D_n^j:=\{z\in\mathbb{B}(1,r_j):\exists k\geq0\ {\rm s.t.}\ F_{\alpha}^{\comp k}(z)\in D_n\}$$
and
$$R_n^j:=\{z\in\mathbb{B}(1,r_j):\exists k\geq0\ {\rm s.t.}\ F_{\alpha}^{\comp k}(z)\in R_n\}.$$
By Subsection \ref{ss5.2}, the dilatation $\frac{1+|\mu_0(z)|}{1-|\mu_0(z)|}$ on $D_n$ is
at most
$$\sup_{m\geq n}C_1'\left(1+(\log b_m)^2\right)
\leq\max\left\{C_1'\left(1+C_1^2n\right), C_1'\left(1+(\log M)^2\right)\right\}$$
with the constant $C_1'$ not depending on $n$.
Then if $K>C_1'\left(1+(\log M)^2\right)$,
since $C_1'\left(1+C_1^2\left\lfloor\frac{\frac{K}{C_1'}-1}{C_1^2}\right\rfloor\right)\leq K$,
we have
$$\left\{z\in\mathbb{B}(1,r_j):\frac{1+|\mu_{r_j}(z)|}{1-|\mu_{r_j}(z)|}>K\right\}\subseteq
R_{n_K+t_j-1}^j\cup(D_{s_j}^j\setminus D_{n_K}^j),$$
where $n_K=\left\lfloor\frac{\frac{K}{C_1'}-1}{C_1^2}\right\rfloor$. We take $K_0$ such that $K_0\geq\max\{C_1'\left(1+(\log M)^2\right),1\}$ and $n_{K_0}\geq1$.
Then for $K>K_0$, we have
\begin{equation}
\label{e7.15}{\rm area}\left\{z\in\mathbb{B}(1,r_j):\frac{1+|\mu_{r_j}(z)|}{1-|\mu_{r_j}(z)|}>K\right\}
\leq{\rm area}(R_{n_K+t_j-1}^j)+{\rm area}(D_{s_j}^j\setminus D_{n_K}^j).
\end{equation}
By Step 1, we have
\begin{equation}
\label{e7.16}{\rm area}(R_{n_K+t_j-1}^j)\leq
\nu(R_{n_K+t_j-1})\leq\frac{1}{\pi\lambda^2\sigma^{\beta}}\cdot(\sigma^{\beta})^{n_K}\cdot{\rm area}(\mathbb{B}(1,r_j)).
\end{equation}
By properties of $l_n^-$,
in Step $2$ the choice of $c_1$ makes sure $D_{s_j}\cap\mathbb{B}(1,r_j)=\emptyset$.
Then by properties of $l_n^+$,
$D_{s_j}^j\subseteq p(E_j)\cap\mathbb{B}(1,r_j)$.
Thus by Step 7, we have
\begin{equation}
\label{e7.17}{\rm area}(D_{s_j}^j\setminus D_{n_K}^j)\left\{\begin{matrix}
\leq c_3\cdot(\lambda_2)^{s_j}\cdot{\rm area}(\mathbb{B}(1,r_j)),&n_K\leq s_j\\
=0,\quad\quad\quad&n_K>s_j.\end{matrix}\right.
\end{equation}
By (\ref{e7.15}), (\ref{e7.16}) and (\ref{e7.17}), for $K>K_0$, we have
\begin{align*}
&\qquad\frac{{\rm area}\left\{z\in\mathbb{B}(1,r_j):\frac{1+|\mu_{r_j}(z)|}{1-|\mu_{r_j}(z)|}>K\right\}}{{\rm area}(\mathbb{B}(1,r_j)}\\
&\leq(\frac{1}{\pi\lambda^2\sigma^{\beta}}+c_3)\cdot(\max\{\sigma^{\beta},\lambda_2\})^{n_K}\\
&\leq(\frac{1}{\pi\lambda^2\sigma^{\beta}}+c_3)\cdot(\max\{\sigma^{\beta},\lambda_2\})^{\frac{\frac{K}{C_1'}-1}{C_1^2}-1}\\
&=\frac{\frac{1}{\pi\lambda^2\sigma^{\beta}}+c_3}{(\max\{\sigma^{\beta},\lambda_2\})^{\frac{1}{C_1^2}+1}}
\cdot e^{\log(\max\{\sigma^{\beta},\lambda_2\})\frac{K}{C_1'C_1^2}}.
\end{align*}
Thus by taking
$$K_0=K_0,\ K_1=\frac{\frac{1}{\pi\lambda^2\sigma^{\beta}}+c_3}{(\max\{\sigma^{\beta},\lambda_2\})^{\frac{1}{C_1^2}+1}}\
{\rm and}\ K_2=-\frac{C_1'C_1^2}{\log(\max\{\sigma^{\beta},\lambda_2\})},$$
the proof is completed.

\end{proof}

\begin{corollary}
\label{C2}Using the same notations as those in the proof of Proposition \ref{p1}, we have
$$\lim_{j\to\infty}\frac{{\rm area}(\mathbb{B}(1,r_j)\setminus K_{c\alpha_0}(F_{\alpha}))}{{\rm area}(\mathbb{B}(1,r_j))}=0.$$
\end{corollary}
\begin{proof}
By Step $7$ in the proof of Proposition \ref{p1}, we have
$${\rm area}(p(E_j)\cap\mathbb{B}(1,r_j))\leq c_3\cdot\lambda_2^{s_j}\cdot{\rm area}(\mathbb{B}(1,r_j)).$$
Since $\mathbb{B}(1,r_j)\setminus K_{c\alpha_0}(F_{\alpha})\subseteq p(E_j)\cap\mathbb{B}(1,r_j)$, we have
$${\rm area}(\mathbb{B}(1,r_j)\setminus K_{c\alpha_0}(F_{\alpha}))\leq
{\rm area}(p(E_j)\cap\mathbb{B}(1,r_j))\leq c_3\cdot\lambda_2^{s_j}\cdot{\rm area}(\mathbb{B}(1,r_j)).$$
This implies this corollary.
\end{proof}

\section{The proof of Theorem \ref{T1}}
\noindent\emph{Step 1: Coordinate decompositions}

\vspace{0.2cm}
\noindent By Proposition \ref{p1} and Corollary \ref{C2}, we can obtain that
there exists a sequence $\{r_n\}_{n=1}^{\infty}$ of positive real numbers with $\lim_{n\to\infty}r_n=0$
and three positive real numbers $K_0,K_1,K_2$ such that $\hat{\mu}_{r_n}\in\mathcal{F}(K_0,K_1,K_2)$ for large enough $n$
and
\begin{equation}
\label{e8.1}\lim_{n\to\infty}\frac{{\rm area}(\mathbb{B}(1,r_n)\setminus K_{\alpha_0}(F_{\alpha}))}{{\rm area}(\mathbb{B}(1,r_n))}=0.
\end{equation}
By Theorem \ref{TD1}, for all $n\geq1$, there is a David map $\phi_n:\mathbb{B}(1,1)\to\mathbb{B}(1,1)$ such that
$\phi_n(1)=1$ and $\mu_{\phi_n}=\hat{\mu}_{r_n}$.
For all $n\geq1$, we let
$$E_n:=\{z\in\mathbb{B}(1,1):r_n(z-1)+1\in\mathbb{B}(1,r_n)\setminus K_{\alpha_0}(F_{\alpha})\}.$$
It follows from (\ref{e8.1}) and Corollary \ref{c4.1} that
\begin{equation}
\label{e5.6}\lim_{n\to\infty}{\rm area}(\phi_n(E_n))=0.
\end{equation}

We define
$$T_n(\cdot):=\frac{\cdot-1}{r_n}+1:\mathbb{C}\to\mathbb{C}$$
and
$$\psi_{\alpha,n}(\cdot):=\phi_{\alpha}\comp T_n^{\comp-1}\comp\phi_n^{\comp-1}(\cdot):\mathbb{B}(1,1)\to\mathbb{C}.$$
Then the coordinate map $\phi_{\alpha}$ has the following decomposition:
$$\phi_{\alpha}(z)=\psi_{\alpha,n}\comp\phi_n\comp T_n(z),\ z\in\mathbb{B}(1,r_n).$$
Since $\mu_{\phi_n\comp T_n}(z)=\mu_0(z)$ and $\mu_{\phi_{\alpha}}(z)=\mu_0(z)$ on $\mathbb{B}(1,r_n)$,
by Theorem \ref{TD1} we have that
$\psi_{\alpha,n}$ is a conformal map on $\mathbb{B}(1,1)$.
By definition of $E_n$, we have
\begin{equation}
\label{e8.3}\psi_{\alpha,n}(\phi_n(E_n))=\phi_{\alpha}(\mathbb{B}(1,r_n))\setminus\phi_{\alpha}(K_{\alpha_0}(F_{\alpha})).
\end{equation}
Given $\alpha_1>0$, we let $\alpha_0$ small enough so that
$$\phi_{\alpha}(\mathbb{D}_{1+\alpha_0})\subseteq\Delta_{\alpha}(\alpha_1/2)\ {\rm and\ hence}\
\phi_{\alpha}(K_{\alpha_0}(F_{\alpha}))\subseteq K_{\alpha_1/2}(P_{\alpha}).$$

\vspace{0.2cm}
\noindent\emph{Step 2: Go to the critical point $1$ through $\phi_{\alpha}$}

\vspace{0.2cm}
\noindent By Proposition \ref{pr2.1},
there exists a sequence $\{z_n\}_{n=1}^{+\infty}\subset\overline{U_0}$ with a quasi-log-arithmetic sequence
$\{l_n:=|z_n-1|\}_{n=1}^{+\infty}$ such that
for any $N>0$ and $n>0$, if $z\in S_{l_{n+N}}$ with $\Lambda_{\alpha_0}(z)>l_n$, then
there exists a nonnegative integer $m(n,N,z)$ such that
$$S_{l_1}\cup\{F_{\alpha}^{\comp s}(z)\}_{s=0}^{m(n,N,z)}\subseteq\mathbb{D}_{1+\alpha_0}\ {\rm and}\
d_{\mathbb{C}\setminus\overline{\mathbb{D}}}(F_{\alpha}^{\comp m(n,N,z)}(z),z_n)\preccurlyeq1.$$
For all large enough $j$, there exists a positive integer $n_j$ with
$l_{n_j}\asymp r_j$ such that for all $N\geq1$ and all $z\in S_{n_j+N}$,
\begin{equation}
\label{e701}|F_{\alpha}^{\comp m(n_j,N,z)}(z)-1|<r_j/2.
\end{equation}
For any $w\in\partial\Delta_{\alpha}$, $\phi_{\alpha}^{\comp-1}(w)\in\mathbb{S}^1$
and there exists a positive integer $m_j$ such that
$$F_{\alpha}^{\comp m_j}\comp\phi_{\alpha}^{\comp-1}(w)\in(S_{l_{n_j+1}})^{\comp}.$$
Thus for any small enough square $S$ centering at $w$,
\begin{equation}
\label{e5.3}F_{\alpha}^{\comp m_j}\comp\phi_{\alpha}^{\comp-1}(S)\subseteq S_{l_{n_j+1}}
\end{equation}
and for $0\leq t\leq m_j$,
\begin{equation}
\label{e5.4}F_{\alpha}^{\comp t}\comp\phi_{\alpha}^{\comp-1}(S)\subseteq\mathbb{D}_{1+\alpha_0}.
\end{equation}
Given such a square $S$, for any $z\in S\setminus K_{\alpha_1}(P_{\alpha})$,
we have
$z\not\in\phi_{\alpha}(K_{\alpha_0}(F_{\alpha}))$, that is
$\phi_{\alpha}^{\comp-1}(z)\not\in K_{\alpha_0}(F_{\alpha}).$
Then by (\ref{e5.3}) and (\ref{e5.4}), we have
$$z':=F_{\alpha}^{\comp m_j}\comp\phi_{\alpha}^{\comp-1}(z)\in S_{l_{n_j+1}}\setminus K_{\alpha_0}(F_{\alpha}).$$
Then
$z'\in S_{l_{n_j+1}}$ with $\Lambda_{\alpha_0}(z')>l_{n_j}$
and hence
there exists a nonnegative integer $m:=m(n_j,1,z')$ such that
\begin{equation}
\label{e702}\{F_{\alpha}^{\comp s}(z')\}_{s=0}^m\subseteq\mathbb{D}_{1+\alpha_0}\ {\rm and}\
d_{\mathbb{C}\setminus\overline{\mathbb{D}}}(F_{\alpha}^{\comp m}(z'),z_{n_j})\preccurlyeq1.
\end{equation}
Since $\{F_{\alpha}^{\comp s}(z')\}_{s=0}^m\subseteq\mathbb{D}_{1+\alpha_0}$ and
$\phi_{\alpha}(\mathbb{D}_{1+\alpha_0})\subseteq\Delta_{\alpha}(\alpha_1/2)$, we have
\begin{equation}
\label{e8.6}\phi_{\alpha}\comp F_{\alpha}^{\comp s}(z')\in\Delta_{\alpha}(\alpha_1/2)\ (0\leq s\leq m).
\end{equation}
Thus by (\ref{e701}) and (\ref{e702}), we have that for all large enough $j$,
there exists a positive real number $v_j$ with $v_j\asymp r_j$ such that
$$\mathbb{B}(F_{\alpha}^{\comp m}(z'),v_j)\subseteq\mathbb{B}(1,r_j)\ {\rm and}\
\mathbb{B}(F_{\alpha}^{\comp m}(z'),v_j)\cap\mathbb{D}=\emptyset.$$
Then it follows from $\mathbb{B}(F_{\alpha}^{\comp m}(z'),v_j)\cap\mathbb{D}=\emptyset$ that
$\phi_{\alpha}(\mathbb{B}(F_{\alpha}^{\comp m}(z'),v_j))\cap\Delta_{\alpha}=\emptyset$.

\vspace{0.2cm}
\noindent\emph{Step 3: Pull back $\mathbb{B}(F_{\alpha}^{\comp m}(z'),v_j)$ through the coordinate decomposition}

\vspace{0.2cm}
\noindent By the last part of Corollary \ref{c4.1},
$\phi_j\comp T_j(\mathbb{B}(F_{\alpha}^{\comp m}(z'),v_j))$
contains $\mathbb{B}(\phi_j\comp T_j\comp F_{\alpha}^{\comp m}(z'),v)$,
where $v>0$ is a constant not depending on $j$ and $z$.
By (\ref{e5.6}), we have
\begin{equation}
\label{e5.7}\lim_{j\to\infty}{\rm area}\left(\mathbb{B}(\phi_j\comp T_j\comp F_{\alpha}^{\comp m}(z'),v)
\cap\phi_j(E_j)\right)=0,\ {\rm uniform\ on\ z.}
\end{equation}
By the distortion theorem of conformal maps,
$\psi_{\alpha,j}(\mathbb{B}(\phi_j\comp T_j\comp F_{\alpha}^{\comp m}(z'),v/2))$
contains a ball $\mathbb{B}(\phi_{\alpha}\comp F_{\alpha}^{\comp m}(z'),v_j')$ with
$$v_j'\asymp_{c_4}{\rm diam}(\psi_{\alpha,j}(\mathbb{B}(\phi_j\comp T_j\comp F_{\alpha}^{\comp m}(z'),v/2))),$$
where $c_4$ is a constant not depending on $j$ and $z$. Then by (\ref{e8.3}) and (\ref{e5.7}), we have
\begin{equation*}
\lim_{j\to\infty}\frac{{\rm area}\left(\mathbb{B}(\phi_{\alpha}\comp F_{\alpha}^{\comp m}(z'),v_j')
\setminus\phi_{\alpha}(K_{\alpha_0}(F_{\alpha}))\right)}{(v_j')^2}=0,\ {\rm uniform\ on\ z.}
\end{equation*}
This implies
\begin{equation}
\label{e5.8}\lim_{j\to\infty}\frac{{\rm area}\left(\mathbb{B}(\phi_{\alpha}\comp F_{\alpha}^{\comp m}(z'),v_j')
\setminus K_{\alpha_1/2}(P_{\alpha})\right)}{(v_j')^2}=0,\ {\rm uniform\ on\ z.}
\end{equation}

\vspace{0.2cm}
\noindent\emph{Step 4: Pull back to $z$ by $P_{\alpha}$}

\vspace{0.2cm}
\noindent Since $\phi_{\alpha}(\mathbb{B}(F_{\alpha}^{\comp m}(z'),v_j))\cap\Delta_{\alpha}=\emptyset$,
we have
\begin{equation}
\label{e8.9}\mathbb{B}(\phi_{\alpha}\comp F_{\alpha}^{\comp m}(z'),v_j')\cap\Delta_{\alpha}=\emptyset.
\end{equation}
Then there exists a univalent branch $h$ of $P_{\alpha}^{\comp-(m+m_j)}$ on $\mathbb{B}(\phi_{\alpha}\comp F_{\alpha}^{\comp m}(z'),v_j')$
with $h\comp\phi_{\alpha}\comp F_{\alpha}^{\comp m}(z')=z$.
Combining (\ref{e5.4}) and (\ref{e8.6}),
we have
$$\{P_{\alpha}^{\comp s}(z)\}_{s=0}^{m+m_j}\subseteq\Delta_{\alpha}(\alpha_1/2).$$
Then it follows from the Koebe distortion theorem that
there exists a real number $\hat{v}>1$ (not depending on $j$ and $z$) such that
\begin{equation}
\label{e5.9}h(\mathbb{B}(\phi_{\alpha}\comp F_{\alpha}^{\comp m}(z'),v_j'/\hat{v})\cap K_{\alpha_1/2}(P_{\alpha}))
\subseteq K_{\alpha_1}(P_{\alpha}).
\end{equation}
Then there exists a positive real number $v_j''$ such that
\begin{equation}
\label{e5.10}\mathbb{B}(z,v_j'')\subseteq h(\mathbb{B}(\phi_{\alpha}\comp F_{\alpha}^{\comp m}(z'),v_j'/\hat{v}))\
{\rm and}\ v_j''\asymp_{c_5}{\rm diam}(h(\mathbb{B}(\phi_{\alpha}\comp F_{\alpha}^{\comp m}(z'),v_j'/\hat{v}))),
\end{equation}
where $c_5$ is a constant not depending on $j$ and $z$.
By $(\ref{e5.8})$, $(\ref{e5.9})$ and $(\ref{e5.10})$, we have
\begin{equation*}
\lim_{j\to\infty}\frac{{\rm area}\left(\mathbb{B}(z,v_j'')
\setminus K_{\alpha_1}(P_{\alpha})\right)}{(v_j'')^2}=0,\ {\rm uniform\ on\ z.}
\end{equation*}
Moreover, $v_j''$ is less than the side length of $S$, for $\mathbb{B}(z,v_j'')\cap\Delta_{\alpha}=\emptyset$.
At last, Lemma \ref{l2} gives
$$\lim_{{\rm diam}(S)\to0}\frac{{\rm area}(S\setminus K_{\alpha_1}(P_{\alpha}))}{{\rm area}(S)}=0,$$
which completes the proof of this theorem.

\section{The proof of Theorem \ref{T1.2}}
This section is devoted to proving Theorem \ref{T1.2}.
Given any irrational number $\alpha$, we divide this proof into two parts:
\begin{itemize}
\item[(1)] $1$ is a measurable deep of $K_{\alpha_0}(F_{\alpha})$;
\item[(2)] Any point in $\mathbb{S}^1$ is a Lebesgue density point of $K_{\alpha_0}(F_{\alpha})$.
\end{itemize}

\vspace{0.2cm}
\noindent{\bf The proof of (1):}
By Proposition \ref{pr2.1},
there exists a sequence $\{z_n\}_{n=1}^{+\infty}\subset\overline{U_0}$ with
a quasi-log-arithmetic sequence $\left\{l_n:=|z_n-1|\right\}_{n=1}^{\infty}$ such that
for any $n>1$,
if $z\in S_{l_n}$ with $\Lambda_{\alpha_0}(z)=+\infty$, then for any $t$ with $1\leq t<n$,
there exists a nonnegative integer $m$ such that
$$\{F_{\alpha}^{\comp j}(z)\}_{j=0}^m\subseteq\mathbb{D}_{1+\alpha_0}\ {\rm and}\
d_{\mathbb{C}\setminus\overline{\mathbb{D}}}(F_{\alpha}^{\comp m}(z),z_t)\preccurlyeq1.$$
The next proof is similar to Steps $3-7$ in the proof of Proposition \ref{p1}.

\vspace{0.2cm}
\noindent\emph{Step 1: A sequence of hyperbolic balls contained in $U_0$}

\vspace{0.2cm}
\noindent We fix $n>1$ and a point $z\in S_{l_n}$ with $\Lambda_{\alpha_0}(z)=+\infty$.
Then for all $t$ with $1\leq t<n$,
there exists a nonnegative integer $m_t$ such that
\begin{equation}
\label{e9.1}
\{F_{\alpha}^{\comp s}(z)\}_{s=0}^{m_t}\subseteq\mathbb{D}_{1+\alpha_0}\ {\rm and}\
d_{\mathbb{C}\setminus\overline{\mathbb{D}}}(F_{\alpha}^{\comp m_t}(z),z_t)\preccurlyeq1.
\end{equation}
We can choose a sequence $\{\hat{\mathbb{B}}_t\}_1^{n-1}$ of hyperbolic balls
with the same size with respect to the hyperbolic metric
on $\mathbb{C}\setminus\overline{\mathbb{D}}$ such that for all $1\leq t\leq n-1$,
\begin{equation}
\label{e9.2}\hat{\mathbb{B}}_t\subseteq U_0\cap\mathbb{D}_{1+\alpha_0},\
d_{\mathbb{C}\setminus\overline{\mathbb{D}}}(\hat{\mathbb{B}}_t,z_t)\preccurlyeq1\ {\rm and}\
{\rm diam}_{\mathbb{C}\setminus\overline{\mathbb{D}}}(\hat{\mathbb{B}}_t)\asymp1.
\end{equation}
Let $\mathbb{B}_t$ be a hyperbolic ball with the same center and the half size as $\hat{\mathbb{B}}_t$ for all $1\leq t\leq n-1$.
Thus by (\ref{e9.1}) and (\ref{e9.2}), for all $1\leq t\leq n-1$, we have
\begin{equation}
\label{e9.3} d_{\mathbb{C}\setminus\overline{\mathbb{D}}}(F_{\alpha}^{\comp m_t}(z),\mathbb{B}_t)\preccurlyeq1.
\end{equation}

\vspace{0.2cm}
\noindent\emph{Step 2: Pull back the sequence of hyperbolic balls}

\vspace{0.2cm}
\noindent From now on, we assume that $n$ is large enough.
We set
$$M_n:=\{m_t:1\leq t\leq n-1\}=\{w_1,w_2,\cdots,w_{\#|M_n|}\}$$
with
$$w_1<w_2<\cdots<w_{\#|M_n|}.$$
By (\ref{e9.1}), we have
\begin{equation}
\label{e10}\#|M_n|\asymp n.
\end{equation}
Fixing $j_0\geq2$,
we set
$$M^{j_0}_n:=\{w_{j_0},w_{2j_0},\cdots,w_{\left\lfloor\frac{\#|M_n|}{j_0}\right\rfloor j_0}\}.$$
For any $1\leq t\leq \left\lfloor\frac{\#|M_n|}{j_0}\right\rfloor$,
we set $m_{j_t}:=w_{tj_0}$ and hence by (\ref{e9.3}) we have
\begin{equation*}
d_{\mathbb{C}\setminus\overline{\mathbb{D}}}(F_{\alpha}^{\comp w_{tj_0}}(z),\mathbb{B}_{j_t})\preccurlyeq1.
\end{equation*}
Let $L_{j_t}\subset\mathbb{C}\setminus\overline{\mathbb{D}}$ be a curve of the length
$l_{\mathbb{C}\setminus\overline{\mathbb{D}}}(L_{j_t})\preccurlyeq1$
connecting $F_{\alpha}^{\comp w_{tj_0}}(z)$ and $\mathbb{B}_{j_t}$.
For any $1\leq t\leq \left\lfloor\frac{\#|M_n|}{j_0}\right\rfloor$ and $1\leq w\leq w_{tj_0}$,
we consider the pullback of $\{F_{\alpha}^{\comp w_{tj_0}}(z)\}\cup L_{j_t}\cup\mathbb{B}_{j_t}$
under $F_{\alpha}^{\comp w}$ along $F_{\alpha}^{\comp w_{tj_0}}(z)$, $F_{\alpha}^{\comp(w_{tj_0}-1)}(z)$,
$\cdots$, $F_{\alpha}^{\comp(w_{tj_0}-w)}(z)$.
By $\mathbb{B}_{j_t}^{-w}$ and $L_{j_t}^{-w}$ we denote pre-images of $\mathbb{B}_{j_t}$ and $L_{j_t}$ in the pullback respectively.
It follows from Lemma \ref{l6} and (\ref{e9.1}) that there exists a universal constant $\eta>1$ such that
for any $1\leq t\leq\left\lfloor\frac{\#|M_n|}{j_0}\right\rfloor$,
\begin{equation}
\label{e9.5}\eta^{j_0}\cdot\rho_{\mathbb{C}\setminus\overline{\mathbb{D}}}\left(F_{\alpha}^{\comp w_{(t-1)j_0}}(z)\right)\leq
\rho_{\mathbb{C}\setminus\overline{\mathbb{D}}}\left(F_{\alpha}^{\comp w_{tj_0}}(z)\right)
\cdot\left|\left(F_{\alpha}^{\comp(w_{tj_0}-w_{(t-1)j_0})}\right)'(z)\right|,
\end{equation}
where $w_0=0$.
Next, we will prove that for all $1\leq t\leq \left\lfloor\frac{\#|M_j|}{j_0}\right\rfloor-1$,
$$\eta^{j_0}\cdot{\rm diam}_{\mathbb{C}\setminus\overline{\mathbb{D}}}(\mathbb{B}_{j_{t+1}}^{-w_{(t+1)j_0}})\preccurlyeq{\rm diam}_{\mathbb{C}\setminus\overline{\mathbb{D}}}(\mathbb{B}_{j_t}^{-w_{tj_0}}),$$
and for all $1\leq t\leq \left\lfloor\frac{\#|M_j|}{j_0}\right\rfloor$,
$$d_{\mathbb{C}\setminus\overline{\mathbb{D}}}(\mathbb{B}_{j_t}^{-w_{tj_0}},z)\preccurlyeq
{\rm diam}_{\mathbb{C}\setminus\overline{\mathbb{D}}}(\mathbb{B}_{j_t}^{-w_{tj_0}})\ {\rm and}\
{\rm diam}_{\mathbb{C}\setminus\overline{\mathbb{D}}}(\mathbb{B}_{j_t}^{-w_{tj_0}})\preccurlyeq1.$$
In fact,
by (\ref{e9.5}) and Lemma \ref{l6}, for all $1\leq t\leq \left\lfloor\frac{\#|M_j|}{j_0}\right\rfloor-1$, we have
\begin{equation}
\label{e9.6}\eta^{j_0}\cdot{\rm diam}_{\mathbb{C}\setminus\overline{\mathbb{D}}}(\mathbb{B}_{j_{t+1}}^{-(w_{(t+1)j_0}-w_{tj_0})})\preccurlyeq
{\rm diam}_{\mathbb{C}\setminus\overline{\mathbb{D}}}(\mathbb{B}_{j_{t+1}})=
{\rm diam}_{\mathbb{C}\setminus\overline{\mathbb{D}}}(\mathbb{B}_{j_t}).
\end{equation}
Observe
$$d_{\mathbb{C}\setminus\overline{\mathbb{D}}}(F_{\alpha}^{\comp w_{tj_0}}(z),\mathbb{B}_{j_{t+1}}^{-(w_{(t+1)j_0}-w_{tj_0})})\leq
l_{\mathbb{C}\setminus\overline{\mathbb{D}}}(L_{j_{t+1}}^{-(w_{(t+1)j_0}-w_{tj_0})})\preccurlyeq1$$
and
$$d_{\mathbb{C}\setminus\overline{\mathbb{D}}}(F_{\alpha}^{\comp w_{tj_0}}(z),\mathbb{B}_{j_t})\preccurlyeq1.$$
Then applying Lemma \ref{l6} to (\ref{e9.6}), we have
$$\eta^{j_0}\cdot{\rm diam}_{\mathbb{C}\setminus\overline{\mathbb{D}}}(\mathbb{B}_{j_{t+1}}^{-w_{(t+1)j_0}})
\preccurlyeq{\rm diam}_{\mathbb{C}\setminus\overline{\mathbb{D}}}(\mathbb{B}_{j_t}^{-w_{tj_0}}).$$
For all $1\leq t\leq \left\lfloor\frac{\#|M_j|}{j_0}\right\rfloor$, by Lemma \ref{l6}, we have
$$d_{\mathbb{C}\setminus\overline{\mathbb{D}}}(\mathbb{B}_{j_t}^{-w_{tj_0}},z)\leq
l_{\mathbb{C}\setminus\overline{\mathbb{D}}}(L_{j_t}^{-w_{tj_0}})\preccurlyeq
{\rm diam}_{\mathbb{C}\setminus\overline{\mathbb{D}}}(\mathbb{B}_{j_t}^{-w_{tj_0}})$$
and by the Schwarz lemma
$${\rm diam}_{\mathbb{C}\setminus\overline{\mathbb{D}}}(\mathbb{B}_{j_t}^{-w_{tj_0}})\preccurlyeq1.$$

\vspace{0.2cm}
\noindent\emph{Step 3: Renormalization of the sequence pulled back}

\vspace{0.2cm}
\noindent For all $1\leq t\leq \left\lfloor\frac{\#|M_j|}{j_0}\right\rfloor$,
we let $\hat{\mathbb{B}}_{j_t}^{-w_{tj_0}}$ be the component of $F_{\alpha}^{\comp-w_{tj_0}}(\hat{\mathbb{B}}_{j_t})$
containing $\mathbb{B}_{j_t}^{-w_{tj_0}}$.
Since $\hat{\mathbb{B}}_{j_t}$ does't intersect the critical orbit of $F_{\alpha}$,
we have that $F_{\alpha}^{\comp-w_{tj_0}}$ has a univalent branch from $\hat{\mathbb{B}}_{j_t}$
to $\hat{\mathbb{B}}_{j_t}^{-w_{tj_0}}$, written as $h_{j_t}$.
Observe that $h_{j_t}(\mathbb{B}_{j_t})=\mathbb{B}_{j_t}^{-w_{tj_0}}$ and
${\rm diam}_{\mathbb{C}\setminus\overline{\mathbb{D}}}(\mathbb{B}_{j_t}^{-w_{tj_0}})\preccurlyeq1$.
Then by the Koebe distortion theorem
there exists a hyperbolic ball $B_{j_t}\subseteq\mathbb{B}_{j_t}^{-w_{tj_0}}$ with respect to
the hyperbolic metric on $\mathbb{C}\setminus\overline{\mathbb{D}}$ such that
$${\rm diam}_{\mathbb{C}\setminus\overline{\mathbb{D}}}(B_{j_t})\asymp
{\rm diam}_{\mathbb{C}\setminus\overline{\mathbb{D}}}(\mathbb{B}_{j_t}^{-w_{tj_0}}).$$
Thus by step 2, for all $1\leq t\leq \left\lfloor\frac{\#|M_j|}{j_0}\right\rfloor-1$, we have
$$\eta^{j_0}\cdot{\rm diam}_{\mathbb{C}\setminus\overline{\mathbb{D}}}(B_{j_{t+1}})\preccurlyeq
{\rm diam}_{\mathbb{C}\setminus\overline{\mathbb{D}}}(B_{j_t})$$
and for all $1\leq t\leq \left\lfloor\frac{\#|M_j|}{j_0}\right\rfloor$,
$$d_{\mathbb{C}\setminus\overline{\mathbb{D}}}(B_{j_t},z)\preccurlyeq
{\rm diam}_{\mathbb{C}\setminus\overline{\mathbb{D}}}(B_{j_t})\ {\rm and}\
{\rm diam}_{\mathbb{C}\setminus\overline{\mathbb{D}}}(B_{j_t})<
{\rm diam}_{\mathbb{C}\setminus\overline{\mathbb{D}}}(\mathbb{B}_{j_t}^{-w_{tj_0}})\preccurlyeq1.$$

\vspace{0.2cm}
\noindent\emph{Step 4: Lift to the upper half plane}

\vspace{0.2cm}
\noindent Let $p^{\comp-1}_0$ be the univalent branch of $p^{\comp-1}$ on $\mathbb{C}\setminus(-\infty,0]$
whose value domain contains $0$.
For all $1\leq t\leq \left\lfloor\frac{\#|M_j|}{j_0}\right\rfloor$,
we set $\tilde{B}_{j_t}:=p^{\comp-1}_0(B_{j_t})$ and $\tilde{z}:=p^{\comp-1}_0(z)$.
Since $p$ is a local isometry from the upper half plane $\mathbb{H}$ to $\mathbb{C}\setminus\overline{\mathbb{D}}$,
we have that $\tilde{B}_{j_t}$ is a hyperbolic ball in the upper half plane.
Thus by step 3, for all $1\leq t\leq \left\lfloor\frac{\#|M_j|}{j_0}\right\rfloor-1$, we have
$$\eta^{j_0}\cdot{\rm diam}_{\mathbb{H}}(\tilde{B}_{j_{t+1}})\preccurlyeq
{\rm diam}_{\mathbb{H}}(\tilde{B}_{j_t})$$
and for all $1\leq t\leq \left\lfloor\frac{\#|M_j|}{j_0}\right\rfloor$,
$$d_{\mathbb{H}}(\tilde{B}_{j_t},\tilde{z})\preccurlyeq{\rm diam}_{\mathbb{H}}(\tilde{B}_{j_t})\ {\rm and}\
{\rm diam}_{\mathbb{H}}(\tilde{B}_{j_t})\preccurlyeq1.$$
Observe that a hyperbolic ball in $\mathbb{H}$ is an Euclidean ball in $\mathbb{H}$ and
for any two points $w_1,w_2\in\mathbb{H}$ with $d_{\mathbb{H}}(w_1,w_2)\preccurlyeq1$,
we have $d_{\mathbb{H}}(w_1,w_2)\asymp\frac{d(w_1,w_2)}{{\rm Im}(w_1)}$.
Thus for all $1\leq t\leq \left\lfloor\frac{\#|M_j|}{j_0}\right\rfloor-1$,
\begin{equation}
\label{e9.7}\eta^{j_0}\cdot{\rm diam}(\tilde{B}_{j_{t+1}})\preccurlyeq
{\rm diam}(\tilde{B}_{j_t})
\end{equation}
and for all $1\leq t\leq \left\lfloor\frac{\#|M_j|}{j_0}\right\rfloor$,
\begin{equation}
\label{e9.8}d(\tilde{B}_{j_t},\tilde{z})\preccurlyeq{\rm diam}(\tilde{B}_{j_t})\ {\rm and}\
{\rm diam}(\tilde{B}_{j_t})\preccurlyeq{\rm Im}(\tilde{z})<l_n.
\end{equation}

\vspace{0.2cm}
\noindent\emph{Step 5: Apply Lemma \ref{L1}}

\vspace{0.2cm}
\noindent Let $\tilde{S}_{l_n}$ be the component of $p^{\comp-1}(S_{l_n})$ containing $0$.
Then $\tilde{S}_{l_n}$ is a square as the following:
$$\tilde{S}_{l_n}=\{w:\left|{\rm Re}(w)\right|\leq l_n,\left|{\rm Im}(w)\right|\leq l_n\}.$$
Since
$\{F_{\alpha}^{\comp j}(z)\}_{j=0}^{m_t}\subseteq\mathbb{D}_{1+\alpha_0}$ for all $1\leq t\leq n-1$,
by (\ref{e9.1}) and (\ref{e9.2}) there exists a universal constant $c>1$ such that
\begin{equation}
\label{e9.9}\mathbb{B}_{j_t}^{-w}\subseteq\mathbb{D}_{1+c\alpha_0}
\end{equation}
for all $0\leq w\leq w_{tj_0}$.
Let
$$E_n:=\{z\in\tilde{S}_{l_n}:p(z)\in S_{l_n}\ {\rm with}\ \Lambda_{c\alpha_0}(p(z))=+\infty\}.$$
Then
$$p(E_n)=\{z\in S_{l_n}:\Lambda_{c\alpha_0}(z)=+\infty\}.$$
For all $1\leq t\leq \left\lfloor\frac{\#|M_n|}{j_0}\right\rfloor$,
since $F_{\alpha}^{\comp w_{tj_0}}(p(\tilde{B}_{j_t}))\subseteq\mathbb{B}_{j_t}\subset U_0$ and
(\ref{e9.9}), we have $\tilde{B}_{j_t}\cap E_j=\emptyset$.
By (\ref{e9.7}) and (\ref{e9.8}), for large enough $j_0$, we have that
for all $1\leq t\leq \left\lfloor\frac{\#|M_j|}{j_0}\right\rfloor-1$,
$${\rm diam}(\tilde{B}_{j_{t+1}})\leq\lambda_1^{j_0}\cdot{\rm diam}(\tilde{B}_{j_t})$$
and for all $3\leq t\leq \left\lfloor\frac{\#|M_j|}{j_0}\right\rfloor$,
$$d(\tilde{B}_{j_t},\tilde{z})\preccurlyeq{\rm diam}(\tilde{B}_{j_t})\ {\rm and}\
{\rm diam}(\tilde{B}_{j_t})<\lambda_1^{j_0}\cdot l_n,$$
where $0<\lambda_1<1$ is a universal constant.
Applying Lemma \ref{L1} to $\tilde{S}_{l_n}$ and $E_n$,
we have
$${\rm area}(E_n)\leq\lambda_2^n\cdot{\rm area}(\tilde{S}_{l_n}),$$
where $0<\lambda_2<1$ is a universal constant.
Then
$${\rm area}(p(E_n))\leq c_2\cdot\lambda_2^n\cdot{\rm area}(S_{l_n}),$$
where $c_2$ is a universal constant.
For all small enough $r>0$, there exists a positive integer $n(r)>0$ such that
$\mathbb{B}(1,r)\subset S_{l_{n(r)}}$ and $r\asymp l_{n(r)}$. Then
$${\rm area}(\mathbb{B}(1,r)\setminus K_{c\alpha_0}(F_{\alpha}))={\rm area}(p(E_{n(r)})\cap\mathbb{B}(1,r))\leq
c_3\cdot\lambda_2^{n(r)}\cdot{\rm area}(\mathbb{B}(1,r)),$$
where $c_3>0$ is a universal constant.
Since $r\asymp l_{n(r)}$, we have $n(r)\asymp-\log r$ and hence
$${\rm area}(\mathbb{B}(1,r)\setminus K_{c\alpha_0}(F_{\alpha}))\leq
c_3\cdot r^{-c_4\log\lambda_2}\cdot{\rm area}(\mathbb{B}(1,r)),$$
where $c_4>0$ is a universal constant.
By the arbitrariness of $\alpha_0$, the proof of (1) is completed.

\vspace{0.2cm}
\noindent{\bf The proof of (2):}
By Proposition \ref{pr2.1},
there exists a sequence $\{z_n\}_{n=1}^{+\infty}\subset\overline{U_0}$ with
a quasi-log-arithmetic sequence $\left\{l_n:=|z_n-1|\right\}_{n=1}^{\infty}$ such that
for any $n\geq 1$,
if $z\in S_{l_{n+1}}$ with $\Lambda_{\alpha_0}(z)>l_n$, then there exists a nonnegative integer $m$ such that
$$S_{l_1}\cup\{F_{\alpha}^{\comp j}(z)\}_{j=0}^m\subseteq\mathbb{D}_{1+\alpha_0/2}\ {\rm and}\
d_{\mathbb{C}\setminus\overline{\mathbb{D}}}(F_{\alpha}^{\comp m}(z),z_n)\preccurlyeq1.$$
By (1), we have
\begin{equation}
\label{e9.10}\lim_{n\to\infty}\frac{{\rm area}(S_{l_n}\setminus K_{\alpha_0/2}(F_{\alpha}))}{{\rm area}(S_{l_n})}=0.
\end{equation}
Given any $w\in\mathbb{S}^1$, for large enough $n$, we have that
for any small enough square $S$ centering at $w$, there exists a positive integer $m_n$ such that
\begin{equation*}
F_{\alpha}^{\comp m_n}(S)\subseteq S_{l_{n+1}}
\end{equation*}
and for $0\leq t\leq m_n$,
\begin{equation}
\label{e9.11}F_{\alpha}^{\comp t}(S)\subseteq\mathbb{D}_{1+\alpha_0/2}.
\end{equation}
Given such a square $S$, for any $z\in S\setminus K_{\alpha_0}(F_{\alpha})$,
we have
$$z':=F_{\alpha}^{\comp m_n}(z)\in S_{l_{n+1}}\setminus K_{\alpha_0}(F_{\alpha}).$$
Then
$z'\in S_{l_{n+1}}$ with $\Lambda_{\alpha_0}(z')>l_n$
and hence
there exists a nonnegative integer $m$ such that
\begin{equation}
\label{e9.12}\{F_{\alpha}^{\comp s}(z')\}_{s=0}^m\subseteq\mathbb{D}_{1+\alpha_0/2}\ {\rm and}\
d_{\mathbb{C}\setminus\overline{\mathbb{D}}}(F_{\alpha}^{\comp m}(z'),z_n)\preccurlyeq1.
\end{equation}
Then for all large enough $n$, $F_{\alpha}^{\comp m}(z')\in S_{l_{n-N}}$ with $N\asymp1$.
Thus there exists a positive real number $v_n$ with $v_n\asymp l_n$ such that
$\mathbb{B}(F_{\alpha}^{\comp m}(z'),v_n)\subseteq S_{l_{n-N-1}}$ and
$\mathbb{B}(F_{\alpha}^{\comp m}(z'),v_n)\cap\mathbb{D}=\emptyset$.
Since $\mathbb{B}(F_{\alpha}^{\comp m}(z'),v_n)\cap\mathbb{D}=\emptyset$,
there exists a univalent branch $h$ of $F_{\alpha}^{\comp-(m+m_n)}$ on $\mathbb{B}(F_{\alpha}^{\comp m}(z'),v_n)$
with $h\comp F_{\alpha}^{\comp m}(z')=z$.
Combining (\ref{e9.11}) and (\ref{e9.12}),
we have
$$\{F_{\alpha}^{\comp s}(z)\}_{s=0}^{m+m_n}\subseteq\mathbb{D}_{1+\alpha_0/2}.$$
It follows from Koebe distortion theorem that
there exists a real number $\hat{v}>1$ (not depending on $n$ and $z$) such that
\begin{equation}
\label{e9.13}h(\mathbb{B}(F_{\alpha}^{\comp m}(z'),v_n/\hat{v})\cap K_{\alpha_0/2}(F_{\alpha}))
\subseteq K_{\alpha_0}(F_{\alpha}).
\end{equation}
Then there exists a positive real number $v_n'$ such that
\begin{equation}
\label{e9.14}\mathbb{B}(z,v_n')\subseteq h(\mathbb{B}(F_{\alpha}^{\comp m}(z'),v_n/\hat{v}))\
{\rm and}\ v_n'\asymp_{c_5}{\rm diam}(h(\mathbb{B}(F_{\alpha}^{\comp m}(z'),v_n/\hat{v}))),
\end{equation}
where $c_5$ is a constant not depending on $n$ and $z$.
By $(\ref{e9.10})$, $(\ref{e9.13})$ and $(\ref{e9.14})$, we have
\begin{equation*}
\lim_{n\to\infty}\frac{{\rm area}\left(\mathbb{B}(z,v_n')
\setminus K_{\alpha_0}(F_{\alpha})\right)}{(v_n')^2}=0,\ {\rm uniform\ on\ z.}
\end{equation*}
Moreover, $v_n'$ is less than the side length of $S$, for $\mathbb{B}(z,v_n')\cap\mathbb{D}=\emptyset$.
At last, Lemma \ref{l2} gives
$$\lim_{{\rm diam}(S)\to0}\frac{{\rm area}(S\setminus K_{\alpha_0}(F_{\alpha}))}{{\rm area}(S)}=0,$$
which completes the proof of (2).

\section{Appendix A}
In this appendix, we prove the following five properties about $S^n$, $F^n$ and $E$ in the proof of Lemma \ref{L1}.

\noindent{\bf Properties:}
\begin{itemize}
\item[(1)] $\{S^n\}_{n=1}^{N_0}$ is decreasing,
\item[(2)] for all $1\leq n\leq N_0$, $E\subseteq S^n$,
\item[(3)] for all $1\leq n\leq N_0-1$, $F^n\cap E=\emptyset$ and $F^n\subseteq(S^n)^{\circ}$,
\item[(4)] for all $x\in\cup_{n=1}^{N_0-1}F^n$, at most two of $F^1,F^2,\cdots,F^{N_0-1}$ contain $x$,
\item[(5)] for all $1\leq n\leq N_0-1$,
$$\zeta\cdot{\rm area}(S^n)\leq{\rm area}(F^n)\leq {\rm area}(S^n),$$
where $\zeta:=\frac{1}{2\pi\cdot M^2\cdot(\frac{\sqrt{2}}{2}+\sqrt{2}(c+2)M)^2}$.
\end{itemize}
\begin{proof}
(1) For any $1\leq n\leq N_0-1$, we will prove $S^{n+1}\subseteq S^n$.
To prove it, we only need to prove that
for any admissible square $S_{i_1i_2\cdots i_k}$ of generation $n+1$,
there exists $1\leq\tilde{k}<k$ such that  $S_{i_1i_2\cdots i_{\tilde{k}}}$
is an admissible square of generation $n$.
In fact, since $S_{i_1i_2\cdots i_k}$ is an admissible square of generation $n+1$,
we have
\[\sqrt{2}\frac{l}{M^{k+1}}<R^{n+1}_{i_1i_2\cdots i_k}\leq\sqrt{2}\frac{l}{M^k}.\]
By (\ref{e3.2}),
$$R^n_{i_1i_2\cdots i_k}\geq M\cdot R^{n+1}_{i_1i_2\cdots i_k}>\sqrt{2}\frac{l}{M^k}.$$
Since $R^n_{i_1}\leq\frac{l}{M}$, we have $k\geq2$.
Since $S_{i_1i_2\cdots i_k}\subseteq S_{i_1i_2\cdots i_{k-1}}$, we have
$$R^n_{i_1i_2\cdots i_{k-1}}\geq R^n_{i_1i_2\cdots i_k}>\sqrt{2}\frac{l}{M^k}.$$
Then there exists $1\leq\tilde{k}\leq k-1$ be the smallest positive integer such that
\begin{equation}
\label{F4.2}R^n_{i_1i_2\cdots i_{\tilde{k}}}>\sqrt{2}\frac{l}{M^{\tilde{k}+1}}.
\end{equation}
If $\tilde{k}=1$, then
$$R_{i_1}^n>\sqrt{2}\frac{l}{M^2}.$$
Again, since
$$R_{i_1}^n\leq\frac{l}{M}<\sqrt{2}\frac{l}{M},$$
we have that $S_{i_1}$ is an admissible square of generation $n$.
If $\tilde{k}\geq2$, then
\begin{equation}
\label{F4.3}R^n_{i_1i_2\cdots i_j}\leq\sqrt{2}\frac{l}{M^{j+1}},\ 1\leq j\leq\tilde{k}-1.
\end{equation}
It follows from
$$R^n_{i_1i_2\cdots i_{\tilde{k}-1}}\leq\sqrt{2}\frac{l}{M^{\tilde{k}}}\ {\rm and}\
R^n_{i_1i_2\cdots i_{\tilde{k}}}\leq R^n_{i_1i_2\cdots i_{\tilde{k}-1}}$$
that
\begin{equation}
\label{F4.4}R^n_{i_1i_2\cdots i_{\tilde{k}}}\leq\sqrt{2}\frac{l}{M^{\tilde{k}}}.
\end{equation}
Combining \ref{F4.2}, \ref{F4.3} and \ref{F4.4}, $S_{i_1i_2\cdots i_{\tilde{k}}}$ is an admissible square of generation $n$.

(2) For any $1\leq n\leq N_0$, we will prove $E\subseteq S^n$.
In fact, for any $x\in E$, we let $k$ be a positive integer such that
$$\tilde{r}_n(x)>\sqrt{2}\frac{l}{M^k}.$$ Without loss of generality,
we assume $x\in S_{i_1i_2\cdots i_k}$. Thus
$$R^n_{i_1i_2\cdots i_k}>\sqrt{2}\frac{l}{M^k}.$$
Next, by discussion as the same as that of (1), we have that
there exists $1\leq\tilde{k}<k$ such that
$S_{i_1i_2\cdots i_{\tilde{k}}}$ is an admissible square of generation $n$.
Then
$$x\in S_{i_1i_2\cdots i_k}\subseteq S_{i_1i_2\cdots i_{\tilde{k}}}\subseteq S^n.$$
Thus
$$E\subseteq S^n.$$

(3) For all $1\leq n\leq N_0-1$ and any $x\in F^n$,
by definition of $F^n$, there exists an admissible square $S_{i_1i_2\cdots i_k}$ of generation $n$
such that one of the following two conditions holds:
\begin{itemize}
\item[(${\rm a_1}$)] there exists an admissible square $S_{i_1'i_2'\cdots i_{k'}'}$ ($\subseteq S_{i_1i_2\cdots i_k\frac{M^2+1}{2}}$)
of generation $n+1$ such that $x\in F_{i_1'i_2'\cdots i_{k'}'}$;
\item[(${\rm a_2}$)] $x\in S_{i_1i_2\cdots i_k\frac{M^2+1}{2}}$ and for any admissible square
$S_{i_1'i_2'\cdots i_{k'}'}$ ($\subseteq S_{i_1i_2\cdots i_k\frac{M^2+1}{2}}$)
of generation $n+1$, $x\not\in S_{i_1'i_2'\cdots i_{k'}'}$.
\end{itemize}
If (${\rm a_1}$) holds, then $x\not\in E$.
Assume (${\rm a_2}$) holds. If $x\in E$, then
we let $k_1$ $(>k+1)$ be a positive integer such that
$$\tilde{r}_{n+1}(x)>\sqrt{2}\frac{l}{M^{k_1}}.$$
It follows that there exists a square of form $S_{i_1i_2\cdots i_k\frac{M^2+1}{2}i_{k+2}i_{k+3}\cdots i_{k_1}}$
such that
$$x\in S_{i_1i_2\cdots i_k\frac{M^2+1}{2}i_{k+2}i_{k+3}\cdots i_{k_1}}$$
and
$$R^{n+1}_{i_1i_2\cdots i_k\frac{M^2+1}{2}i_{k+2}i_{k+3}\cdots i_{k_1}}>\sqrt{2}\frac{l}{M^{k_1}}.$$
Set $i_{k+1}:=\frac{M^2+1}{2}$ and then by discussion as the same as that of (1), we have that
there exists $1\leq\tilde{k}<k_1$ such that
$S_{i_1i_2\cdots i_{\tilde{k}}}$ is an admissible square of generation $n+1$.
Then
$$R^n_{i_1i_2\cdots i_{\tilde{k}}}\geq M\cdot
R^{n+1}_{i_1i_2\cdots i_{\tilde{k}}}>\sqrt{2}\frac{l}{M^{\tilde{k}}}.$$
Since $S_{i_1i_2\cdots i_k}$ is an admissible square of generation $n$,
we have $\tilde{k}\geq k+1$ and hence
$$x\in S_{i_1i_2\cdots i_{\tilde{k}}}\subseteq S_{i_1i_2\cdots i_k\frac{M^2+1}{2}}.$$
This contradicts with Condition (${\rm a_2}$).
Thus $x\not\in E$.
Then we prove that for all $1\leq n\leq N_0-1$, $F^n\cap E=\emptyset$.

Next, we prove that for all $1\leq n\leq N_0-1$, $F^n\subseteq(S^n)^{\circ}$.
Let $x\in F^n$. By definition, there exists an admissible square $S_{i_1i_2\cdots i_k}$
of generation $n$ such that one of the following two conditions holds:
\begin{itemize}
\item[(A)] $x\in S_{i_1i_2\cdots i_k\frac{M^2+1}{2}}$;
\item[(B)] $S_{i_1i_2\cdots i_k\frac{M^2+1}{2}}\cap E\not=\emptyset$ and
there exists an admissible square $S_{i_1'i_2'\cdots i_{k'}'}$ ($\subseteq S_{i_1i_2\cdots i_k\frac{M^2+1}{2}}$) such that
$x\in F_{i_1'i_2'\cdots i_{k'}'}$.
\end{itemize}
If (A) holds, then $x\in S_{i_1i_2\cdots i_k\frac{M^2+1}{2}}\subseteq S_{i_1i_2\cdots i_k}^{\circ}\subseteq(S^n)^{\circ}$.
If (B) holds, by (b') in the definition of $F_{i_1'i_2'\cdots i_{k'}'}$,
the distance between $F_{i_1'i_2'\cdots i_{k'}'}$ and $S_{i_1'i_2'\cdots i_{k'}'}$
is less than or equal to $\sqrt{2}(c+1)\frac{l}{M^{k'}}\leq\sqrt{2}(c+1)\frac{l}{M^{k+1}}$.
Then the distance between $F_{i_1'i_2'\cdots i_{k'}'}$ and $S_{i_1i_2\cdots i_k\frac{M^2+1}{2}}$
is less than or equal to $\sqrt{2}(c+1)\frac{l}{M^{k+1}}$.
Observe that side lengths of $F_{i_1'i_2'\cdots i_{k'}'}$ and $S_{i_1i_2\cdots i_k\frac{M^2+1}{2}}$
are less than or equal to $\frac{l}{M^{k+1}}$.
Then the distance between any point of $F_{i_1'i_2'\cdots i_{k'}'}$ and the center of $S_{i_1i_2\cdots i_k}$
is less than $(\frac{2\sqrt{2}}{M}+\frac{\sqrt{2}(c+1)}{M})\frac{l}{M^k}<\frac{l}{2M^k}$ (Thanks to (\ref{F1})).
Thus $F_{i_1'i_2'\cdots i_{k'}'}\subseteq S_{i_1i_2\cdots i_k}^{\circ}$.
This implies $x\in(S_{i_1i_2\cdots i_k})^{\circ}\subseteq(S^n)^{\circ}$.
Thus $F^n\subseteq(S^n)^{\circ}$.

(4) {\bf Claim $3$:}
If $S_{i_1i_2\cdots i_k}$ is an admissible square of generation $n$ with $1\leq n\leq N_0$,
then for any $1\leq\tilde{k}<k$,
$S_{i_1i_2\cdots i_{\tilde{k}}}$ is not an admissible square of generation $m$ with $n\leq m\leq N_0$.

In fact, if not, then
\begin{equation}
\label{F4.5}R^m_{i_1i_2\cdots i_{\tilde{k}}}>\sqrt{2}\frac{l}{M^{\tilde{k}+1}}.
\end{equation}
Since $S_{i_1i_2\cdots i_k}$ is an admissible square of generation $n$,
we have
$$R^n_{i_1i_2\cdots i_{\tilde{k}}}\leq\sqrt{2}\frac{l}{M^{\tilde{k}+1}}$$
and hence
\begin{equation}
\label{F4.6}R^m_{i_1i_2\cdots i_{\tilde{k}}}\leq\frac{1}{M^{m-n}}R^n_{i_1i_2\cdots i_{\tilde{k}}}
\leq\sqrt{2}\frac{l}{M^{\tilde{k}+m-n+1}}.
\end{equation}
(\ref{F4.5}) contradicts with (\ref{F4.6}).
Thus for any $1\leq\tilde{k}<k$, $S_{i_1i_2\cdots i_{\tilde{k}}}$ is not an admissible square of generation $m$
with $n\leq m\leq N_0$.

Let $x\in\cup_{n=1}^{N_0-1}F^n$ and let $j_0$ be the smallest positive integer in $\{n:x\in F^n\}$.
By definition of $F^{j_0}$, we have one of the following two conditions holds:
\begin{itemize}
\item[(A')] there exists an admissible square $S_{i_1i_2\cdots i_k}$ of generation $j_0$ such that $x\in S_{i_1i_2\cdots i_k\frac{M^2+1}{2}}$
and for any admissible square
$S_{i_1'i_2'\cdots i_{k'}'}$ ($\subseteq S_{i_1i_2\cdots i_k\frac{M^2+1}{2}}$) of generation $j_0+1$,
$x\not\in S_{i_1'i_2'\cdots i_{k'}'}$;
\item[(B')] there exists an admissible square $S_{i_1i_2\cdots i_k}$ of generation $j_0+1$
such that $x\in F_{i_1i_2\cdots i_k}$.
\end{itemize}
We suppose that (A') holds.
By Claim $3$, we have the following fact: if $S_{i_1'i_2'\cdots i_{k'}'}$ is an admissible square of generation $j_0+1$,
then one of the following two cases holds:
\begin{itemize}
\item $S_{i_1'i_2'\cdots i_{k'}'}\cap(S_{i_1i_2\cdots i_k\frac{M^2+1}{2}})^{\circ}=\emptyset$;
\item $S_{i_1'i_2'\cdots i_{k'}'}\subseteq S_{i_1i_2\cdots i_k\frac{M^2+1}{2}}$ and
$x\not\in S_{i_1'i_2'\cdots i_{k'}'}$.
\end{itemize}
If $S_{j_1'j_2'\cdots j_{t'}'}$ is an admissible square of generation $n$ with $j_0+1\leq n\leq N_0-1$,
similar to that in the proof of (1),
we have that there exists an admissible square $S_{i_1'i_2'\cdots i_{k'}'}$ of generation $j_0+1$
such that $S_{j_1'j_2'\cdots j_{t'}'}\subseteq S_{i_1'i_2'\cdots i_{k'}'}$.
Together with the above fact, this implies that
if $S_{i_1'i_2'\cdots i_{k'}'}$ is an admissible square of generation $n$ with $j_0+1\leq n\leq N_0-1$,
then one of the following two cases holds:
\begin{itemize}
\item $S_{i_1'i_2'\cdots i_{k'}'}\cap(S_{i_1i_2\cdots i_k\frac{M^2+1}{2}})^{\circ}=\emptyset$;
\item $S_{i_1'i_2'\cdots i_{k'}'}\subseteq S_{i_1i_2\cdots i_k\frac{M^2+1}{2}}$ and
$x\not\in S_{i_1'i_2'\cdots i_{k'}'}$.
\end{itemize}
Then for any admissible square $S_{i_1'i_2'\cdots i_{k'}'}$ of generation $n$ with $j_0+1\leq n\leq N_0-1$,
we have $x\not\in S_{i_1'i_2'\cdots i_{k'}'}\ {\rm or}\ x\in\partial S_{i_1'i_2'\cdots i_{k'}'}.$
Thus the distance between $x$ and the center of $S_{i_1'i_2'\cdots i_{k'}'}$ is greater than or equal to $\frac{l}{2M^{k'}}$.
This implies that $x\not\in S_{i_1'i_2'\cdots i_{k'}'\frac{M^2+1}{2}}$
and for any admissible square $S_{j_1j_2\cdots j_t}$ ($\subseteq S_{i_1'i_2'\cdots i_{k'}'\frac{M^2+1}{2}}$) of generation $n+1$,
$x\not\in F_{j_1j_2\cdots j_t}$.
Then by the definition of $F^n$,
we have that for any $j_0+1\leq n\leq N_0-1$,
$x\not\in F^n$.

Next, assume (B') holds.
We let $F_{i_1i_2\cdots i_k}=S_{j_1j_2\cdots j_kj_{k+1}}$.
We firstly consider the case: $S_{j_1j_2\cdots j_kj_{k+1}}\subseteq S_{i_1i_2\cdots i_k}$.
For all $j_0+1\leq n\leq N_0-1$, by Claim $3$, we have that
for any admissible $S_{i_1'i_2'\cdots i_{k'}'}$ of generation $n$,
we have that one of the following two cases occurs:
\begin{itemize}
\item[(C)] $S_{i_1'i_2'\cdots i_{k'}'}\subseteq S_{i_1i_2\cdots i_k}$, that is $k'\geq k$, $i_1'=i_1,\cdots,i_k'=i_k$;
\item[(D)] $S_{i_1'i_2'\cdots i_{k'}'}\cap(S_{i_1i_2\cdots i_k})^{\circ}=\emptyset.$
\end{itemize}
If (C) holds and $n\geq j_0+2$, then $i_1'=i_1,i_2'=i_2,\cdots,i_k'=i_k$ and $k'\geq k+1$.
Thus
$$(S_{j_1j_2\cdots j_kj_{k+1}})^{\circ}\cap S_{i_1'i_2'\cdots i_{k'}'}=\emptyset.$$
If (D) holds, it follows from $S_{j_1j_2\cdots j_kj_{k+1}}\subseteq S_{i_1i_2\cdots i_k}$ and
$S_{i_1'i_2'\cdots i_{k'}'}\cap(S_{i_1i_2\cdots i_k})^{\circ}=\emptyset$ that
$$S_{i_1'i_2'\cdots i_{k'}'}\cap(S_{j_1j_2\cdots j_kj_{k+1}})^{\circ}=\emptyset.$$
Thus for all $j_0+2\leq n\leq N_0-1$ and any admissible square $S_{i_1'i_2'\cdots i_{k'}'}$
of generation $n$, we have
$$(F_{i_1i_2\cdots i_k})^{\circ}\cap S_{i_1'i_2'\cdots i_{k'}'}=\emptyset.$$
Then by the definition of $S^n$, we have
$$F_{i_1i_2\cdots i_k}\cap(S^n)^{\circ}=\emptyset.$$
By (3), we have
$$F_{i_1i_2\cdots i_k}\cap F^n=\emptyset.$$
Thus for all $j_0+2\leq n\leq N_0-1$, $x\not\in F^n$.

Next, we consider the other case: $S_{j_1j_2\cdots j_kj_{k+1}}\not\subseteq S_{i_1i_2\cdots i_k}$,
that is $S_{j_1j_2\cdots j_kj_{k+1}}\cap(S_{i_1i_2\cdots i_k})^{\circ}=\emptyset$.
For any $j_0+1\leq n\leq N_0-1$ such that $S_{j_1j_2\cdots j_k}$ is not an admissible square of generation $n$
and any admissible square $S_{i_1'i_2'\cdots i_{k'}'}$ of generation $n$,
we will prove that
the distance between $S_{j_1j_2\cdots j_kj_{k+1}}$ and $S_{i_1'i_2'\cdots i_{k'}'\frac{M^2+1}{2}}$ is
greater than $\sqrt{2}(c+2)\frac{l}{M^{k'+1}}$.
In fact, if $S_{j_1j_2\cdots j_kj_{k+1}}\not\subseteq S_{i_1'i_2'\cdots i_{k'}'}$,
that is $S_{j_1j_2\cdots j_kj_{k+1}}\cap(S_{i_1'i_2'\cdots i_{k'}'})^{\circ}=\emptyset$,
then the distance between $S_{j_1j_2\cdots j_kj_{k+1}}$ and $S_{i_1'i_2'\cdots i_{k'}'\frac{M^2+1}{2}}$ is
greater than or equal to $\frac{l}{2M^{k'}}-\frac{l}{2M^{k'+1}}>\sqrt{2}(c+2)\frac{l}{M^{k'+1}}$ (Thanks to (\ref{F1})).
If
$S_{j_1j_2\cdots j_kj_{k+1}}\subseteq S_{i_1'i_2'\cdots i_{k'}'}$,
then by Claim $3$ and $S_{j_1j_2\cdots j_kj_{k+1}}\cap(S_{i_1i_2\cdots i_k})^{\circ}=\emptyset$,
we have
\begin{equation}
\label{F4.7}S_{i_1'i_2'\cdots i_{k'}'}\cap(S_{i_1i_2\cdots i_k})^{\circ}=\emptyset.
\end{equation}
Since $S_{j_1j_2\cdots j_kj_{k+1}}\subseteq S_{i_1'i_2'\cdots i_{k'}'}$ and
$S_{j_1j_2\cdots j_k}$ is not an admissible square of generation $n$,
there exists $1\leq\tilde{k}\leq k-1$ such that
\begin{equation}
\label{F4.8}S_{i_1'i_2'\cdots i_{k'}'}=S_{j_1j_2\cdots j_{\tilde{k}}},
\end{equation}
that is
$k'=\tilde{k}$, $i_1'=j_1, i_2'=j_2,\cdots, i_{k'}'=j_{\tilde{k}}$.
By the definition of $F_{i_1i_2\cdots i_k}$, we have that
the distance between $S_{j_1j_2\cdots j_kj_{k+1}}$ and $S_{i_1i_2\cdots i_k}$
is less than or equal to $\sqrt{2}(c+1)\frac{l}{M^k}$.
Then by (\ref{F4.7}) and (\ref{F4.8}),
we have that the distance between $S_{j_1j_2\cdots j_kj_{k+1}}$ and
$S_{j_1j_2\cdots j_{\tilde{k}}\frac{M^2+1}{2}}$ ($=S_{i_1'i_2'\cdots i_{k'}'\frac{M^2+1}{2}}$)
is more than
$\frac{l}{2M^{\tilde{k}}}-\frac{l}{2M^{\tilde{k}+1}}-\sqrt{2}(c+1)\frac{l}{M^k}-\sqrt{2}\frac{l}{M^{k+1}}>
\sqrt{2}(c+2)\frac{l}{M^{\tilde{k}+1}}=\sqrt{2}(c+2)\frac{l}{M^{k'+1}}$ (Thanks to (\ref{F1})).
Thus for any $j_0+1\leq n\leq N_0-1$ such that $S_{j_1j_2\cdots j_k}$ is not an admissible square of generation $n$
and any admissible square $S_{i_1'i_2'\cdots i_{k'}'}$ of generation $n$,
we have that the distance between $S_{j_1j_2\cdots j_kj_{k+1}}$ and
$S_{i_1'i_2'\cdots i_{k'}'\frac{M^2+1}{2}}$($=S_{j_1j_2\cdots j_{\tilde{k}}\frac{M^2+1}{2}}$) is
greater than $\sqrt{2}(c+2)\frac{l}{M^{k'+1}}$. Again since
for any admissible square $S_{j_1'j_2'\cdots j_{t'}'}$($\subseteq S_{j_1j_2\cdots j_{\tilde{k}}\frac{M^2+1}{2}}$)
of generation $n+1$, the distance between $F_{j_1'j_2'\cdots j_{t'}'}$ and
$S_{j_1j_2\cdots j_{\tilde{k}}\frac{M^2+1}{2}}$ ($=S_{i_1'i_2'\cdots i_{k'}'\frac{M^2+1}{2}}$)
is less than or equal to $\sqrt{2}(c+1)\frac{l}{M^{\tilde{k}+1}}\leq\sqrt{2}(c+2)\frac{l}{M^{k'+1}}-{\rm diam}(F_{j_1'j_2'\cdots j_{t'}'})$.
Thus $S_{j_1j_2\cdots j_kj_{k+1}}\cap F_{j_1'j_2'\cdots j_{t'}'}=\emptyset$.
This implies that for all $j_0+1\leq n\leq N_0-1$ such that $S_{j_1j_2\cdots j_k}$ is not an admissible square of generation $n$,
$S_{j_1j_2\cdots j_kj_{k+1}}\cap F^n=\emptyset$ and hence $x\not\in F^n$.
Thus $x$ is contained in at most two of $F^1,F^2,\cdots,F^{N_0-1}$.

(5) It follows from (3) that ${\rm area}(F^n)\leq {\rm area}(S^n)$ holds immediately.
Next, we prove
$$\zeta\cdot{\rm area}(S^n)\leq{\rm area}(F^n).$$
It is sufficient for completing the proof to prove that
for all $1\leq n\leq N_0-1$ and all admissible square $S_{i_1i_2\cdots i_k}$ of generation $n$,
$${\rm area}(F^n\cap S_{i_1i_2\cdots i_k})\geq\zeta\cdot{\rm area}(S_{i_1i_2\cdots i_k}).$$

For all $1\leq n\leq N_0-1$ and any admissible square $S_{i_1i_2\cdots i_k}$ of generation $n$,
if $S_{i_1'i_2'\cdots i_{k'}'}$ ($\subseteq S_{i_1i_2\cdots i_k\frac{M^2+1}{2}}$) is an admissible square of generation $n+1$,
by definition of $F_{i_1'i_2'\cdots i_{k'}'}$, we have that the distance between $F_{i_1'i_2'\cdots i_{k'}'}$ and
$S_{i_1i_2\cdots i_k\frac{M^2+1}{2}}$ is less than or equal to $\sqrt{2}(c+1)\frac{l}{M^{k+1}}$.
Since
$$\sqrt{2}(c+1)\frac{l}{M^{k+1}}+\sqrt{2}\frac{l}{2M^{k+1}}+\sqrt{2}\frac{l}{M^{k+2}}<\frac{l}{2M^k},$$
we have
$$F_{i_1'i_2'\cdots i_{k'}'}\subseteq(S_{i_1i_2\cdots i_k})^{\circ}.$$
Then by the arbitrariness of $S_{i_1i_2\cdots i_k}$, we have the fact (*) that $x\in F^n\cap S_{i_1i_2\cdots i_k}$ if and only if
one of the following two conditions holds:
\begin{itemize}
\item there exists an admissible square $S_{i_1'i_2'\cdots i_{k'}'}$ ($\subseteq S_{i_1i_2\cdots i_k\frac{M^2+1}{2}}$)
of generation $n+1$ such that $x\in F_{i_1'i_2'\cdots i_{k'}'}$;
\item $x\in S_{i_1i_2\cdots i_k\frac{M^2+1}{2}}$ and
for any admissible square $S_{i_1'i_2'\cdots i_{k'}'}$ ($\subseteq S_{i_1i_2\cdots i_k\frac{M^2+1}{2}}$) of generation $n+1$,
$x\not\in S_{i_1'i_2'\cdots i_{k'}'}$.
\end{itemize}
Let
\begin{align*}
K_{i_1i_2\cdots i_k}:=\cup\{&S_{i_1'i_2'\cdots i_{k'}'}:
{\rm S_{i_1'i_2'\cdots i_{k'}'}\ (\subseteq S_{i_1i_2\cdots i_k\frac{M^2+1}{2}})\ is}\\
&{\rm an\ admissible\ square\ of\ generation\ n+1}\}
\end{align*}
and
\begin{align*}
FK_{i_1i_2\cdots i_k}:=\cup\{&F_{i_1'i_2'\cdots i_{k'}'}:
{\rm S_{i_1'i_2'\cdots i_{k'}'}\ (\subseteq S_{i_1i_2\cdots i_k\frac{M^2+1}{2}})\ is}\\
&{\rm an\ admissible\ square\ of\ generation\ n+1}\}.
\end{align*}
If $S_{i_1'i_2'\cdots i_{k'}'}$ ($\subseteq S_{i_1i_2\cdots i_k\frac{M^2+1}{2}}$) is an admissible square
of generation $n+1$, by definition of $F_{i_1'i_2'\cdots i_{k'}'}$, we have that
the distance between $F_{i_1'i_2'\cdots i_{k'}'}$ and $S_{i_1'i_2'\cdots i_{k'}'}$ is less than or equal to
$\sqrt{2}(c+1)\frac{l}{M^{k'}}$.
This implies
$$S_{i_1'i_2'\cdots i_{k'}'}\subseteq EF_{i_1'i_2'\cdots i_{k'}'},$$
where $EF_{i_1'i_2'\cdots i_{k'}'}$ is a closed ball with radius $(\frac{\sqrt{2}}{2}+\sqrt{2}(c+2)M)\cdot\frac{l}{M^{k'+1}}$
having the same center as $F_{i_1'i_2'\cdots i_{k'}'}$.
Thus
$${\rm area}(FK_{i_1i_2\cdots i_k})\geq\frac{1}{\pi(\frac{\sqrt{2}}{2}+\sqrt{2}(c+2)M)^2}\cdot{\rm area}(K_{i_1i_2\cdots i_k}).$$
Then by the fact (*), we have
\begin{align*}
{\rm area}(F^n\cap S_{i_1i_2\cdots i_k})&
\geq\max\left\{{\rm area}(FK_{i_1i_2\cdots i_k}),\ {\rm area}(S_{i_1i_2\cdots i_k\frac{M^2+1}{2}}\setminus K_{i_1i_2\cdots i_k})\right\}\\
&\geq\frac{1}{\pi(\frac{\sqrt{2}}{2}+\sqrt{2}(c+2)M)^2}
\max\left\{{\rm area}(K_{i_1i_2\cdots i_k}),\ {\rm area}(S_{i_1i_2\cdots i_k\frac{M^2+1}{2}}\setminus K_{i_1i_2\cdots i_k})\right\}\\
&\geq\frac{1}{2\pi M^2(\frac{\sqrt{2}}{2}+\sqrt{2}(c+2)M)^2}{\rm area}(S_{i_1i_2\cdots i_k})\\
&=\zeta\cdot{\rm area}(S_{i_1i_2\cdots i_k}).
\end{align*}
\end{proof}

\section{Appendix B}
We recall that a compact set $\Lambda\subset\mathbb{C}$ is shallow if and only if for any $z\in\Lambda$ and $r<1$,
there is a ball $B$ disjoint from $\Lambda$ with ${\rm diam}(B)\asymp r$ and $d(z,B)=\mathcal{O}(r)$.
In this appendix we prove the following proposition:

\begin{proposition}
For any unbounded type irrational number $0<\alpha<1$, $J(F_{\alpha})$ is not shallow.
\end{proposition}
\begin{proof}
We prove this proposition by contradiction. We suppose that $J(F_{\alpha})$ is shallow.
Thus for any $z\in J(F_{\alpha})$ and $r<1$,
there is a ball $B$ disjoint from $J(F_{\alpha})$ with ${\rm diam}(B)\asymp r$ and $d(z,B)=\mathcal{O}(r)$.

By Proposition \ref{pr2.1},
there exists a sequence $\{z_n\}_{n=1}^{+\infty}\subset\overline{U_0}$ with
a quasi-log-arithmetic sequence $\left\{l_n:=|z_n-1|\right\}_{n=1}^{\infty}$ such that
for any $n>1$,
if $z\in S_{l_n}$ with $\Lambda_{\alpha_0}(z)=+\infty$, then there exists a nonnegative integer $m_{n,k}$ ($1\leq k\leq n-1$) such that
\begin{equation}
\label{a10.2}d_{\mathbb{C}\setminus\overline{\mathbb{D}}}(F_{\alpha}^{\comp m_{n,k}}(z),z_{n-k})\preccurlyeq1.
\end{equation}
Let
$$M_n:=\{m_{n,1},m_{n,2},\cdots,m_{n,n-1}\}=\{\omega_0^n,\omega_1^n,\omega_2^n,\cdots,\omega_{k_n}^n\}$$
with $\omega_0^n<\omega_1^n<\cdots<\omega_{k_n}^n$. By (\ref{a10.2}), for large enough $n$ we have
\begin{equation}
\label{a10.1}k_n\asymp n.
\end{equation}
By Lemma \ref{l6}, for any $1\leq t\leq n$ we have
$$\eta\cdot\rho_{\mathbb{C}\setminus\overline{\mathbb{D}}}(F_{\alpha}^{\comp \omega_{t-1}^n}(z))\leq
\rho_{\mathbb{C}\setminus\overline{\mathbb{D}}}(F_{\alpha}^{\comp \omega_t^n}(z))\cdot|(F_{\alpha}^{\comp(\omega_t^n-\omega_{t-1}^n)})'(z)|,$$
where $\eta>1$ is a universal constant.
Thus
$$\eta^{k_n}\cdot\rho_{\mathbb{C}\setminus\overline{\mathbb{D}}}(z)\leq
\eta^{k_n}\cdot\rho_{\mathbb{C}\setminus\overline{\mathbb{D}}}(F_{\alpha}^{\comp\omega_0^n}(z))\cdot|(F_{\alpha}^{\comp\omega_0^n})'(z)|\leq
\rho_{\mathbb{C}\setminus\overline{\mathbb{D}}}(F_{\alpha}^{\comp\omega_{k_n}^n}(z))\cdot|(F_{\alpha}^{\comp\omega_{k_n}^n})'(z)|$$
and
$$d_{\mathbb{C}\setminus\overline{\mathbb{D}}}(F_{\alpha}^{\comp\omega_{k_n}^n}(z),J(F_{\alpha}))\preccurlyeq1.$$
Let $L_n$ be a curve connecting $F_{\alpha}^{\comp\omega_{k_n}^n}(z)$ and $J(F_{\alpha})$
with $l_{\mathbb{C}\setminus\overline{\mathbb{D}}}(L_n)\preccurlyeq1$
and let $L_n^{-\omega_{k_n}^n}$ be the lift of $L_n$ under $F_{\alpha}^{\comp\omega_{k_n}^n}$
along $F_{\alpha}^{\comp\omega_{k_n}^n}(z), F_{\alpha}^{\comp(\omega_{k_n}^n-1)}(z), \cdots, F_{\alpha}(z), z.$
By Lemma \ref{l6},
for any $w\in L_n^{-\omega_{k_n}^n}$,
$$\eta^{k_n}\cdot\rho_{\mathbb{C}\setminus\overline{\mathbb{D}}}(w)\preccurlyeq
\rho_{\mathbb{C}\setminus\overline{\mathbb{D}}}(F_{\alpha}^{\comp\omega_{k_n}^n}(w))\cdot|(F_{\alpha}^{\comp\omega_{k_n}^n})'(w)|.$$
This implies that
\begin{align*}
d_{\mathbb{C}\setminus\overline{\mathbb{D}}}(z,J(F_{\alpha}))
&\leq
\int_{L_n^{-\omega_{k_n}^n}}\rho_{\mathbb{C}\setminus\overline{\mathbb{D}}}(w)|{\rm d}w|\\
&\preccurlyeq
\eta^{-k_n}\int_{L_n^{-\omega_{k_n}^n}}
\rho_{\mathbb{C}\setminus\overline{\mathbb{D}}}(F_{\alpha}^{\comp\omega_{k_n}^n}(w))\cdot|(F_{\alpha}^{\comp\omega_{k_n}^n})'(w)||{\rm d}w|\\
&=
\eta^{-k_n}\int_{L_n}
\rho_{\mathbb{C}\setminus\overline{\mathbb{D}}}(w)|{\rm d}w|\\
&\preccurlyeq\eta^{-k_n},
\end{align*}
that is
\begin{align}
\label{e10.1}d_{\mathbb{C}\setminus\overline{\mathbb{D}}}(z,J(F_{\alpha}))
\preccurlyeq\eta^{-k_n}.
\end{align}
Let $\alpha=[a_1,a_2,\cdots]$ be the continued fraction expansion of $\alpha$.
Since $\alpha$ is unbounded, we have that there exists a subsequence $\{a_{n_j}\}_{j=1}^{+\infty}$ such that
$\lim_{j\to+\infty}a_{n_j}=+\infty$.

Following [Lemma 5.5, \cite{Ya}], we have that
for large enough $j$, the restriction
$f_{\alpha}^{\comp q_{n_j-1}}|_{[x_{-q_{n_j-2}},1]_{\mathbb{S}^1}}$ is a small perturbation of a parabolic map with the fixed point index $2$ and hence
there exists a univalent map $\Psi$ from a neighborhood
$W_j\supset[x_{-q_{n_j-2}+q_{n_j}-sq_{n_j-1}},x_{sq_{n_j-1}}]_{\mathbb{S}^1}$
with $\lim_{j\to\infty}\sup_{z\in W_j}d(z,1)=0$
to a vertical strip
$\mathcal{D}_{n_j}=\{0<{\rm Re}(z)<a_{n_j}-2s\},$
mapping $W\cap\mathbb{S}^1$ into $\mathbb{R}$ and conjugating $f_{\alpha}^{\comp-q_{n_j-1}}$ to the unit translation $z\mapsto z+1$.
Moreover, the constant $s$ can be chosen independent of $f_{\alpha}$.

We define
$$E_j^{(1)}:=\left\{z=x+iy:\left|x-\frac{a_{n_j}-2s}{2}\right|<r_j,\ \left|y\right|<r_j\right\},$$
$$E_j^{(2)}:=\left\{z=x+iy:\left|x-\frac{a_{n_j}-2s}{2}\right|<2r_j,\ \left|y\right|<2r_j\right\},$$
and
$$R_j:=E_j^{(2)}\setminus\overline{E_j^{(1)}}$$
with $r_j=\frac{a_{n_j}-2s}{8}$.
It is easy to see that ${\rm Mod}(R_j)$ is a positive constant, denoted by $M$.
Set
$$z_j:=\frac{a_{n_j}-2s}{2}+r_ji\ {\rm and}\ w_j:=\Psi^{-1}(z_j).$$
We choose an Euclidean ball $\mathbb{B}(w_j,t_j)$ contained in $\mathbb{C}\setminus\overline{\mathbb{D}}$
such that the conformal modulus satisfies
$${\rm Mod}(\hat{\mathbb{C}}\setminus(\overline{\mathbb{D}}\cup\overline{\mathbb{B}(w_j,t_j)}))=\frac{2}{M}>\frac{1}{M}.$$
If $\mathbb{B}(w_j,t_j)\not\subseteq\Psi^{-1}(E_j^{(2)})$,
then $\hat{\mathbb{C}}\setminus(\overline{\mathbb{D}}\cup\overline{\mathbb{B}(w_j,t_j)})$
crosses $\Psi^{-1}(R_j)$ and hence
$${\rm Mod}(\hat{\mathbb{C}}\setminus(\overline{\mathbb{D}}\cup\overline{\mathbb{B}(w_j,t_j)}))\leq\frac{1}{{\rm Mod}(R_j)}=\frac{1}{M},$$
that is a contradiction. Thus $\mathbb{B}(w_j,t_j)\subseteq\Psi^{-1}(E_j^{(2)})$.

We firstly consider the case that there exists $a\in\mathbb{B}(w_j,t_j/4)\cap J(F_{\alpha})$.
In this case, the shallowness of $J(F_{\alpha})$ implies that
there is a ball $B\subseteq\mathbb{B}(w_j,t_j/2)$ disjoint from $J(F_{\alpha})$ with ${\rm diam}(B)\asymp t_j$.
Thus
\begin{equation}
\label{e10.2}{\rm Mod}(\hat{\mathbb{C}}\setminus(\overline{\mathbb{D}}\cup\overline{B}))\preccurlyeq1\ {\rm and}\
\frac{1}{{\rm diam}_{\mathbb{C}\setminus\overline{D}}(B)}\preccurlyeq1.
\end{equation}
Observe that
$B\subseteq\mathbb{C}\setminus K(F_{\alpha})$ or $B$ is contained in some component of $\cup_{n=0}^{\infty}F_{\alpha}^{\comp-n}(U_0)$,
where $K(F_{\alpha})$ consists of all points whose orbits are bounded.
By (\ref{a10.1}), (\ref{e10.1}) and (\ref{e10.2}), for large enough $j$, $B\subseteq\mathbb{C}\setminus K(F_{\alpha})$ doesn't hold.
Thus $B$ is contained in some component of $\cup_{n=0}^{\infty}F_{\alpha}^{\comp-n}(U_0)$.
This implies that for any $n\geq1$, $F_{\alpha}^{\comp n}(B)\cap B=\emptyset$.
Since $B\subseteq\mathbb{B}(w_j,t_j/2)$ and ${\rm diam}(B)\asymp t_j$,
by the Koebe distortion theorem we have that
$\Psi(B)$ contains a ball $\tilde{B}$ with ${\rm diam}(\tilde{B})\asymp{\rm diam}(\Psi(B))\asymp d(z_j,\Psi(B))$.
Since for any $n\geq1$, $F_{\alpha}^{\comp n}(B)\cap B=\emptyset$, we have ${\rm diam}(\tilde{B})\leq1$.
Then ${\rm diam}(\Psi(B))\asymp d(z_j,\Psi(B))\asymp1$ and hence for large enough $j$,
$$\lim_{j\to\infty}{\rm Mod}(\mathcal{D}_{n_j}\cap\mathbb{H}\setminus\Psi(B))=+\infty.$$
Pulling back by $\Psi$, we have
$$\lim_{j\to\infty}{\rm Mod}(\hat{\mathbb{C}}\setminus(\overline{\mathbb{D}}\cup\overline{B}))=+\infty,$$
which contradicts (\ref{e10.2}).
For the case that $\mathbb{B}(w_j,t_j/4)\cap J(F_{\alpha})=\emptyset$,
we can also get a contradiction in the same way.
Thus $J(F_{\alpha})$ is not shallow.
\end{proof}

\end{document}